\documentclass[twoside, reqno, 10pt]{amsart}

\usepackage[left=3.8cm, right=3.8cm, top=3cm, bottom=2.5cm]{geometry}

\usepackage{algorithm}
\usepackage{algpseudocode}

\usepackage[colorlinks=true, pdfstartview=FitV, linkcolor=blue,citecolor=blue, urlcolor=blue]{hyperref}

\usepackage{xcolor}

\definecolor{labelkey}{rgb}{0,0,1}
\definecolor{Red}{rgb}{0.7,0,0.1}
\definecolor{Green}{rgb}{0,0.7,0}

\usepackage{amsfonts, amssymb, amsmath, amsthm, mathrsfs, bbm, cjhebrew, gensymb, textcomp, mathtools, dsfont, calligra}

\usepackage[normalem]{ulem}

\usepackage{commath, setspace, subcaption, parcolumns, multirow, multicol, accents, comment, marginnote, verbatim, empheq, enumerate, stackrel, enumitem, float, physics, todonotes}
\usepackage[capitalize,nameinlink,noabbrev]{cleveref}

\usepackage{graphicx, graphics, epsfig, psfrag, tikz, tikz-cd,svg}

\numberwithin{equation}{section}

\usepackage{tikz}

\newtheorem{Thm}{Theorem}[subsection]
\newtheorem{Lem}[Thm]{Lemma}
\newtheorem{Prop}[Thm]{Proposition}
\newtheorem{Cor}[Thm]{Corollary}

\newtheorem{Def}{Definition}[section]
\newtheorem{Rmk}[Thm]{Remark}

\newtheorem*{Thm*}{Theorem}

\newcommand{\ZZ}{\mathbb{Z}}

\newcommand{\RR}{\mathbb{R}}

\newcommand{\PP}{\mathbb{P}}

\newcommand{\cT}{\mathcal{T}}

\DeclareMathOperator{\Id}{Id}

\newcommand{\no}[2]{\lVert#2\rVert_{#1}}
\newcommand{\Abs}[2]{\lvert#2\rvert_{#1}}

\newcommand{\goesto}{\rightarrow}
\newcommand{\smod}{\setminus}

\newcommand{\al}{\alpha}
\newcommand{\be}{\beta}
\newcommand{\de}{\delta}
\newcommand{\De}{\Delta}

\newcommand{\eps}{\epsilon}

\newcommand{\lam}{\lambda}

\newcommand{\tht}{\theta}

\newcommand{\Om}{\Omega}

\newcommand{\bdy}{\partial}
\newcommand{\lb}{\langle}
\newcommand{\rb}{\rangle}

\newcommand{\lp}{\left(}
\newcommand{\rp}{\right)}
\newcommand{\lpp}{(\!(}
\newcommand{\rpp}{)\!)}

\newcommand{\tv}{\tilde{v}}

\newcommand{\tg}{\tilde{g}}

\newcommand{\Cm}{\mathfrak{m}}

\newcommand{\Ck}{\mathfrak{k}}
\newcommand{\Ch}{\mathfrak{h}}
\newcommand{\Cd}{\mathfrak{d}}
\newcommand{\Cf}{\mathfrak{f}}

\newcommand{\gr}{\mathfrak{g}}

\newcommand{\tgr}{\tilde{\gr}}

\newcommand{\mut}{\tilde{\mu}}

\newcommand{\Cr}{\mathfrak{r}}
\newcommand{\Cp}{\mathfrak{p}}

\newcommand{\cM}{\mathcal{M}}

  \title[Determining Modes, State Reconstruction, and Intertwinement: Part 2]{Determining Modes, State Reconstruction, and Intertwinement: Existence of Self-Synchronizing Intertwinements}
 \author{Elizabeth Carlson, Aseel Farhat, Vincent R. Martinez$^\dagger$, Collin Victor}

\date{December 7, 2025\\
\indent $^\dagger$Corresponding author}
\begin{document}

\begin{abstract}
In the companion paper \cite{CarlsonFarhatMartinezVictor2025a}, a general synchronization framework was developed in the paradigmatic context of the 2D Navier-Stokes equations that allows one to precisely study the relation between the \textit{determining modes property} of the corresponding dynamical system and the ability of certain continuous data assimilation algorithms to reconstruct unobserved state variables from sufficiently many observed state variables in this system, i.e., the \textit{reconstruction property}. In this framework, the determining modes property and the reconstruction property can be viewed in a unified way as the ability of certain couplings of the Navier-Stokes equations to self-synchronize; due to the bi-directionality of coupling, the coupled system is referred to as an \textit{intertwinement}. A central achievement of this framework is to deduce a conceptual equivalence between the determining modes property and the reconstruction property of continuous data assimilation algorithms. In this paper, we prove that there are at least two non-trivial classes of intertwinements to which this framework applies. Moreover, these intertwinements encompass the continuous data assimilation algorithms studied in \cite{OlsonTiti2003} and \cite{AzouaniOlsonTiti2014} as special cases. Specifically, we show that there exist two types of intertwinements that are globally well-posed and we identify conditions under which these intertwinements self-synchronize. We emphasize that the intertwinements studied here can be induced by nonlinear perturbations of the underlying system, which are subsequently coupled bi-directionally to form the intertwinement. Thus, establishing global well-posedness and suitable global-in-time uniform bounds, which are crucial to proving the claimed synchronization phenomenon, requires careful consideration.
\end{abstract}

\maketitle

\setcounter{tocdepth}{2}

\vspace{1em}

{\noindent \small {\it {\bf Keywords: synchronization, intertwinement, determining modes, continuous data assimilation, Navier-Stokes equations}
  }  
  }

\tableofcontents
\section{Introduction}\label{sect:intro}

In \cite{FoiasProdi1967}, C. Foias and G. Prodi introduced the notion of \textit{determining modes} in the context of the two-dimensional (2D), externally driven Navier-Stokes Equations (NSE) for incompressible fluids and showed that the process associated to the NSE satisfies this property. This property effectively asserts the finite-dimensionality of the long-time dynamics of solutions to the 2D NSE. Since this seminal result, this property has been located at heart of several applications such as the problem of unique ergodicity of stochastically forced systems \cite{EMattinglySinai2001, KuksinShirikyan2012, Debussche2013}, downscaling in dissipative systems \cite{FoiasTiti1991, OlsonTiti2003, AzouaniOlsonTiti2014, KalantarovKostiankoZelik2023}, the existence of determining forms \cite{FoiasJollyKravchenkoTiti2012, FoiasJollyKravchenkoTiti2014, JollySadigovTiti2015, JollySadigovTiti2017, FoiasJollyLithioTiti2017, JollyMartinezSadigovTiti2018}, and most recently in continuous data assimilation for the purpose of state and parameter reconstruction \cite{OlsonTiti2003, BlomkerLawStuartZygalakis2013, AzouaniOlsonTiti2014, HaydenOlsonTiti2011, FarhatJollyTiti2015, AlbanezLopesTiti2016,  FoiasMondainiTiti2016, JollyMartinezTiti2017, AltafTitiGebraelKnioZhaoMcCabe2017, FarhatJohnstonJollyTiti2018, CarlsonHudsonLarios2020, CarlsonHudsonLariosMartinezNgWhitehead2021, PachevWhiteheadMcQuarrie2022, Martinez2022, BiswasHudson2023, Martinez2024, FarhatLariosMartinezWhitehead2024, AlbanezBenvenutti2024}. To begin, let us recall its definition. Note that throughout this manuscript, we will work over the doubly-periodic domain $\Om=[0,2\pi]^2$ and in a space of mean-free vector fields, although many of our considerations will apply over bounded domains with Dirichlet boundary conditions as well.

Given a divergence-free, possibly time-dependent, vector field, $f$, representing the external force, we let $S(t,s;f)$ denotes the associated two-parameter process of the Navier-Stokes system driven by $f$, i.e., $u(t;u_0,f(s))=S(t,s;f)u_0$, where $S(t,t)=\Id$ and  $S(t,r;f)S(r,s;f)=S(t,s;f)$, for all $s\leq r\leq t$, such that $u$ satisfies
    \begin{align}\label{eq:nse:ff}
    \frac{du}{dt}+\nu Au+B(u,u)=f(s+\cdot),\quad u(s;u_0,f(s))=u_0,
    \end{align}
where $A=-\PP\De$ denotes the Stokes operator, $\PP$ the Leray-Helmholtz projection onto divergence-free vector fields, $B(u,v)=\PP(u\cdotp\nabla v)$, and $\nu>0$ denotes the kinematic viscosity. Note that when $s=0$, we may simply write $u(t,0;u_0, f(0))$ as $u(t;u_0,f)$, and $S(t,0;f)$ as $S(t;f)$. We will adopt this convention throughout the paper.
Then one says that the Navier-Stokes process, $\{S(t;\cdotp)\}_{t\geq 0}$, possesses the \textit{finite-determining modes property} if given any $f_1,f_2$, there exists $N\geq0$ such that
     \begin{align}\label{cond:dm}
         |P_Nu(t;u_0^1,f_1) - P_Nu(t;u_0^2,f_2)| \goesto 0\quad\text{and}\quad|f_1(t) - f_2(t)| \goesto 0,\quad \text{as}\ t\goesto\infty,
     \end{align}
implies
     \begin{align}\label{eq:dm}
         |u(t;u_0^1,f_1) - u(t;u_0^2,f_2)| \goesto 0,\quad\text{as}\ t\goesto\infty,
     \end{align}
for all $u_0^1,u_0^2$, where $P_N$ denotes projection onto Fourier wavenumbers $|k|\leq N$.

In \cite{CarlsonFarhatMartinezVictor2025a}, a synchronization framework was developed in order to study the connection between the success of continuous data assimilation algorithms to reconstruct unobserved state variables from observed state variables and the finite determining modes property of the NSE. The key insight underlying the framework is to view the determining modes property as a property about the joint variable $(u_1,u_2)$. If one introduces the notion of a \textit{synchronous pair} to be a pair $(g_1,g_2)$ such that $|g_1(t)-g_2(t)|\goesto0$ as $t\goesto\infty$, the finite determining modes property can be characterized as follows: 
    \begin{align}\label{eq:fdm}\tag{DM}
    \begin{split}
&\text{For every synchronous pair $(f_1,f_2)$, there exists $N\geq0$ such that $(u_1,u_2)$}
\\
&\text{is a synchronous pair whenever $(P_Nu_1,P_Nu_2)$ is synchronous pair, where}
\\
&\text{$u_1=u(\cdotp,\cdotp;u_0^1,f_1)$ and $u_2=u(\cdotp,\cdotp;u_0^2,f_2)$ satisfy}\ (\ref{eq:nse:ff}).
    \end{split}
    \end{align}
From this perspective, the state reconstruction property of the continuous data assimilation algorithms of Olson and Titi \cite{OlsonTiti2003} and Azouani, Olson, and Titi \cite{AzouaniOlsonTiti2014} can be rephrased on the same conceptual level. To see this, let us present these algorithms in this particular light. Firstly, the algorithm of Olson and Titi \cite{OlsonTiti2003} can be viewed as the joint process $(u,v)$ such that
     \begin{align}\label{eq:sync}
        \begin{split}
        \frac{du}{dt}+\nu Au+B(u,u)&=f,\quad u(0)=u_0\\
        \frac{dv}{dt}+\nu Av+B(v,v)&=f+P_K\left(B(v,v)-B(u,u)\right),\quad v(0)=P_Nu_0+q_0.
        \end{split}
    \end{align}
On the other hand, the algorithm of Azouani, Olson, and Titi can viewed as the joint process defined by $(u,\tv)$ such that
    \begin{align}\label{eq:nudge}
        \begin{split}
        \frac{du}{dt}+\nu Au+B(u,u)&=f,\quad u(0)=u_0\\
         \frac{d\tv}{dt}+\nu A\tv+B(\tv,\tv)&=f-\mu P_K\tv+\mu P_Ku,\quad \tv(0)=\tv_0.
         \end{split}
    \end{align}
Observe that $v=P_Ku+Q_Nv$ in \eqref{eq:sync}. Subsequently, we refer to \eqref{eq:sync} as the \textit{direct-replacement algorithm}. On the other hand, since observations, $\{P_Ku(t)\}_{t\geq0}$, are inserted exogenously into the system in such a way as to drive the approximating state $v$ towards the observations $u$, we refer to \eqref{eq:nudge} as the \textit{nudging algorithm}. From the perspective of \eqref{eq:fdm}, the \textit{state reconstruction property} for \eqref{eq:sync} and \eqref{eq:nudge} can equivalently be rephrased as
    \begin{align}\label{eq:rp}
        \text{There exists $K\geq0$ (resp. $K, \mu$) such that $(u,v)$ (resp. $(u,\tv)$) is a synchronous pair.}\tag{RP}
    \end{align}
In each case of $(u_1,u_2)$, $(u,v)$, or $(u,\tv)$, the system satisfied by the joint process has the following structure:
    \begin{align}\label{eq:nse:intertwined}
        \begin{split}
        \frac{dv_1}{dt}+\nu Av_1+B(v_1,v_1)&=g_1+m_{11}{F}(v_1)+m_{12}{F}(v_2)
        \\
         \frac{dv_2}{dt}+\nu A{v}_2+B(v_2,v_2)&=g_2+m_{21}{F}(v_1)+m_{22}{F}(v_2),
        \end{split}
    \end{align}
for some choice of matrix $M\in\RR^{2\times2}$, where $M=(m_{ij})$, and operator $F$. Note that the function $F$ may be nonlinear, as in the case of \eqref{eq:sync}. Although the matrix $M$ corresponding to particular cases of \eqref{eq:sync} and \eqref{eq:nudge} only couples the NSE in one direction, in general, it may enforce coupling in \textit{both} directions. As a result, we refer to the triplet $(v,F,M)$ as an \textit{intertwinement} of the NSE and refer to $M$ as the \textit{intertwining matrix} and $F$ as the \textit{intertwining function}. 

One may now easily see that the finite determining modes property \eqref{eq:fdm} and the reconstruction property \eqref{eq:rp} can be realized as special cases of synchronization properties of the intertwined system \eqref{eq:nse:intertwined} for certain choices of $M$ and $F$. In fact, from this perspective \eqref{eq:fdm} admits a very natural generalization. Indeed, {given an intertwinement $(v,F,M)$ with globally well-posed initial value problem, let us denote
    \begin{align}\label{def:v:notation}
    v(\cdotp;v_0,g)=(v_1(\cdotp;v_0^1,g_1),v_2(\cdotp;v_0^2,g_2))
    \end{align}
the unique solution of \eqref{eq:nse:intertwined} corresponding to initial data $v|_{t=0}=v_0:=(v_0^1,v_0^2)$ and external force $g=(g_1,g_2)$.
} We then introduce the following definition.
\begin{Def}\label{def:fdss}
An intertwinement is \textbf{finite-dimensionally-driven self-synchronous} if for all synchronous pairs $(g_1,g_2)$, there exists $N\geq0$ such that $(v_1,v_2)$ is a synchronous pair whenever $(P_Nv_1,P_Nv_2)$ is a synchronous pair,
for all $(v_0^1,v_0^2)$. {We denote the smallest such $N$, for a given pair $g=(g_1,g_2)$ as above, by $N_{ss}(g)$ and refer to the function $N_{ss}:=N_{ss}(g)$ as the \textbf{dimension of self-synchronization}. When $N_{ss}\equiv0$, we simply say that $(v,F,M)$ is \textbf{self-synchronous}.
} 
\end{Def}
Therefore, the property \eqref{eq:fdm} is simply the assertion that the ``trivial" intertwinement, i.e., $(v,F,0_{2\times2})$ or $(v,0,M)$, is finite-dimensionally-driven self-synchronous, while the property \eqref{eq:rp} for the direct-replacement algorithm is the assertion that $(v,P_NB,M)$, for $M$ given by $m_{11}=m_{12}=0$, $m_{21}=m_{22}=1$, is self-synchronous. Similarly, \eqref{eq:rp} for the nudging algorithm is the assertion that $(v, P_N, M)$, for $M$ given by $m_{11}=m_{12}=0$, $m_{21}=m_{22}=\mu$ is self-synchronous. However, for other choices of $M$, we obtain a generalization of \eqref{eq:sync} and \eqref{eq:nudge}. We will refer to the intertwinements that formalize these generalizations as the \textit{direct-replacement intertwinement} and the \textit{nudging-intertwinement}. These will be introduced more precisely below (see \cref{def:intertwined:nudge} and \cref{def:intertwined:sync}).

In \cite{CarlsonFarhatMartinezVictor2025a}, \cref{def:fdss} was introduced and its elementary consequences were systematically developed, particularly with regards to its connection with continuous data assimilation. Two primary examples were furthermore proposed proposed which generalized the direct-replacement algorithm and nudging algorithm through intertwinements in exactly the way just described above. However, the rigorous analysis justifying the well-posedness and synchronizing properties of these intertwinements was not developed in \cite{CarlsonFarhatMartinezVictor2025a}. The main concern of the present study is to therefore  establish these well-posedness properties and synchronizing properties in the two particular cases of the nudging intertwinement and the direct-replacement intertwinement, thereby completing the presentation in \cite{CarlsonFarhatMartinezVictor2025a} and legitimizing the intertwinement framework.

In the next section, we formulate these goals more precisely and provide formal statements of the main results of the paper. We then conclude the introduction by indicating the analytical challenges encountered in establishing our main results.

\subsection{Main Goals and Results}

{We let $H$ denote the space of $L^2$ real-valued vector fields, which are $2\pi$-periodic in each direction, divergence-free, and mean-free over $\Om=[0,2\pi]$, in the sense of distribution}. We let $\PP$ denote the Leray projection. Observe that $\PP H=H$. {We let $V$ denote the subspace of $H$ endowed with the $V$ topology}. We make use of the following notation for the inner products and norms on $H$ and $V$, respectively:
    \begin{align}\label{def:H}
        \lp u, v\rp=\int_\Om u(x)\cdotp v(x)dx,\quad |u|^2=\lp u,u\rp,
    \end{align}
and
    \begin{align}\label{def:V}
        \lpp u,v\rpp=\sum_{j=1,2}\int_\Om \bdy_ju(x)\cdotp\bdy_jv(x)\ dx,\quad \lVert u\rVert^2=\lpp u, u\rpp.
    \end{align}
The dual spaces of $H, V$ will be denoted by $H^*, V^*$ respectively. That $V\subset H\subset H^*\subset V^*$ is a sequence of continuous embeddings follows from the Poincar\'e inequality \eqref{est:Poincare}. 

For each $1\leq p\leq\infty$, we will also make use of the Lebesgue spaces, $L^p$, which denote the space of $p$-integrable functions endowed with the following norm:
    \begin{align}\label{def:Lp}
        \Abs{p}{u}=\lp\int_\Om|u(x)|^p dx\rp^{1/p},
    \end{align}
with the usual modification when $p=\infty$. For convenience, we will view them as subspaces of completely integrable functions over $\Om$, which are mean-free and $2\pi$-periodic in each direction. It will be convenient to abuse notation and consider $L^p$ as a space of either scalar functions or vector fields. 

Recall that we denote the Stokes operator by $A=-\PP\De$. We then define, for each $n\geq0$, integer powers, $A^{n/2}$, of $A$ by
    \begin{align}\label{def:A}
        A^{n/2}u=\sum_{k\in\ZZ^2\smod\{(0,0)\}}|k|^n\hat{u}_kw_k,\quad w_k(x)=\exp(ik\cdotp x).
    \end{align}
Then the domain of $A^{n/2}$, $D(A^{n/2})$, is a dense subspace of $H$ endowed with the topology induced by
    \begin{align}\label{def:Hn}
        \no{n}{u}=|A^{n/2}u|=\lp\sum_{k\in\ZZ^2}|k|^{2n}|\hat{u}_k|^2\rp^{1/2}.
    \end{align}
With the notation above, note that we have
    \begin{align}\notag
        |u|=\Abs{0}{u},\qquad \rVert u\rVert=\no{0}{u}=\Abs{0}{A^{1/2}u}.
    \end{align}

To state our results on the solution theory of intertwinements, we define the sense of well-posedness that will be used in this paper. 

\begin{Def}\label{def:gwp}
{ 
Suppose $(v,F,M)$ is an intertwinement. Given $g\in L^\infty(0,\infty;H)^2$ and $v_0\in V\times V$, we say that the associated initial value problem of \eqref{eq:nse:intertwined} is \textbf{globally well-posed} if there exists a unique pair $v$ such that for all $T>0$, it holds that $v\in (C([0,T];V)\cap L^2(0,T;D(A)))^2$ and satisfies \eqref{eq:nse:intertwined} for $t\in (0,T)$, with $v|_{t=0}=v_0$.
}
\end{Def}

We point out that \cref{def:gwp} is consistent with the notion of well-posedness developed for \eqref{eq:nse:ff}, \eqref{eq:sync}, and \eqref{eq:nudge}; we refer the reader to \cite{ConstantinFoiasBook, TemamBook2001, OlsonTiti2003, AzouaniOlsonTiti2014}. We also refer the reader to \cite{CarlsonFarhatMartinezVictor2025a} where precise statements of the well-posedness of \eqref{eq:nse:ff}, \eqref{eq:sync}, \eqref{eq:nudge} are provided altogether. 

As previously mentioned, this paper will study two particular classes of intertwinements generalizing \eqref{eq:nudge} and \eqref{eq:sync}. The first of which we shall refer to the \textit{nudging intertwinement}, generalizes \eqref{eq:nudge}:

\begin{Def}\label{def:intertwined:nudge}
Given $g\in L^\infty(0,\infty;H)^2$, $K>0$, and matrix $M\in\RR^{2\times 2}$, consider:
    \begin{align}\label{eq:intertwined:nudge}
        \begin{split}
        \frac{dv_1}{dt} + \nu Av_1 + B(v_1,v_1) 
        &= 
        {g_1}+ m_{11}P_K v_1+ m_{12}P_K v_2,
	\\
        \frac{dv_2}{dt} + \nu Av_2 + B(v_2,v_2) 
        &= 
        {g_2} +m_{21}P_K v_1+ m_{22}P_Kv_2.
        \end{split}
    \end{align}
We refer to \eqref{eq:intertwined:nudge} as the \textbf{nudging intertwinement} whenever $M$ belongs to either of the classes
    \begin{align}\label{def:intertwined:nudge:matrix}
    \cM_{\mu}^{sym}:=\left\{\begin{pmatrix}-\mu_1&\mu_2\\\mu_2&-\mu_1\end{pmatrix}:\mu_1\geq\mu_2\geq0\right\},\quad \cM_{\mu}^{mut}:=\left\{ \begin{pmatrix}-\mu_1&\mu_1\\ \mu_2&-\mu_2\end{pmatrix}:\mu_1,\mu_2\geq0\right\}.
    \end{align}
When $M$ is given by the first class of matrices, we refer to the intertwinement as the \textbf{symmetric nudging intertwinement}, while the intertwinements corresponding to the second class of matrices will be referred to as the \textbf{mutual nudging intertwinement}.
\end{Def}

The second will be referred to as the \textit{direct-replacement intertwinement}, generalizing \eqref{eq:sync}:

\begin{Def}\label{def:intertwined:sync}
Given $g\in L^\infty(0,\infty;H)^2$, $K>0$, and matrix $M\in\RR^{2\times 2}$, consider:
	\begin{align}\label{eq:intertwined:sync}
		\begin{split}
		\bdy_tv_1+\nu Av_1+B(v_1,v_1)&={g_1}+m_{11}P_KB(v_1,v_1)+m_{12}P_KB(v_2,v_2)
		\\
		\bdy_tv_2+\nu Av_2+B(v_2,v_2)&={g_2}+m_{12}P_KB(v_1,v_1)+m_{22}P_KB(v_2,v_2).
		\end{split}
	\end{align}
We refer to \eqref{eq:intertwined:sync} as the \textbf{direct-replacement intertwinement} whenever $M$ belongs to the class
    \begin{align}\label{def:intertwined:sync:matrix}
        \begin{split}
        \cM_{\tht}^{{sym}}&:=\left\{\begin{pmatrix}\tht_1&-\tht_2\\-\tht_2&\tht_1\end{pmatrix}: \tht_1+\tht_2=1\right\},
        \\
        \cM_{\tht}^{{mut}}&:=\left\{\begin{pmatrix}\tht_1&-\tht_1\\ -\tht_2&\tht_2\end{pmatrix}: \tht_1+\tht_2=1,\ \tht_1,\tht_2\geq0\right\}.
        \end{split}
    \end{align}
When $M\in{\cM_{\tht}^{sym}}$, we refer to the \eqref{eq:intertwined:sync} as the \textbf{symmetric direct-replacement intertwinement}, while the intertwinement corresponding to the second class of matrices are referred to as the \textbf{mutual direct-replacement intertwinement}.
\end{Def}

Finally, we are ready state our main theorems regarding the global well-posedness of the nudging intertwinement and direct-replacement intertwinement.

\subsubsection{Global Well-Posedness}\label{sect:gwp}

\begin{Thm}\label{thm:intertwined:nudge:gwp}
The initial value problem corresponding to the nudging intertwinement is globally well-posed, for all $g\in L^\infty(0,\infty;H)^2$ and $v_0\in V\times V$.
\end{Thm}

\begin{Thm}\label{thm:dr:mut:gwp}
The initial value problem corresponding to the direct-replacement intertwinement is globally well-posed for all $M\in\cM_\tht^{mut}$ such that $\tht_1+\tht_2=1$, and also for $M\in\cM_\tht^{sym}$ such that either $\tht_1=1$ or $\tht_1=\tht_2=1/2$.
\end{Thm}

\begin{Thm}\label{thm:dr:sym:gwp}
For each $g\in L^\infty(0,\infty;H)^2$ and $v_0\in V\times V$, there exists $K_*>0$ and $\de_*>0$ such that the corresponding initial value problem associated to the symmetric direct-replacement intertwinement is globally well-posed for all $K\geq K_*$ and all $\tht_1,\tht_2$ such that either $\tht_1\geq1-\de_*$ or  $|\tht_1-\tht_2|\leq \de_*$.
\end{Thm}

In \cref{sect:apriori}, we develop the apriori bounds required to construct solutions to the various intertwinements with the properties asserted in \cref{def:gwp}. Since the procedure for constructing solutions is standard, we omit presentation of the details. Indeed, one proceeds by considering the corresponding Galerkin approximation of the intertwinement. All of the apriori estimates developed in \cref{sect:apriori} hold uniformly in the dimension of the approximation. One can moreover develop estimates on time-derivative of the approximating solutions uniform in the dimension. Construction of solutions then follows from carrying out a compactness argument. Our apriori estimates also give way to uniqueness of solutions and continuity in time. For further details, we refer the reader to \cite{OlsonTiti2003, AzouaniOlsonTiti2014, BiswasBrownMartinez2022}.

In order to prove the main synchronization properties of the nudging intertwinement and direct-replacement intertwinement, we will require suitable uniform-in-time bounds on solutions to the respective intertwinements. We formally state these bounds next.

\subsubsection{Uniform-in-time Estimates}\label{sect:apriori:intro}

Given $\mut>0$ and $\tg_1,\tg_2\in L^\infty(0,\infty;H)$, let $g, \tg, {g_{\mut}}$ denote
    \begin{align}\label{def:G:tG}
        g=(g_1,g_2),\quad \tg=(\tg_1,\tg_2),\quad{g_{\mut}}=g-\mut\tg.
    \end{align}
Then $g={g_{\mut}}+\mut\tg$. We then introduce
    \begin{align}\label{def:gri}
        \gr_j:=\frac{\sup_{t\geq0}|{g_j}(t)|}{\nu^2},\quad j=1,2,\qquad \gr^2:=\gr_1^2+\gr_2^2.
    \end{align}
and
    \begin{align}\label{def:tilde:g}
         \tgr_j:=\frac{\sup_{t\geq0}|{\tg_j}(t)|}{\nu^2},\quad j=1,2,\quad \tgr^2:=\tgr_1^2+\tgr_2^2,\quad \gr_{\mut}:=\frac{\sup_{t\geq0}|{g_{\mut}}(t)|}{\nu^2}.
    \end{align}

\begin{Thm}\label{thm:uniform:nudge}
Suppose $(v,F,M)$ is a nudging intertwinement. Then
    \begin{align}\notag
        \limsup_{t\goesto\infty}\|v(t)\|\leq\nu\begin{cases} \frac{\mu_1\vee\mu_2}{\mu_1\wedge\mu_2}\gr,&M\in\cM_{\mu}^{mut}
        \\
        \nu \min\{\gr^2,\gr_{\mut}^2+\mu\tgr^2\}^{1/2},&M\in\cM_{\mu}^{sym},
        \\
        \nu\min\left\{\frac{\nu\gr^2}{\mu_1-\mu_2}, \frac{\nu\gr_{\mut}^2}{\mu_1-\mu_2}+\tgr^2\right\}^{1/2},&M\in\cM_{\mu}^{sym},\ \mu_1>\mu_2,\ \mut=\mu_1-\mu_2.
        \end{cases}
    \end{align}
\end{Thm}

We refer the reader to \cref{sect:mutual:nudging:proof} and \cref{sect:symmetric:nudging:proof} for the proof of \cref{thm:uniform:nudge}.

\begin{Rmk}
The introduction of the different forces $g,\tg,g_{\mut}$ in \eqref{def:G:tG} addresses a technical, but subtle point that arises in the main application of the results established here. Specifically, in order for \cite[Theorem 3.1.9]{CarlsonFarhatMartinezVictor2025a} to be applicable to the nudging intertwinement, it is required to consider the symmetric intertwinement with $\mu_1=\mu$, $\mu_2=0$ and $g_1=f_1+\mu P_Ku_1$ and $g_2=f_2+\mu P_Ku_2$, where $u_1,u_2$ are  solutions to \eqref{eq:nse:ff} corresponding to $f_1,f_2$. If one then chooses $\mut=\mu$ and $\tg=(P_Ku_1,P_Ku_2)$, then $g_{\mut}=(f_1,f_2)$ and $\limsup_{t\goesto\infty}\|v(t)\|$ is bounded independently of $\mu$ according to the third case in \cref{thm:uniform:nudge}.
\end{Rmk}

\begin{Thm}\label{thm:uniform:dr:mut}
Suppose $(v,F,M)$ is a mutual direct-replacement intertwinement corresponding to a self-synchronous pair $(g_1,g_2)$. There exists a positive constant $C_0$ such that if $K\geq C_0 \ln(e+K)^{1/2}\gr$, then 
    \begin{align}\notag
        \limsup_{t\goesto\infty}\|v(t)\|\leq C_1\nu\gr,
    \end{align}
for some positive constant $C_1$.
\end{Thm}

We refer the reader to \cref{sect:apriori:dr} for the relevant details of the proof of \cref{thm:uniform:dr:mut}; specifically, see \cref{lem:dr:apriori:mut:refined}.

\begin{Thm}\label{thm:uniform:tht1}
Suppose $(v,F,M)$ is a symmetric direct-replacement intertwinement with $\tht_1=1$, corresponding to a self-synchronous pair $(g_1,g_2)$. There exists a positive constant $C_0$ such that if $K\geq C_0 \ln(e+K)^{1/2}\gr$, then
    \begin{align}\notag
        \limsup_{t\goesto\infty}\|v(t)\|\leq C_1\nu\gr,
    \end{align}
for some positive constant $C_1$.
\end{Thm}

We refer the reader to \cref{sect:apriori:tht1} for the relevant details of the proof of \cref{thm:uniform:tht1}; specifically, see \cref{lem:dr:apriori:tht1:refined}.

\begin{Thm}\label{thm:uniform:tht12}
Suppose $(v,F,M)$ is a symmetric direct-replacement intertwinement with $\tht_1=\tht_2=1/2$, corresponding to a self-synchronous pair $(g_1,g_2)$. Then
    \begin{align}\notag
        \limsup_{t\goesto\infty}\|v(t)\|\leq C_1\nu\gr,
    \end{align}
for some positive constant $C_1$.
\end{Thm}

We refer the reader to \cref{sect:apriori:tht12} for the relevant details of the proof of \cref{thm:uniform:tht12}; specifically, see \cref{lem:apriori:tht12:refined}.

\begin{Thm}\label{thm:uniform:tht1tht2}
Suppose $(v,F,M)$ is a symmetric direct-replacement intertwinement corresponding to a self-synchronous pair $(g_1,g_2)$.
There exist positive constants $C_0, C_1, C_2$ such that if
    \[  
    \|v_0\|\leq \nu\Cm_0,\quad K\geq C_1\ln(e+K)^{1/2}\gr,\quad   \min\left\{1-\tht_1,|\tht_1-\tht_2|\right\}\leq \frac{C_2}{\ln(e+K)(1+\Cm_0+\gr)^4},
    \]
for some $\Cm_0$, then
    \begin{align}\notag
        \limsup_{t\goesto\infty}\|v(t)\|\leq
            C_3\nu\gr,
    \end{align}
for some positive constant $C_3$.
\end{Thm}

We refer the reader to \cref{sect:apriori:dr:sym} for the relevant details of the proof of \cref{thm:uniform:tht1tht2}; specifically, see \cref{lem:apriori:tht2:refined} in \cref{sect:apriori:tht2} and \cref{lem:apriori:tht1tht2:refined} in \cref{sect:apriori:tht1tht2}.

\subsubsection{Synchronization Properties}\label{sect:sync:intro}

We will ultimately prove the following properties regarding the long-time behavior of the nudging and direct-replacement intertwinements.

\begin{Thm}\label{thm:intertwined:nudge:fdss}
Nudging intertwinements are finite-dimensionally-driven self-synchronous.
\end{Thm}

\begin{Thm}\label{thm:intertwined:nudge:ss}
For each $g\in L^\infty(0,\infty;H)^2$, there exists $K_*>0$ such that the nudging intertwinement is self-synchronous whenever
    \begin{align}\notag
        K\geq K_*\quad\text{and}\quad \mu_1+\mu_2\geq\frac{1}4K_*^2.
    \end{align}
\end{Thm}

\begin{Thm}\label{thm:dr:intertwinement}
Direct-replacement intertwinements are self-synchronous for all $M\in\cM_\tht^{mut}$ and also for $M\in\cM_\tht^{sym}$ such that $\tht_1\in\{1/2,1\}$.
\end{Thm}

\begin{Thm}\label{thm:dr:intertwinement:sym}
For each $g\in L^\infty(0,\infty;H)^2$ and $v_0\in V\times V$, there exists $K_1^*, K_2^*>0$ and $\de_1^*,\de_2^*>0$ such that corresponding symmetric direct-replacement intertwinement is self-synchronous whenever $K\geq K_1^*$ and $\tht_1\geq 1-\de_1^*$ or $K\geq K_2^*$ and $|\tht_1-\tht_2|\leq \de_2^*$. 
\end{Thm}

The proofs of these theorems will carried out in \cref{sect:sync}. Since they essentially rely on uniform-in-time estimates stated above in \cref{sect:apriori:intro}, it will suffice to prove the above theorems assuming the uniform-in-time bounds, then deduce \cref{thm:intertwined:nudge:fdss}, \cref{thm:intertwined:nudge:ss}, \cref{thm:dr:intertwinement}, \cref{thm:dr:intertwinement:sym} as corollaries.

We conclude the introduction with a discussion on the analytical challenges faced in proving the above theorems.

\subsection{Analytical Challenges}

As we mentioned at the end of the previous section, \cref{sect:sync:intro}, the synchronization properties rely on the availability of suitable uniform-in-time bounds. We may thus reduce the proofs of the synchronization properties of the various intertwinements to that of obtaining global well-posedness with suitable uniform-in-time bounds. In \cref{sect:sync}, we show how to derive the synchronization properties assuming these bounds. In effect, the analytical difficulties of this paper are concentrated in establishing global well-posedness and the requisite uniform-in-time bounds.

In both the case of the nudging intertwinement and the direct-replacement intertwinement, the key difficulty is in identifying a set of coordinates that allows one to capture the stabilizing properties of the perturbation of \eqref{eq:nse:ff} induced by the intertwining matrix and function pair $(M,F)$ properly. Indeed, in the special case where $(M,F)$ corresponds to the direct-replacement algorithm \eqref{eq:sync} or nudging algorithm \eqref{eq:nudge}, such stabilizing effects were exploited in \cite{OlsonTiti2003} and \cite{AzouaniOlsonTiti2014} to establish the property \eqref{eq:rp}.

For the nudging intertwinement, the mechanism is more or less clear in either the mutual or symmetric case. Indeed, when $M\in\cM_\mu^{mut}$, the stabilizing mechanism is essentially an extension of the mechanism present in \eqref{eq:nudge}.  When $M\in\cM_\mu^{sym}$, then the stabilizing properties can be extracted by the fact that $-M$ is non-negative definite. Details for the mutual nudging intertwinement are presented in \cref{sect:mutual:nudging:proof}, while details fo rthe symmetric nudging intertwinement are presented in \cref{sect:symmetric:nudging:proof}.

On the other hand, for the direct-replacement intertwinement, the mechanism is more difficult to extract. Even in the mutual case $M\in\cM_\tht^{mut}$, which is \textit{designed} for the synchronization variable $w=v_1-v_2$ to possess the same stabilizing mechanism as \eqref{eq:sync}, it is not clear how to establish apriori bounds that lead to global well-posedness due to bi-directionality of the coupling induced by $M$. Indeed, although the synchronization error variables $w=v_1-v_2$ is governed by the linear heat equation over the low-mode subspace $P_KH$, its high-mode evolution has a more complicated structure:
    \begin{align}\notag
		\bdy_t{w}+\nu A{w}={h}-Q_K(B(v_1,v_1)-B(v_2,v_2)).
	\end{align}
Due to the bi-directional coupling in $v_1,v_2$, it is difficult to obtain apriori bounds for the high-modes of $v_1,v_2$ in $V$ in the form presented by \eqref{eq:intertwined:sync}. To overcome this, we introduce the ``twisted variable" $v^\tht=\tht_2v_1+\tht_1v_2$ and rescale the synchronization variable, $w_\tht= \sqrt{\tht_1\tht_2}w$, to identify a fortuitous coupling structure for the evolution of the new joint variable $(v^\tht, w_\tht)$:
    \begin{align} 
        \bdy_t{v^\tht}+\nu A{v^\tht}+B({v^\tht},{v^\tht})&={g^\tht}-B({w_\tht},{w_\tht})\notag
        \\
        \bdy_t{w}+\nu A{w}+(\tht_1-\tht_2)Q_KB({w},{w})&={h}-Q_KDB({v^\tht}){w}.\notag
    \end{align}
We may then establish apriori estimates through this representation of the governing equation. Details are presented in \cref{sect:mutual:sync:proof}.

The difficulties in case of the symmetric direct-replacement intertwinement are particularly acute, however, since the stabilizing mechanism does not explicitly derive from the one naturally present in \eqref{eq:sync}. In this case, we must contend with the fact that the perturbation induced by the intertwining function, $F$, becomes \textit{genuinely nonlinear} when the coupling is induced by $M\in\cM_\tht^{sym}$ and may de-stabilize the coupled system without additional assumptions on the $M$. The bulk of the analytical difficulties of the paper is concentrated in treating symmetric direct-replacement intertwinements. 

To treat the symmetric direct-replacement intertwinement, we introduce new variables $z=v_1+v_2$ and $w=v_1-v_2$. Then \eqref{eq:intertwined:sync} can be rewritten as
 \begin{align}\notag
        \begin{split}
        \bdy_tz+\nu Az+\frac{1}2B(z,z)+\frac{1}2B(w,w)&=k+\frac{\tht_1-\tht_2}2\left(P_KB(z,z)+P_KB(w,w)\right)
        \\
        \bdy_tw+\nu Aw+\frac{1}2Q_KDB(z)w&=h
        \end{split}
    \end{align}
From here, we identify two special cases: one when $\tht_1=1$, $\tht_2=0$, which yields
    \begin{align}\notag
        \begin{split}
         \bdy_tz+\nu Az+\frac{1}2Q_KB(z,z)&=k-\frac{1}2Q_KB(w,w)
        \\
        \bdy_tw+\nu Aw+\frac{1}2Q_KDB(z)w&=h,
        \end{split}
    \end{align}
and the other when $\tht_1=\tht_2=1/2$, 
    \begin{align}\notag
        \begin{split}
        \bdy_tz+\nu Az+\frac{1}2B(z,z)&=k-\frac{1}2B(w,w)
        \\
        \bdy_tw+\nu Aw+\frac{1}2Q_KDB(z)w&=h.
        \end{split}
    \end{align}
The first case is simply the case of two \textit{de-coupled} copies of \eqref{eq:sync}. The second case is highly symmetric coupling of two copies of \eqref{eq:sync}. Although the case $\tht_1=1, \tht_2=0$ was already treated in \cite{OlsonTiti2003}, we are able to present a unified analysis that allows us to treat these two special cases in addition to perturbations of $M$ from these special case. Details are presented in \cref{sect:apriori:dr:sym}.

\section{Mathematical Preliminaries}\label{sect:preliminaries}

In this section, we will recall various properties, notations, and inequalities which will be useful for our analysis. First, we recall the Poincar\'e inequality:
    \begin{align}\label{est:Poincare}
        |u|\leq \lVert u\rVert,
    \end{align}
for all $u\in V$. 
    
We will also make use of the Ladyzhenskaya and Agmon, respectively, interpolation inequalities: there exist absolute constants $C_L, C_A>0$ such that
	\begin{align}\label{est:interpolation}
		\Abs{4}{u}^2\leq C_L\lVert u\rVert|u|,\qquad
		\Abs{\infty}{u}^2\leq C_A|Au||u|.
	\end{align}
Another useful interpolation inequality, is the following:
	\begin{align}\label{est:interpolation:CS}
		\lVert u\rVert^2\leq |Au||u|
	\end{align}
We will also make use of the Bernstein inequality: let $P_N$ denote projection onto Fourier wavenumbers, $|k|\leq N$, where $N>0$ is a real number. Denote the complementary projection by
    \begin{align}\label{def:QN}
        Q_N:=I-P_N
    \end{align}
Then for any integers $m\leq n$
	\begin{align}\label{est:Bernstein}
		\no{n}{P_Nu}\leq N^{n-m}\no{m}{P_Nu},\quad \no{m}{Q_Nu}\leq N^{m-n}\no{n}{Q_Nu}.
	\end{align}
Observe that we also have the following borderline Sobolev inequality
    \begin{align}\label{est:Sobolev}
        |P_Nu|_\infty\leq C_S(\ln N)^{1/2}\|P_Nu\|
    \end{align}

Recall that we defined the bilinear form $B$ appearing in \eqref{eq:nse:ff} by
    \begin{align}\label{def:B}
        B(u,v):=\PP((u\cdotp\nabla)v).
    \end{align}
When pairing \eqref{def:B} with a third function, it is convenient to view the result as a trilinear form. In particular, we write
    \begin{align}\label{def:B:trilinear}
        b(u,v,w):=\lp B(u,v),w\rp.    
    \end{align}
Then we have the well-known, skew-symmetric property of $b(u,v,w)$:
	\begin{align}\label{eq:B:skew}
		b(u,v,w)=-b(u,w,v),
	\end{align}
for $u,v,w\in V$, which immediately implies
    \[
        b(u,v,v)=0.    
    \]
We will also make use of the identity 
    \begin{align}\label{eq:B:enstrophy:miracle}
    b(v,v,Au)+b(u,v,Av)+b(v,u,Av)=0,
    \end{align}
which implies
    \begin{align}\label{eq:B:enstrophy}
        b(u,u,Au)=0
    \end{align}
Observe that $B:D(A)\times V\goesto H$ via 
    \begin{align}\label{est:B:ext:H}
        |B(u,v)|\leq C_A^{1/2}|Au|^{1/2}|u|^{1/2}\|v\|.
    \end{align}
whenever $u \in D(A)$ and $v\in V$. Moreover, $B$ is also continuous as a bilinear mapping $B:V\times V\goesto V'$ via
    \begin{align}\label{est:B:ext}
        |b(u,v,w)|\leq C_L\lVert u\rVert^{1/2}|u|^{1/2}\lVert v\rVert\|w\|^{1/2}|w|^{1/2},
    \end{align}
where $u,v\in V$ and $w\in V$, and $C_L$ is the constant appearing in \eqref{est:interpolation}.  The Frech\'et derivative of $B$ will be denoted by $DB$. It is elementary to show that $DB$ is given by
	\begin{align}\label{def:DB}
		DB(u)v=B(u,v)+B(v,u).
	\end{align}
By \eqref{est:B:ext:H}, it follows that $DB: D(A)\goesto L(D(A),H)$, $u\mapsto DB(u)$, while \eqref{est:B:ext} implies $DB: V\goesto L(V,V')$, where $L(X,Y)$ denotes the space of bounded linear operators mapping $X$ to $Y$.

We conclude this section {with} two elementary results {which} are crucial to establishing several of the main results of this article. The first result is {a} Gr\"onwall-type lemma that controls the long-time behavior of solutions, while the second collects some important bounds on solutions to the heat {equation; both} results are fundamental for establishing the synchronization properties of the various intertwinements.

\begin{Lem}\label{lem:gronwall:decay}
Let $z_1,z_2,z_3:(0,\infty)\goesto[0,\infty)$ be given such that $\lim_{t\goesto\infty}z_j(t)=0$, for $j=1,2$. Suppose $x:[0,\infty)\goesto[0,\infty)$ is a differentiable function such that
	\begin{align}\notag
		x'+\al x+\be y \leq z_1+z_2x+z_3y,
	\end{align}
holds for for all $t>0$, for some $\al, \be>0$, and some dominating function $y:[0,\infty)\goesto[0,\infty)$, i.e., $x\leq y$. Then $\lim_{t\goesto\infty}x(t)=0$.
\end{Lem}

\begin{Lem}\label{lem:heat:apriori}
Given $N>0$, $h\in L^\infty(0,\infty;H)$, and $p_0\in V$ such that $p_0=P_Np_0$. Let $p$ denote the unique solution of the initial value problem
    \begin{align}\notag
        \bdy_t p+\nu Ap=P_Nh,\quad p(0)=p_0.
    \end{align}
Then $p=P_Np$ and 
    \begin{align}\notag
        \|p(t)\|^2\leq e^{-\nu t}\|p(t_0)\|^2+\nu^2\left(\frac{\sup_{t\geq 0}|P_Nh(t)|}{\nu^2}\right)^2,
    \end{align}
for all $t\geq0$ and $N>0$. Moreover
    \begin{align}\label{est:t2:p:L2}
        \sup_{t\geq t_0}\|p(t)\|^2\leq 2\nu^2\left(\frac{\sup_{t\geq t_0'}|P_Nh(t)|}{\nu^2}\right)^2
    \end{align}
provided that
    \begin{align}\label{def:t2:t1:p:L2}
         t_0\geq t_0'+\frac{2}{\nu}\ln\left(\frac{\nu\|p(t_0')\|}{\sup_{t\geq t_0'}|P_Nh(t)|}\right),\quad t_0'\geq \frac{2}{\nu}\ln\left(\frac{\nu\|p_0\|}{\sup_{t\geq 0}|P_Nh(t)|}\right).
    \end{align}
In particular, if $\lim_{t\goesto\infty}|{P_Nh}(t)|=0$, then $\lim_{t\goesto\infty}\|{p}(t)\|=0$.
\end{Lem}

In fact, by \eqref{est:Bernstein}, we immediately obtain the following corollary.

\begin{Cor}\label{cor:apriori:heat}
Under the assumptions of \cref{lem:heat:apriori}, if $\lim_{t\goesto\infty}|P_Nh(t)|=0$, then $\lim_{t\goesto0}\sup_{0\leq \ell \leq m}\|p(t)\|_m=0$, for all $m\geq0$.
\end{Cor}

\begin{proof}[Proof of \cref{lem:heat:apriori}]
The claim that $p=P_Np$ follows simply by applying $P_N$ to the heat equation, observing that $P_NA=AP_N$, then applying uniqueness of solutions.

Next, recall that the energy balance for $p$ is given by
    \begin{align}\notag
        \frac{1}2\frac{d}{dt}\|{p}\|^2+\nu|A{p}|^2&=\lp {h},{Ap}\rp.
    \end{align}
An application of the Cauchy-Schwarz inequality and Gronwall's inequality then yields
    \begin{align}\label{est:p:L2}
        \|p(t)\|^2\leq e^{-\nu (t-t')}\|p(t')\|^2+\nu^2\left(\frac{\sup_{t\geq t'}|P_Nh(t)|}{\nu^2}\right)^2,
    \end{align}
for all $t\geq t'\geq0$. In particular 
    \begin{align}\label{est:p:L2:abs:ball}
        \sup_{t\geq t_0'}\|p(t)\|^2\leq 2\nu^2\left(\frac{\sup_{t\geq 0}|P_Nh(t)|}{\nu^2}\right)^2,
    \end{align}
for $t_0'$ given by \eqref{def:t2:t1:p:L2}. Furthermore, \eqref{est:p:L2:abs:ball} and \eqref{est:p:L2} imply
    \begin{align}\notag
        \sup_{t\geq t_0}\|p(t)\|^2\leq 2\nu^2\left(\frac{\sup_{t\geq t_0'}|P_Nh(t)|}{\nu^2}\right)^2,
    \end{align}
for $t_0$ given by \eqref{def:t2:t1:p:L2}.
\end{proof}

\section{Synchronization}\label{sect:sync}

As mentioned in \cref{sect:sync:intro}, we will deduce the main synchronization results as corollaries of those which hold under the availability of suitable uniform-in-time bounds. These bounds are stated in \cref{sect:apriori:intro}. It therefore suffices to establish the conditional results, which we state and prove as \cref{prop:intertwined:nudge:mutual}, \cref{prop:intertwined:nudge:ss}, and \cref{prop:intertwined:dr}. Nevertheless, for the sake of completeness, after the proofs of each of these propositions, we prove \cref{thm:intertwined:nudge:fdss}, \cref{thm:intertwined:nudge:ss}, \cref{thm:dr:intertwinement}, and \cref{thm:dr:intertwinement:sym}.

For the purposes of this section, it will thus be convenient to introduce the following condition: Given $g_1,g_2\in L^\infty(0,\infty;H)$ and a global solution $v=(v_1,v_2)$ of either the nudging intertwinement \eqref{eq:intertwined:nudge} or direct-replacement intertwinement \eqref{eq:intertwined:sync}, suppose that
    \begin{align}\label{cond:intertwined:uniform}
        \min\left\{\sup_{t\geq t_0}\|v_1(t)\|,\sup_{t\geq t_0}\|v_2(t)\|\right\}\leq \Cm\nu,\tag{U}
    \end{align}
holds for some positive constants $\Cm$ and $t_0$.

To prove the synchronization properties claimed in \cref{sect:sync:intro}, we will make use of the following results, proved in the companion work \cite{CarlsonFarhatMartinezVictor2025a}, which identify sufficient conditions that ensure self-synchronization.

\begin{Lem}\label{thm:fdss}
Let $(v,F,M)$ be an intertwinement. Suppose that there exists $N_*>0$ and $t_*$ sufficiently large such that 
    \begin{align}\label{cond:coercivity}
        &\left(B(v_1,v_1)-B(v_2,v_2)-\left((m_{11}-m_{21})F(v_1)+(m_{12}-m_{22})F(v_2)\right),v_1-v_2\right)\notag
               \\
        &\geq \left[-\eps\nu+C_2(\|P_Nv_1-P_Nv_2\|_{m_2})\right]\|v_1-v_2\|^2-C_1(\|P_Nv_1-P_Nv_2\|_{m_1})|v_1-v_2|^2\notag
        \\
        &\quad-C_0(\|P_Nv_1-P_Nv_2\|_{m_0}),
    \end{align}
holds for all $t\geq t_*$, for all $N\geq N_*$, for some $\eps\in(0,1)$, some integers $m_0,m_1,m_2$, and some non-negative, non-decreasing functions $C_0(x),C_1(x), C_2(x)$, defined for $x\geq0$, which are continuous at $0$. Then  {$(v,F,M)$} is finite-dimensionally-driven self-synchronous.
\end{Lem}

\begin{Lem}\label{thm:ss}
Let $(v,F,M)$ be an intertwinement. Suppose that there exists $N_*>0$ and $t_*$ sufficiently large such that 
    \begin{align}\label{cond:coercivity:ss}
        &\left(B(v_1,v_1)-B(v_2,v_2)-\left((m_{11}-m_{21})F(v_1)+(m_{12}-m_{22})F(v_2)\right),v_1-v_2\right)\notag
               \\
        &\geq -\eps\nu\|v_1-v_2\|^2-H(t),
    \end{align}
holds for all $t\geq t_*$, for all $N\geq N_*$, for some $\eps\in(0,1)$, and some non-negative function $H$ such that $H(t)\goesto0$ as $|g_1(t)-g_2(t)|\goesto\infty$. Then  {$(v,F,M)$} is self-synchronous.
\end{Lem}

\subsubsection{Nudging Intertwinement}\label{sect:nudge:sync}
A helpful observation will help facilitate the proof of our synchronization results for the nudging intertwinement is that the synchronization error defined by the variable $w:=v_1-v_2$ is governed by the \textit{same equation} regardless of whether the intertwining matrix comes from the symmetric class or the mutual class. Indeed, if $v$ satisfies \eqref{eq:intertwined:nudge} with $M\in\cM_\mu^{sym}\cup\cM_\mu^{mut}$, then
    \begin{align}\notag
        \bdy_tw+\nu Aw=h-\left(B(v_1,v_1)-B(v_2,v_2)+\mu P_Kw\right),
    \end{align}
where $h:=g_1-g_2$ and $\mu:=\mu_1+\mu_2$. 
Subsequently, it will be expedient to have the following lemma on hand.

\begin{Lem}\label{lem:error:nudge:trilinear}
Given $v_1,v_2\in V$, let $w=v_1-v_2$, $p=P_Kw$, and $q=Q_Kw$. For all $K>0$, we have
    \begin{align}\label{est:wv2w:1}
        |b(w,v_2,w)|&\leq C_L(\|v_1\|\wedge\|v_2\|)\left(K+4C_L\frac{\|v_1\|\wedge\|v_2\|}{\nu}\right)|p|^2\notag
        \\
        &\quad+\left(\frac{\nu}4+C_L\frac{\|v_1\|\wedge\|v_2\|}{ K}\right)\|w\|^2,
    \end{align}
and 
    \begin{align}\label{est:wv2w:2}
        |b(w,v_2,w)|
         &\leq  C_L^2\frac{\left(\|v_1\|\wedge\|v_2\|\right)^2}{\nu}|p|^2+\left(\frac{\nu}4+C_L^2\frac{\left(\|v_1\|\wedge\|v_2\|\right)^2}{K^2}\right)\|w\|^2,
    \end{align}
\end{Lem}

\begin{proof}
To prove \eqref{est:wv2w:1}, first observe that
    \begin{align}
        &b(w,v_2,w)=b(p,v_2,p)+b(p,v_2,p)+b(p,v_2,p)+b(p,v_2,p)\notag.
    \end{align}
By repeated application of H\"older's inequality, we obtain
    \begin{align}
        |b(w,v_2,w)|
            &\leq   C_L \|p\| |p| \|v_2\|
         + C_L|q|^{1/2}\|q\|^{1/2} \|v_2\||p|^{1/2} \|p\|^{1/2}\notag
         \\
         &\quad+C_L|p|^{1/2} \|p\|^{1/2} \|v_2\| |q|^{\frac12}\|q\|^{1/2}
         + C_L\|q\||q|\|v_2\|.\notag 
    \end{align}
For $t\geq t_0$, we then apply \eqref{est:Bernstein}, Young's inequality, and \eqref{est:intertwined:nudge:H1:uniform} for each term on the right-hand side and obtain
    \begin{align}\notag
        \begin{split}
        C_L \|p\| |p| \|v_2\|&\leq C_LK\|v_2\||p|^2
        \\
        C_L|q|^{1/2}\|q\|^{1/2} \|v_2\||p|^{1/2} \|p\|^{1/2}&\leq C_L\|v_2\|\|w\||p|
        \\
        &\leq \frac{\nu}8\|w\|^2+2C_L^2\frac{\|v_2\|^2}{\nu}|p|^2
        \\
        C_L|p|^{1/2} \|p\|^{1/2} \|v_2\| |q|^{1/2}\|q\|^{1/2}&\leq \frac{\nu}8\|w\|^2+2C_L^2\frac{\|v_2\|^2}{\nu}|p|^2
        \\
        C_L\|p\||p|\|v_2\|&\leq \frac{C_L\|v_2\|}{K}\|w\|^2.
        \end{split}
    \end{align}
Since $b(w,v_2,w)=b(w,v_1,w)$, we may thus replace $\|v_2\|$ with $\|v_1\|\wedge\|v_2\|$ above. This ultimately yields \eqref{est:wv2w:1}, as claimed.

On the other hand, we may alternatively estimate as follows:
    \begin{align}\notag
        |b(w,v_2,w)|&\leq C_L(\|v_1\|\wedge\|v_2\|)\|w\||w|\leq \frac{C_L^2\left(\|v_1\|\wedge\|v_2\|\right)^2}{\nu}|w|^2+\frac{\nu}4|w\|^2,
    \end{align}
for all $t\geq t_0$. It follows from orthogonality, \eqref{est:Bernstein}, \eqref{cond:intertwined:uniform}, and Young's inequality that
    \begin{align}
         |b(w,v_2,w)| &\leq C_L^2\frac{\left(\|v_1\|\wedge\|v_2\|\right)^2}{\nu}|p|^2+C_L^2\frac{\left(\|v_1\|\wedge\|v_2\|\right)^2}{\nu}|p|^2+\frac{\nu}4|w\|^2\notag
            \\
         &\leq  C_L^2\frac{\left(\|v_1\|\wedge\|v_2\|\right)^2}{\nu}|p|^2+\left(\frac{\nu}4+C_L^2\frac{\left(\|v_1\|\wedge\|v_2\|\right)^2}{K^2}\right)\|w\|^2,\notag
    \end{align}
which is precisely \eqref{est:wv2w:2}.
\end{proof}

\begin{Prop}\label{prop:intertwined:nudge:mutual}
Suppose that condition \eqref{cond:intertwined:uniform} holds. Then the corresponding nudging intertwinement is {finite-dimensionally-driven} self synchronous whenever $ K\geq 2C_L\Cm$.
\end{Prop}

\begin{proof}
Let $w = v_1-v_2$, $p=P_Kw$, $q=Q_Kw$, $h = {g_1}-{g_2}$, and $\mu = \mu_1 + \mu_2$. We seek to apply \cref{thm:fdss}.

By \cref{lem:error:nudge:trilinear}, we see that \eqref{est:wv2w:1} holds. Upon further applying \eqref{cond:intertwined:uniform}, we obtain
    \begin{align}
        |b(w,v_2,w)|\leq \nu\left(\frac{1}4+\frac{C_L\Cm}{K}\right)\|w\|^2+ \nu C_L\Cm\left(K+4C_L\Cm\right)|p|^2\notag
    \end{align}
Thus
    \begin{align}
        b(w,v_2,w)+\mu|p|^2&\geq -\nu\left(\frac{1}4+\frac{C_L\Cm}{K}\right)\|w\|^2+\left[\mu-\nu C_L\Cm\left(K+4C_L\Cm\right)\right]|p|^2\notag
        \\
        &\geq -\nu\left(\frac{1}4+\frac{C_L\Cm}{K}\right)\|w\|^2-\nu C_L\Cm\left(K+4C_L\Cm\right)|p|^2\notag
    \end{align}
Finally, we see that \eqref{cond:coercivity} is satisfied for $N_*=2C_L\Cm$, $t_*=t_0$, and $\eps=1/2$. We therefore conclude the proof with an application of \cref{thm:fdss}.
\end{proof} 

\begin{Prop}\label{prop:intertwined:nudge:ss}
Suppose that condition \eqref{cond:intertwined:uniform} holds. Then the corresponding nudging intertwinement is self-synchronous whenever $K\geq 2C_L\Cm$ and $\mu_1+\mu_2\geq C_L^2\Cm^2 \nu$.
\end{Prop}

\begin{proof}
Once again, let $w = v_1-v_2$, $p=P_Kw$, $q=Q_Kw$, $h = {g_1}-{g_2}$, and $\mu = \mu_1 + \mu_2$. We seek to apply \cref{thm:ss}. By \cref{lem:error:nudge:trilinear}, we have that \eqref{est:wv2w:2} holds. Applying \eqref{cond:intertwined:uniform}, we have
    \begin{align}
     -b(w,v_2,w)+\mu|p|^2\geq \left(\mu-C_L^2\Cm^2\nu\right)|p|^2-\nu\left(\frac{1}4+\frac{C_L^2\Cm^2}{K^2}\right)\|w\|^2\notag,
    \end{align}
holds for $t\geq t_0$. It then follows from our assumptions on $\mu, K$ that  \eqref{cond:coercivity:ss} holds with $N_*=2C_L\Cm$, $t_*=t_0$, $\eps=1/2$, and $H\equiv0$. We thus conclude the proof with an application of \cref{thm:ss}.
\end{proof}

Finally, we prove \cref{thm:intertwined:nudge:fdss}, \cref{thm:intertwined:nudge:ss}.

\begin{proof}[Proofs of \cref{thm:intertwined:nudge:fdss}, \cref{thm:intertwined:nudge:ss}]

We apply \cref{thm:uniform:nudge} to verify \eqref{cond:intertwined:uniform}; the claim then follows from \cref{prop:intertwined:nudge:mutual} and \cref{prop:intertwined:nudge:ss}, respectively, with the appropriate choice of $K_*$.

\end{proof}

\subsubsection{Direct-Replacement Intertwinement}
As in the case of the nudging intertwinement, when $v$ satisfies \eqref{eq:intertwined:sync}, the corresponding synchronization error $w=v_1-v_2$ is governed by the same equation, independent on whether the intertwining matrix comes from the mutual class or symmetric class. Indeed, from \eqref{eq:intertwined:sync} and using the fact that $\tht_1+\tht_2=1$, one immediately has
    \begin{align}
        \bdy_t{w}+\nu A{w}&={h}-Q_K(B(v_1,v_1)-B(v_2,v_2)),\quad M\in\cM_\tht^{mut}\label{eq:error:dr:mut}
        \\
        \bdy_tw+\nu Aw&=h-\frac{1}2Q_KDB(z)w,\quad M\in \cM_\tht^{sym},\label{eq:error:dr:sym}
    \end{align}
where $z:=v_1+v_2$. However, upon further inspection, one in fact has the following identity
    \begin{align}\label{eq:B:DB}
        B(v_1,v_1)-B(v_2,v_2)&=B(w,w)+DB(v_2)w\notag
        \\
        &=B(w,w)+DB(z)w-DB(v_1)w\notag
        \\
        &=B(w,w)+\frac{1}2DB(z)w-\frac{1}2DB(w,w)=\frac{1}2DB(z)w.
    \end{align}
Thus \eqref{eq:error:dr:mut} and \eqref{eq:error:dr:sym} are the same equation. Nevertheless, the availability of the two different forms appearing in \eqref{eq:error:dr:mut} and \eqref{eq:error:dr:sym} will be helpful when establishing global well-posedness and the desired time-uniform bounds. An important mechanism that distinguishes the direct-replacement intertwinement from the nudging intertwinement that one gleans from \eqref{eq:error:dr:mut}, \eqref{eq:error:dr:sym} is that the low-modes of the synchronization error satisfy a linear heat equation:
    \begin{align}\label{eq:error:dr:heat}
        \bdy_tP_Kw+\nu AP_Kw=P_Kh.
    \end{align}
We will leverage this fact for addressing global well-posedness, time-uniform bounds, and synchronization.

Before proving the main theorem, let us prove a preliminary lemma analogous to \cref{lem:error:nudge:trilinear}.

\begin{Lem}\label{lem:error:dr:trilinear}
Given $v_1,v_2\in V$, let $w=v_1-v_2$, $p=P_Kw$, and $q=Q_Kw$. Then
    \begin{align}\label{eq:v1v2:trilinear}
       b(w,p,q)+b(w,v_2,q)+b(v_2,p,q)=-b(w,p,q)+b(w,v_1,q)+b(v_1,p,q)
    \end{align}
Moreover, for all $K>0$, we have
    \begin{align}
        b(w,p,q)+b(w,v_2,q)+b(v_2,p,q)
        &\geq -\frac{3}{\nu}\left(C_L^2|p|^2+C_S^2\ln(e+K)(|v_1|\wedge|v_2|)^2\right)\|p\|^2\notag
        \\
        &\quad-\left(\frac{5\nu}{12}+\frac{3C_L^2\|p\|^2}{\nu K^2}+\frac{C_L(\|v_1\|\wedge\|v_2\|)}{K}\right)\|w\|^2.\notag
    \end{align}
\end{Lem}

\begin{proof}
The identity \eqref{eq:v1v2:trilinear} follows directly from orthogonality (of $p$ and $q$) and \eqref{eq:B:skew}.

To estimate the trilinear terms, first observe that
    \begin{align}
        &b(w,p,q)+b(w,v_2,q)+b(v_2,p,q)\notag
        \\
        &=b(p,p,q)+b(q,p,q)+b(p,v_2,q)+b(q,v_2,q)+b(v_2,p,q)\notag
        \\
        &=-b(p,q,p)+b(q,p,q)+b(p,v_2,q)+b(q,v_2,q)-b(v_2,q,p)\notag.
    \end{align}
We now estimate each of the above  trilinear terms with H\"older's inequality, \eqref{est:Sobolev}, \eqref{est:interpolation}, \eqref{est:Bernstein}, \eqref{est:Poincare}, and Young's inequality:
    \begin{align}
        |b(p,q,p)|&\leq C_L\|p\||p|\|q\|\leq \frac{3C_L^2}{\nu}\|p\|^2|p|^2+\frac{\nu}{12}\|w\|^2\notag
        \\
        |b(q,p,q)|&\leq C_L\|q\||q|\|p\|\leq \frac{C_L}{K}\|q\|^2\|p\|\notag
        \\
        &\leq\frac{3C_L^2}{\nu K^2}\|p\|^2\|w\|^2+\frac{\nu}{12}\|w\|^2\notag
        \\
        |b(p,v_2,q)|&\leq C_S\ln(e+K)^{1/2}\|p\|\|v_2\||q|\leq \frac{C_S\ln(e+K)^{1/2}}{K}\|p\|\|v_2\|\|q\|\notag
        \\
        &\leq \frac{3C_S^2\ln(e+K)}{\nu K^2}\|p\|^2\|v_2\|^2+\frac{\nu}{12}\|w\|^2\notag
        \\
        |b(q,v_2,q)|&\leq C_L\|q\||q|\|v_2\|\leq \frac{C_L}{K}\|v_2\|\|q\|^2\leq \frac{C_L}{K}\|v_2\|\|w\|^2\notag
        \\
        |b(v_2,q,p)|&\leq C_S\ln(e+K)^{1/2}|v_2|\|q\|\|p\|\leq \frac{5C_S^2\ln(e+K)}{2\nu}|v_2|^2\|p\|^2+\frac{\nu}{12}\|w\|^2.\notag
    \end{align}
Due to \eqref{eq:v1v2:trilinear}, we may replace all instances of $v_2$ in the estimates above with $v_1$. We thus deduce that
    \begin{align}
        b(w,p,q)+b(w,v_2,q)+b(v_2,p,q)
        &\geq -\frac{3}{\nu}\left(C_L^2|p|^2+C_S^2\ln(e+K)(|v_1|\wedge|v_2|)^2\right)\|p\|^2\notag
        \\
        &\quad-\left(\frac{5\nu}{12}+\frac{3C_L^2\|p\|^2}{\nu K^2}+\frac{C_L(\|v_1\|\wedge\|v_2\|)}{K}\right)\|w\|^2,\notag
    \end{align}
as desired.
\end{proof}

\begin{Prop}\label{prop:intertwined:dr}
Suppose that condition \eqref{cond:intertwined:uniform} holds. Then the direct-replacement intertwinement is self-synchronous for $K\geq 12C_L\Cm$.
\end{Prop}

\begin{proof}
We seek to apply \cref{thm:ss}. First observe that we can make use of \eqref{eq:B:DB} and \eqref{eq:B:skew} to write
    \begin{align}\notag
        \lp Q_K(B(v_1,v_1)-B(v_2,v_2)),w\rp&=b(w,p,q)+b(w,v_2,q)+b(v_2,p,q).
    \end{align}
In other words, for either $M\in \cM_\tht^{mut}\cup\cM_\tht^{sym}$, the energy balance associated for the synchronization error is given by
    \begin{align}
        \frac{1}2\frac{d}{dt}|w|^2+\nu\|w\|^2&=\lp h,w\rp -b(w,p,q)-b(w,v_2,q)-b(v_2,p,q)\notag.
    \end{align}
By \cref{lem:error:dr:trilinear} and \eqref{cond:intertwined:uniform}, we deduce that
    \begin{align}
        b(w,p,q)+b(w,v_2,q)+b(v_2,p,q)
        &\geq -\frac{3}{\nu}\left(C_L^2|p|^2+C_S^2\ln(e+K)\Cm^2\nu^2\right)\|p\|^2\notag
        \\
        &\quad-\nu\left(\frac{5}{12}+\frac{3C_L^2\|p\|^2}{\nu^2 K^2}+\frac{C_L\Cm}{K}\right)\|w\|^2.\notag
    \end{align}
Thus, \eqref{cond:coercivity:ss} is verified for $N_*=12C_L\Cm$, $t_*=t_0$, $\eps=1/2$. We conclude the proof with an application of \cref{thm:ss}.
\end{proof}

Finally, we prove \cref{thm:dr:intertwinement}, \cref{thm:dr:intertwinement:sym}. 

\begin{proof}[Proofs of \cref{thm:dr:intertwinement}, \cref{thm:dr:intertwinement:sym}]

For $M\in\cM_\tht^{mut}$, we apply \cref{thm:uniform:dr:mut} to verify \eqref{cond:intertwined:uniform}. For $M\in\cM_\tht^{sym}$ with $\tht_1=1$, we apply \cref{thm:uniform:tht1} to verify \eqref{cond:intertwined:uniform}. For $M\in\cM_\tht^{sym}$ with $\tht_1=\tht_2=1/2$, we apply \cref{thm:uniform:tht12} to verify \eqref{cond:intertwined:uniform}. For $M\in\cM_\tht^{sym}$ with $1-\tht_1\ll1$ or $|\tht_1-\tht_2|\ll1$,  we apply \cref{thm:uniform:tht1tht2} to verify \eqref{cond:intertwined:uniform}. Finally, in each case, we apply \cref{prop:intertwined:dr} with an appropriate choice of $K_*$ in order to conclude the proof.
\end{proof}

\section{Apriori Estimates}\label{sect:apriori}

We are left to prove global well-posedness of the nudging and direct-replacement intertwinements (\cref{thm:intertwined:nudge:gwp}, \cref{thm:dr:mut:gwp}, \cref{thm:dr:sym:gwp}), as well as the claimed uniform-in-time bounds (\cref{thm:uniform:nudge}, \cref{thm:uniform:dr:mut}, \cref{thm:uniform:tht1}, \cref{thm:uniform:tht12}, \cref{thm:uniform:tht1tht2}). In this section, we will derive the requisite apriori estimate needed to establish global well-posedness. Subsequently, we also establish the uniform-in-time used in establishing the synchronization properties of the various intertwinements. As we mentioned at the end of \cref{sect:gwp}, since the procedure for constructing the desired solutions is standard, we omit the proofs of global well-posedness and simply refer the reader to \cite{OlsonTiti2003, AzouaniOlsonTiti2014, BiswasBrownMartinez2022} for the relevant details.

\subsection{Mutual Nudging Intertwinement}\label{sect:mutual:nudging:proof}

First we develop apriori bounds for $(v_1,v_2)$, which ultimately yield global well-posedness of \eqref{eq:intertwined:nudge}, whenever $M\in \cM_{\mu}^{mut}$. For convenience, we rewrite \eqref{eq:intertwined:nudge} in this particular case
\begin{align}\label{eq:intertwined:nudge:mutual}
        \begin{split}
        \frac{dv_1}{dt} + \nu Av_1 + B(v_1,v_1) 
        &= 
        {g_1}-\mu_1P_K v_1+ \mu_1P_K v_2,
	\\
        \frac{dv_2}{dt} + \nu Av_2 + B(v_2,v_2) 
        &= 
        {g_2} +\mu_2P_K v_1-\mu_2P_Kv_2.
        \end{split}
    \end{align}
We note that in the endpoint cases, $\mu_1=0$ or $\mu_2=0$, \eqref{eq:intertwined:nudge:mutual} reduces to \eqref{eq:nudge}. Thus, the apriori estimates available in the endpoint cases are exactly those obtained in \cite{AzouaniOlsonTiti2014}, so global well-posedness in this case follows from \cite[Theorem 6]{AzouaniOlsonTiti2014}. This leaves us to treat the case $\mu_1,\mu_2>0$. 

\begin{Lem}\label{lem:intertwined:nudge:apriori}
Suppose $g_1,g_2\in L^\infty_{loc}(0,\infty:H)$ and $v_0^1,v_0^2\in V$. For $\mu_1,\mu_2>0$, let $\lambda_1 = \frac{\mu_2}{\mu_1+\mu_2}$ and $\lambda_2 = \frac{\mu_1}{\mu_1+\mu_2}$. Then for all $K\geq0$
    \begin{align} 
        \begin{split}
        &\lam_1|v_1(t)|^2 + \lam_2|v_2(t)|^2 + {\nu}\int_{t_0}^t (\lam_1\|v_1(s)\|^2 + \lam_2\|v_2(s)\|^2)ds
        \\
        &\leq 
        (\lam_1|v_1(t_0)|^2 + \lam_2|v_2(t_0)|^2) 
        + \frac{1}{\nu} \int_{t_0}^t  \lp \lam_1|{g_1}(s)|^2 + \lam_2|{g_2}(s)|^2\rp ds,
        \end{split}\label{est:intertwined:nudge:apriori:L2}
        \\
        \begin{split}
        &\lam_1\|v_1(t)\|^2 + \lam_2\|v_2(t)\|^2 + {\nu} \int_{t_0}^t (\lam_1|Av_1(s)|^2 + \lam_2|Av_2(s)|^2) ds
        \\
        &\leq 
        (\lam_1\|v_1(t_0)\|^2 + \lam_2\|v_2(t_0)\|^2) 
        + \frac{1}{\nu} \int_{t_0}^t  (\lam_1|{g_1}(s)|^2 + \lam_2|{g_2}(s)|^2)ds.
        \end{split}\label{est:intertwined:nudge:apriori:H1}
    \end{align}
Moreover
    \begin{align}\label{est:intertwined:nudge:H1}
     \lam_1\|v_1(t)\|^2 + \lam_2\|v_2(t)\|^2&\leq 
        e^{-\nu t}(\lam_1\|v_0^1\|^2 + \lam_2\|v_0^2\|^2)+ \lam_1\gr_1^2 + \lam_2\gr_2^2,
    \end{align}
for all $t\geq t_0\geq0$. In particular, there exists $t_0=t_0(v_0^1,v_0^2,g_1,g_2)$ such that 
	\begin{align}\label{est:intertwined:nudge:H1:uniform}
		\sup_{t\geq t_0}\left(\|v_1(t)\|^2+\|v_2(t)\|^2\right)\leq 2\frac{(\lam_1\vee\lam_2)}{(\lam_1\wedge\lam_2)}\gr^2.
	\end{align}
where $\lam_1\vee\lam_2=\max\{\lam_1,\lam_2\}$ and $\lam_1\wedge\lam_2=\min\{\lam_1,\lam_2\}$.
\end{Lem}

\begin{proof}[Proof of \cref{lem:intertwined:nudge:apriori}]
Taking the $H$ inner product of $v_1, v_2$ with their respective equations in \eqref{eq:intertwined:nudge:mutual}, we obtain
    \begin{align*}
        \frac12 \frac{d}{dt} |v_1|^2 + \nu \| v_1\|^2 
        &= 
        \lp {g_1}, v_1 \rp -\mu_1 |P_K v_1|^2 + \mu_1 \lp P_K v_2, P_K v_1 \rp
    \\
        \frac12 \frac{d}{dt} |v_2|^2 + \nu \| v_2\|^2 
        &= 
        \lp {g_2}, v_2 \rp -\mu_2 |P_K v_2|^2 + \mu_2 \lp P_K v_1, P_K v_2 \rp.
    \end{align*}
Similarly, upon taking the inner product in $V$, we obtain
    \begin{align*}
        \frac12 \frac{d}{dt} \|v_1\|^2 + \nu |A v_1|^2 
        &= 
        \lp {g_1}, Av_1 \rp -\mu_1 \|P_K v_1\|^2 + \mu_1 \lp P_K v_2, P_K Av_1 \rp
    \\
        \frac12 \frac{d}{dt} \|v_2\|^2 + \nu |A v_2|^2 
        &= 
        \lp {g_2}, Av_2 \rp -\mu_2 \|P_K v_2\|^2 + \mu_2 \lp P_K v_1, P_K Av_2 \rp.
    \end{align*}

Since $\mu_1,\mu_2>0$, we see that $\lam_1,\lam_2>0$ and $\lam_1+\lam_2=1$. Then
    \begin{align}
        \begin{split}
        \frac{1}2\frac{d}{dt} (\lambda_1|v_1|^2 + \lambda_2|v_2|^2) &+  \nu(\lambda_1\|v_1\|^2 + \lambda_2\|v_2\|^2)\notag
        \\
        &=  \lambda_1\lp g_1,v_1\rp +\lambda_2\lp g_2,v_2\rp-\frac{\mu_1\mu_2}{\mu_1 + \mu_2}|P_Kv_2 - P_Kv_1|^2
        \end{split}
        \\
        \begin{split}
        \frac{1}2\frac{d}{dt} (\lambda_1\|v_1\|^2 + \lambda_2\|v_2\|^2) &+  \nu(\lambda_1|Av_1|^2 + \lambda_2|Av_2|^2)\notag
        \\
        &=  \lambda_1\lp g_1,Av_1\rp +\lambda_2\lp g_2,Av_2\rp-\frac{\mu_1\mu_2}{\mu_1 + \mu_2}\|P_Kv_2 - P_Kv_1\|^2.
        \end{split}
    \end{align}

By the Cauchy-Schwarz inequality
    \begin{align}
        \begin{split}
        & \frac{d}{dt} (\lam_1|v_1|^2 + \lam_2|v_2|^2) + {\nu} (\lam_1\|v_1\|^2 + \lam_2\|v_2\|^2)
        \\
        &\leq 
        \frac{1}{\nu} \lp 
        \lam_1|{g_1}(t)|^2 + \lam_2|{g_2}(t)|^2\rp
        -\frac{2\mu_1\mu_2}{\mu_1+\mu_2}|P_Kv_1-P_Kv_2|^2,
        \end{split}\label{est:nudge:intertwined:pregronwall:L2}
        \\
        \begin{split}
        &\frac{d}{dt} (\lam_1\|v_1\|^2 + \lam_2\|v_2\|^2) + \nu (\lam_1|Av_1|^2 + \lam_2|Av_2|^2)
        \\
        &\leq 
        \frac{1}{\nu} \lp \lam_1|{g_1}(t)|^2 + \lam_2|{g_2}(t)|^2\rp
        -\frac{2\mu_1\mu_2}{\mu_1+\mu_2}\|P_Kv_1-P_Kv_2\|^2.
        \end{split}\label{est:nudge:intertwined:pregronwall:H1}
    \end{align}
Thus, by Gr{\"o}nwall's inequality
    \begin{align} 
        \begin{split}
        &\lam_1|v_1(t)|^2 + \lam_2|v_2(t)|^2 + {\nu}\int_{t_0}^t (\lam_1\|v_1(s)\|^2 + \lam_2\|v_2(s)\|^2)ds
        \\
        &\leq 
        (\lam_1|v_1(t_0)|^2 + \lam_2|v_2(t_0)|^2) 
        + \frac{1}{\nu} \int_{t_0}^t  \lp \lam_1|{g_1}(s)|^2 + \lam_2|{g_2}(s)|^2\rp ds,
        \end{split}\notag
        \\
        \begin{split}
        &\lam_1\|v_1(t)\|^2 + \lam_2\|v_2(t)\|^2 + {\nu} \int_{t_0}^t (\lam_1|Av_1(s)|^2 + \lam_2|Av_2(s)|^2) ds
        \\
        &\leq 
        (\lam_1\|v_1(t_0)\|^2 + \lam_2\|v_2(t_0)\|^2) 
        + \frac{1}{\nu} \int_{t_0}^t  (\lam_1|{g_1}(s)|^2 + \lam_2|{g_2}(s)|^2)ds.
        \end{split}\notag
    \end{align}
On the other hand, we also have
    \begin{align} 
        \begin{split}
        \lam_1\|v_1(t)\|^2 + \lam_2\|v_2(t)\|^2&\leq 
        e^{-\nu (t-t_0)}(\lam_1\|v_1(t_0)\|^2 + \lam_2\|v_2(t_0)\|^2)\notag
        \\
        &\quad+ \lam_1\left(\frac{\sup_{t\geq{t_0}}|{g_1}(t)|}{\nu^2}\right)^2 + \lam_2\left(\frac{\sup_{t\geq{t_0}}|{g_2}(t)|}{\nu^2}\right)^2,
        \end{split}\notag
    \end{align}
for all $t_0\geq0$, as desired.
\end{proof}

\subsection{Symmetric Nudging Intertwinement}\label{sect:symmetric:nudging:proof}
We  now consider the nudging intertwinement corresponding to matrices $M\in\cM_\mu^{sym}$.

\subsubsection{Preliminaries} It will be useful to write the system in vector form and to consider a particular affine form for the force. Indeed, let $M\in \RR^{2\times2}$ be any symmetric, non-negative definite matrix of the following form:
    \begin{align}\label{def:M}
        M=\begin{pmatrix} 
        \mu_1&-\mu_2\\ 
        -\mu_2&\mu_1
        \end{pmatrix}.
    \end{align}
Then observe that for $v=(v_1,v_2)^t$ and $g,\tg,{g_{\mut}}$ given by \eqref{def:G:tG}, \eqref{eq:intertwined:nudge} can be rewritten as
    \begin{align}\label{eq:nse:V}
        \frac{dv}{dt}+\nu Av+B(v)={g_{\mut}}+\mut\tg-MP_Kv,
    \end{align}
where
    \begin{align}\label{def:B:vector}
        B(v)=\begin{pmatrix} B(v_1,v_1)\\ B(v_2,v_2)
        \end{pmatrix}.
    \end{align}
Recall that $M$ induces an inner product
	\begin{align}\label{def:M:ip}
		\lb u,u'\rb_M=u^tMu',\quad u,u'\in\RR^2.
	\end{align}
In particular, $M$ induces an inner product in $H\times H$ via
	\begin{align}\label{def:M:H}
		\lp u,u'\rp_M=\lp Mu,u'\rp,\quad |u|_M=|M^{1/2}u|,\quad u,u'\in H\times H,
	\end{align}
where $M^{1/2}M^{1/2}=M$. Similarly, $M$ an induces an inner product in $V\times V$ via
	\begin{align}\label{def:M:V}
		\|u\|_M=\|M^{1/2}u\|=|M^{1/2}A^{1/2}u|,
	\end{align}
Observe that the eigenvalues of $M$ are given by
    \begin{align}\label{eq:M:eigenvalues}
        \lam_1=\mu_1-\mu_2,\quad \lam_2=\mu_1+\mu_2.
    \end{align}
Since $M$ is assumed to be non-negative definite, it immediately follows that $\mu_1\geq\mu_2$. One may then directly verify that
    \[
        \lam_1|u|^2\leq u^tMu\leq \lam_2|u|^2,
    \]
which, in turn, implies
    \begin{align}\label{est:M:equiv}
        \begin{split}
        \lam_1|u|^2\leq |&u|_M^2\leq \lam_2|u|^2,\quad \text{for all}\ u\in H\times H\\
        \lam_1\|u\|^2\leq \|&u\|_M^2\leq \lam_2\|u\|^2,\quad \text{for all}\ u\in V\times V
        \end{split}
    \end{align}
Also, observe that for any $M\in\RR^{2\times2}$ and $K>0$, $MP_K=P_KM$ and $MA^{m/2}=A^{m/2}M$, for all integers $m$. With these basic facts in mind, we develop apriori estimates for \eqref{eq:nse:V}. 

\subsubsection{Apriori estimates}

\begin{Lem}\label{lem:intertwined:symmetric:apriori}
Let $v_0=(v_0^1,v_0^2)\in V\times V$ and $\mut>0$. Then
	\begin{align}\notag
		\|v(t)\|^2 + \nu \int_0^T|Av(t)|^2dt\leq \|v_0\|^2+\frac{1}{\nu}\int_0^T\min\left\{|g(t)|^2,|{g_{\mut}}(t)|^2+\mut^2|\tg(t)|^2\right\} dt,
	\end{align}
holds for all $t\geq[0,T]$ and $T\geq0$. Moreover
	\begin{align}\notag
		\|v(t)\|^2\leq \|v_0\|^2e^{-\nu t}+\nu^2\min\{\gr^2,\gr_{\mut}^2+\mu\tgr^2\}(1-e^{-\nu t}),
	\end{align}
holds for all $t\geq0$. On the other hand, if $\mu_1>\mu_2$ and $K$ satisfies
    \begin{align}\label{cond:K:symmetric:apriori}
        K^2\geq \lam_2,
    \end{align}
then
    \begin{align}\notag
        \|v(t)\|^2\leq e^{-\lam_1t}\|v_0\|^2+\frac{\nu^3}{\lam_1}\min\left\{\gr^2,\gr_{\mut}^2+\frac{\mut^2}{\nu\lam_1}\tgr^2\right\}.
    \end{align}
\end{Lem}

\begin{proof}
Upon recalling \eqref{def:G:tG}, we first take the $H$--inner product of  \eqref{eq:nse:V} with $Av$ to write
    \begin{align}
        \frac{1}{2} \frac{d}{dt} \|v\|^2 + \nu |Av|^2
        &=
        \lp g,Av \rp-\lp MP_Kv, Av\rp\label{eq:V:g}
        \\
        &=
        \lp {g_{\mut}},Av \rp+\mut\lp \tg,Av \rp-\lp MP_Kv, Av\rp.\label{eq:V:G:tG}
    \end{align}
By the Cauchy-Schwarz inequality and Young's inequality, we have
    \begin{align}
        |\lp g,Av\rp|&\leq |g||Av|\leq \frac{1}{2\nu}|g|^2+\frac{\nu}2|Av|^2\notag
        \\
         |\lp {g_{\mut}},Av\rp|&\leq |{g_{\mut}}||Av|\leq \frac{1}{2\nu}|{g_{\mut}}|^2+\frac{\nu}2|Av|^2\notag
        \\
        \mut|\lp \tg,Av\rp|&\leq \mut|\tg||Av|\leq \frac{\mut^2}{2\nu}|\tg|^2+\frac{\nu}2|Av|^2\notag.
    \end{align}
Upon integrating by parts and applying the assumption $M\geq0$, we deduce
    \begin{align}
        \lp MP_KV, Av\rp=\|P_Kv\|_M^2\geq0.\notag
    \end{align}
By \eqref{est:Poincare}, it now follows that
    \begin{align}\notag
        \frac{d}{dt} \|v\|^2 +\nu\|v\|^2\leq \frac{d}{dt} \|v\|^2 +\nu|Av|^2\leq \min\left\{\frac{|g|^2}{\nu},\frac{|{g_{\mut}}|^2}{\nu}+\mut^2\frac{|\tg|^2}{\nu}\right\}\leq \nu^3\min\{\gr^2,\gr_{\mut}^2+\mu\tgr^2\}
    \end{align}
Integrating over $[0,T]$ yields
    \begin{align}\notag
        \|v(t)\|^2 + \nu \int_0^T|Av(t)|^2dt
        &\leq |v_0|^2+\frac{1}{\nu}\int_0^T\min\left\{|g(t)|^2,|{g_{\mut}}(t)|^2+\mut^2|\tg(t)|^2\right\} dt,
    \end{align}
for all $0\leq t\leq T$. On the other hand, Gr\"onwall's inequality implies
    \begin{align}\notag
        \|v(t)\|^2\leq \|v_0\|^2e^{-\nu t}+\nu^2\min\{\gr^2,\gr_{\mut}^2+\mu\tgr^2\}(1-e^{-\nu t}).
    \end{align}
This establishes the first two inequalities.

We may alternatively proceed by observing that energy balance can also be written as
    \begin{align}
        \frac{1}{2} \frac{d}{dt} \|v\|^2 + \nu |Av|^2
        &=
        \lp g,Av \rp-\|v\|_M^2+\|(I-P_K)v\|_M^2\notag\\
        &=
        \lp {g_{\mut}},Av \rp+\mut\lp \tg,Av\rp-\|v\|_M^2+\|(I-P_K)v\|_M^2.\notag
    \end{align}
 Observe that \eqref{est:Bernstein} implies $\|(I-P_K)v\|_M^2\leq K^{-2}|Av|_M^2$. Now recall that if $\mu_1>\mu_2$, then \eqref{est:M:equiv} holds and $\lam_1>0$. Also, recall that $\mut>0$, so that we may alternatively estimate
    \begin{align}
        \mut|\lp \tg,Av\rp|&\leq \mut\|\tg\|\|v\|\leq \frac{\mut^2}{2\lam_1}\|\tg\|^2+\frac{\lam_1}2\|v\|^2\notag.
    \end{align}
Combining these observations and estimating the remaining terms as before yields
    \begin{align}
        \frac{d}{dt}\|v\|^2+\nu|Av|^2+\lam_1\|v\|^2\leq \frac{|g|^2}{\nu}\notag,
    \end{align}
and
    \begin{align}
        \frac{d}{dt}\|v\|^2+\left(\nu-\frac{\lam_2}{K^2}\right)|Av|^2+\lam_1\|v\|^2\leq \frac{|{g_{\mut}}|^2}{\nu}+\frac{\mut^2}{\nu\lam_1}\frac{|\tg|^2}{\nu}\notag.
    \end{align}
Since $K$ is assumed to satisfy \eqref{cond:K:symmetric:apriori}, we deduce
    \begin{align}
        \|v(t)\|^2\leq e^{-\lam_1t}\|v_0\|^2+\frac{\nu^3}{\lam_1}\min\left\{\gr^2,\gr_{\mut}^2+\frac{\mut^2}{\nu\lam_1}\tgr^2\right\},\notag
    \end{align}
as desired.
\end{proof}

We immediately deduce the following.

\begin{Cor}\label{cor:intertwined:symmetric:apriori}
Let $v_0 = (v_0^1,v_0^2)\in V\times V$. Then
	\begin{align}\label{est:symmetric:apriori:mu}
		\sup_{t\geq t_0}\|v(t)\|^2\leq 2\nu^2\min\{\gr^2,\gr_{\mut}^2+\mu\tgr^2\},
	\end{align}
where 
	\begin{align}\label{cond:t0:tV}
		t_0\geq\frac{1}{\nu}\ln\left(\frac{\|v_0\|^2}{\nu^2\min\{\gr^2,\gr_{\mut}^2+\mu\tgr^2\}}\right).
	\end{align}
Moreover, if $\lam_1>0$, $\mut=\lam_1$, and $K$ satisfies \eqref{cond:K:symmetric:apriori}, then
    \begin{align}\label{est:symmetric:apriori:nomu}
        \sup_{t\geq t_0}\|v(t)\|^2\leq 2\nu^2\min\left\{\frac{\nu}{\lam_1}\gr^2, \frac{\nu}{\lam_1}\gr_{\mut}^2+\tgr^2\right\}
    \end{align}
such that
    \begin{align}\label{cond:t0:tV:nomu}
       t_0\geq\frac{1}{\lam_1}\ln\left[\frac{\|v_0\|^2}{{\nu^2}\min\left\{\frac{\nu}{\lam_1}\gr^2,\frac{\nu}{\lam_1}\gr_{\mut}^2+\tgr^2\right\}}\right].
    \end{align}
\end{Cor}

\subsection{Mutual Direct-Replacement Intertwinement}\label{sect:mutual:sync:proof}

Let us now consider the direct-replacement intertwinement corresponding to the matrices $M\in\cM_{\tht}^{mut}$. For ${g_1},{g_2}\in L^\infty(0,\infty;H)$, define
	\begin{align}\label{def:tGrlam}
		\gr_\tht:=\frac{\sup_{t\geq0}|{g^\tht}(t)|}{\nu^2},\quad {g^\tht}:=\tht_2{g_1}+\tht_1{g_2},\quad \tht_1+\tht_2=1.
	\end{align}
Due to the symmetry in $M\in\cM_\tht^{mut}$, we may assume that $0\leq \tht_1\leq \tht_2$ without loss of generality.

\subsubsection{Preliminaries} 

We introduce a change of variables that will be convenient for the analysis and collect some useful preliminary facts. To this end, let 
	\begin{align}\label{def:w:h:sync}
		w=v_1-v_2,\quad {h}={g_1}-{g_2}.
	\end{align}
Then the system governing ${w}$ is given by
	\begin{align}\label{eq:tw:sync:pre}
		\bdy_t{w}+\nu A{w}={h}-Q_K(B(v_1,v_1)-B(v_2,v_2)).
	\end{align}
Next, we define the ``twisted variable" $v^\tht$ given by
	\begin{align}\label{def:vlam}
		{v^\tht}=\tht_2v_1+\tht_1v_2.
	\end{align}
Observe that
	\begin{align}\label{eq:B:diff:sync}
	B(v_1,v_1)-B(v_2,v_2)&=\tht_1(B(v_1,v_1)-B(v_2,v_2))-(1-\tht_1)(B(v_2,v_2)-B(v_1,v_1))\notag
	\\
	&=\tht_1(B({w},{w})+DB(v_2){w})-(1-\tht_1)(B({w},{w})-DB(v_1)({w}))\notag
	\\
	&=(2\tht_1-1)B({w},{w})+DB(\tht_1v_2+(1-\tht_1)v_1){w},
	\end{align}
Then, using the fact that $\tht_1+\tht_2=1$, we obtain
	\begin{align}\label{eq:tw:sync}
		\bdy_t{w}+\nu A{w}={h}-(\tht_1-\tht_2)Q_KB({w},{w})-Q_KDB({v^\tht}){w}.
	\end{align}
Whenever $\tht_1\neq0$, we introduce the rescaled variables
    \begin{align}\label{def:wh:rescale}
        w_\tht=\sqrt{\tht_1\tht_2}{w},\quad h_\tht=\sqrt{\tht_1\tht_2}h.
    \end{align}
Then we can also rewrite \eqref{eq:tw:sync} as
    \begin{align}\label{eq:tw:sync:rescale}
		\bdy_t{w_\tht}+\nu A{w_\tht}={h_\tht}-\frac{\tht_1-\tht_2}{\sqrt{\tht_1\tht_2}}Q_KB({w_\tht},{w_\tht})-Q_KDB({v^\tht}){w_\tht}.
	\end{align}

Now, for 
	\begin{align}\label{def:p:q:glam}
		{p}=P_K{w},\quad {q}=Q_K{w},\quad p_\tht=P_K{w_\tht},\quad q_\tht=Q_K{w_\tht},\quad{g^\tht}=\tht_2{g_1}+\tht_1{g_2},
	\end{align}
we have, for any $\tht_1,\tht_2\geq0$ such that $\tht_1+\tht_2=1$: 
\begin{align}\label{eq:vlam:pre}
		&\bdy_t{v^\tht}+\nu A{v^\tht}+\tht_2B(v_1,v_1)+\tht_1B(v_2,v_2)=g^\tht.
	\end{align}
Thus, for $\tht_1=0$:
	\begin{align}\label{eq:tp:tq:sync}
		\begin{split}
		\bdy_t{p}+\nu A{p}&=P_K{h}
		\\
		\bdy_t{q}+\nu A{q}&=Q_K{h}-(\tht_1-\tht_2)Q_KB({w},{w})-Q_KDB({v^\tht}){w};
		\end{split}
	\end{align}
and for $\tht_1\neq0$:
\begin{align}\label{eq:tp:tq:sync:rescale}
		\begin{split}
		\bdy_t{p_\tht}+\nu A{p_\tht}&=P_K{h_\tht}
		\\
		\bdy_t{q_\tht}+\nu A{q_\tht}&=Q_K{h_\tht}-\frac{\tht_1-\tht_2}{\sqrt{\tht_1\tht_2}}Q_KB({w_\tht},{w_\tht})-Q_KDB({v^\tht}){w_\tht}.
		\end{split}
	\end{align}
Focusing now on the left-hand side of \eqref{eq:vlam:pre}, we see that
	\begin{align}
		&\tht_2B(v_1,v_1)+\tht_1B(v_2,v_2)\notag
		\\
		&=B(\tht_2v_1,v_1)+B(\tht_1v_2,v_2)\notag
		\\
		&=B({v^\tht},v_1)-B(\tht_1v_2,v_1)+B(\tht_1v_2,\tht_2v_2)+B(\tht_1v_2,\tht_1v_2)\notag
		\\
		&=B({v^\tht},\tht_2v_1)+B({v^\tht},\tht_1v_1)\notag
            \\
            &\quad-\tht_1\tht_2B(v_2,v_1)-\tht_1^2B(v_2,v_1)+\tht_1\tht_2B(v_2,v_2)+\tht_1^2B(v_2,v_2)\notag
		\\
		&=B({v^\tht},{v^\tht})-B({v^\tht},\tht_1v_2)\notag
            \\
            &\quad+\tht_2\tht_1B(v_1,v_1)-\tht_1\tht_2B(v_2,v_1)+\tht_1\tht_2B(v_2,v_2)+\tht_1^2B(v_2,v_2)\notag
		\\
		&=B({v^\tht},{v^\tht})-\tht_2\tht_1B(v_1,v_2)+\tht_2\tht_1B(v_1,v_1)-\tht_1\tht_2B(v_2,v_1)+\tht_1\tht_2B(v_2,v_2)\notag
		\\
		&=B({v^\tht},{v^\tht})+\tht_2\tht_1B(v_1,{w})-\tht_1\tht_2B(v_2,{w})\notag
		\\
		&=B({v^\tht},{v^\tht})+\tht_1\tht_2B({w},{w})\notag.
	\end{align}
Therefore \eqref{eq:vlam:pre} reduces to
	\begin{align}\label{eq:vlam:final}
		\bdy_t{v^\tht}+\nu A{v^\tht}+B({v^\tht},{v^\tht})={g^\tht}-B({w_\tht},{w_\tht}).
	\end{align}

\begin{Rmk}\label{rmk:eigen}
Note that the ``twisted variable," $v^\tht$, is simply obtained by taking multiplying  the left eigenvector, $(\tht_2,\tht_1)$, of the intertwining matrix corresponding to eigenvalue zero with $v$. This is the reason for the fortuitous cancellations seen above.
\end{Rmk}
	
\subsubsection*{Summary} When $\tht_1=0$, we may write the system satisfied by $({v^\tht},{w})$ as
        \begin{align}\label{eq:vlam:tw}
            \begin{split}
            \bdy_t{v^\tht}+\nu A{v^\tht}+B({v^\tht},{v^\tht})&={g^\tht}
            \\
            \bdy_t{w}+\nu A{w}&={h}+Q_KB({w},{w})-Q_KDB({v^\tht}){w},
            \end{split}
        \end{align}
which, upon combining \eqref{eq:tp:tq:sync} and \eqref{eq:vlam:final}, we may equivalently rewrite as
	\begin{align}\label{eq:vlam:tp:tq}
		\begin{split}
		\bdy_t{v^\tht}+\nu A{v^\tht}+B({v^\tht},{v^\tht})&={g^\tht}
		\\
		\bdy_t{p}+\nu A{p}&=P_K{h}
		\\
		\bdy_t{q}+\nu A{q}&=Q_K{h}+Q_KB({w},{w})-Q_KDB({v^\tht}){w}.
		\end{split}
	\end{align}

When $\tht_1\neq0$, we may equivalently consider the system satisfied by $({v^\tht},{w_\tht})$ given by
    \begin{align}\label{eq:vlam:tw:rescale}
            \begin{split}
            \bdy_t{v^\tht}+\nu A{v^\tht}+B({v^\tht},{v^\tht})&={g^\tht}-B({w_\tht},{w_\tht})
            \\
            \bdy_t{w_\tht}+\nu A{w_\tht}&={h_\tht}-\frac{\tht_1-\tht_2}{\sqrt{\tht_1\tht_2}}Q_KB({w_\tht},{w_\tht})-Q_KDB({v^\tht}){w_\tht},
            \end{split}
        \end{align}
which, upon combining \eqref{eq:tp:tq:sync:rescale} and \eqref{eq:vlam:final}, we may equivalently rewrite as
	\begin{align}\label{eq:vlam:tp:tq:rescale}
		\begin{split}
		\bdy_t{v^\tht}+\nu A{v^\tht}+B({v^\tht},{v^\tht})&={g^\tht}-B(w_\tht,w_\tht)
		\\
		\bdy_t{p_\tht}+\nu A{p_\tht}&=P_K{h_\tht}
		\\
		\bdy_t{q_\tht}+\nu A{q_\tht}&=Q_K{h_\tht}-\frac{\tht_1-\tht_2}{\sqrt{\tht_1\tht_2}}Q_KB({w_\tht},{w_\tht})-Q_KDB({v^\tht}){w_\tht}.
		\end{split}
	\end{align}
 
\subsubsection*{Energy balances} When $\tht_1=0$, the energy balance for \eqref{eq:vlam:tw} is given by
    \begin{align}\label{eq:balance:vlam:tw}
        \begin{split}
        \frac{1}2\frac{d}{dt}|{v^\tht}|^2+\nu\|{v^\tht}\|^2&=\lp {g^\tht},{v^\tht}\rp
        \\
        \frac{1}2\frac{d}{dt}|{w}|^2+\nu\|{w}\|^2&=\lp {h},{w}\rp+b(w,w,q)
		-b(v^\tht,w,q)+b(w,q,v^\tht).
        \end{split}
    \end{align}
From \eqref{eq:vlam:tp:tq}, we equivalently obtain the energy balance 
	\begin{align}\label{eq:balance:vlam:tp:tq}
		\begin{split}
		\frac{1}2\frac{d}{dt}|{v^\tht}|^2+\nu\|{v^\tht}\|^2&=\lp {g^\tht},{v^\tht}\rp
		\\
		\frac{1}2\frac{d}{dt}\|{p}\|^2+\nu|A{p}|^2&=\lp {h},{Ap}\rp
		\\
		\frac{1}2\frac{d}{dt}|{q}|^2+\nu\|{q}\|^2&=\lp {h},{q}\rp+b(p,p,q)+b(q,p,q)
		\\
		&\quad-b(v^\tht,p,q)+b(p,q,v^\tht)+b(q,q,v^\tht).
		\end{split}
	\end{align}

When $\tht_1\neq0$, the energy balance for \eqref{eq:vlam:tw:rescale} is given by
\begin{align}\label{eq:balance:vlam:tw:rescale}
        \begin{split}
        \frac{1}2\frac{d}{dt}|{v^\tht}|^2+\nu\|{v^\tht}\|^2&=\lp {g^\tht},{v^\tht}\rp-
        b(w_\tht,p_\tht,v^\tht)-b(w_\tht,q_\tht),v^\tht)
        \\
        \frac{1}2\frac{d}{dt}|{w_\tht}|^2+\nu\|{w_\tht}\|^2&=\lp {h_\tht},{w_\tht}\rp-\frac{\tht_1-\tht_2}{\sqrt{\tht_1\tht_2}}b(w_\tht,w_\tht,q_\tht)
        \\
        &\quad-b(v^\tht,w_\tht,q_\tht)+b(w_\tht,q_\tht,v^\tht).
        \end{split}
    \end{align}
From \eqref{eq:vlam:tp:tq:rescale}, we equivalently obtain the energy balance 
	\begin{align}\label{eq:balance:vlam:tp:tq:rescale}
		\begin{split}
		\frac{1}2\frac{d}{dt}|{v^\tht}|^2+\nu\|{v^\tht}\|^2&=\lp {g^\tht},{v^\tht}\rp-b(p_\tht,p_\tht,v^\tht)-b(q_\tht,p_\tht,v^\tht)
		\\
        &\quad-b(p_\tht,q_\tht,v^\tht)-b(q_\tht,q_\tht,v^\tht)
        \\
		\frac{1}2\frac{d}{dt}\|{p_\tht}\|^2+\nu|A{p_\tht}|^2&=\lp {h_\tht},{Ap_\tht}\rp
		\\
		\frac{1}2\frac{d}{dt}|{q_\tht}|^2+\nu\|{q_\tht}\|^2&=\lp {h_\tht},{q_\tht}\rp-\frac{\tht_1-\tht_2}{\sqrt{\tht_1\tht_2}}b(p_\tht,p_\tht,q_\tht)-\frac{\tht_1-\tht_2}{\sqrt{\tht_1\tht_2}}b(q_\tht,p_\tht,q_\tht)
		\\
		&\quad-b(v^\tht,p_\tht,q_\tht)+b(p_\tht,q_\tht,v^\tht)+b(q_\tht,q_\tht,v^\tht).
		\end{split}
	\end{align}

It is from these formulations of \eqref{eq:intertwined:sync} that we will establish apriori estimates. In particular, in obtaining bounds for $(v^\tht,w)$ when $\tht_1=0$ and $(v^\tht,w_\tht)$ when $\tht_1\neq0$, we automatically obtain bounds for $(v_1, v_2)$. Ultimately, we will do so by obtaining bounds for $(v^\tht,p,q)$ and $(v^\tht,p_\tht,q_\tht)$.

\subsubsection{Apriori estimates}\label{sect:apriori:dr}
To this end, the first bound we establish is based on the observation from \eqref{eq:vlam:tp:tq}, \eqref{eq:vlam:tp:tq:rescale} that the low modes of the difference, $p=P_Kw$, satisfies a heat equation. Thus, the low-mode evolution of $w$ obeys the bounds asserted by \cref{lem:heat:apriori}. In particular, $p$ exists \textit{independently} of the solution theory of \eqref{eq:intertwined:sync}.

Lastly, we note that although the case $\tht_1=0$ corresponds to \eqref{eq:sync}, which was studied in \cite{OlsonTiti2003}, we must develop the apriori estimates carefully in both cases to obtain precise dependence of $N$ on the system parameters for their subsequent application. We thus carry out the analysis in both cases for clarity and completeness. 

\begin{Lem}\label{lem:apriori:sync:lam0}
Suppose that $\tht_1=0$. Given $g_1,g_2\in L^\infty(0,\infty;H)$, $v_0^1,v_0^2\in V$, we have that
    \begin{align}\notag
        \|v^\tht(t)\|^2+\nu\int_0^T|Av_1(s)|^2\leq \|v_1^0\|^2+\nu^2\left(\frac{\sup_{t\geq 0}|g_1(t)|}{\nu^2}\right)^2,
    \end{align}
holds for all $0\leq t\leq T$, and any $T>0$. Moreover
    \begin{align}\label{est:tv1:H1:abs:ball}
		\sup_{t\geq t_0}\|v^\tht(t)\|^2\leq 2\nu^2\gr_0^2,
    \end{align}
whenever $t_0$ satisifes
	\begin{align}\label{def:t0:tv1:H1}
		t_0\geq\frac{1}{\nu}\ln_+\left(\frac{\|v_0^1\|^2}{\nu^2\gr_0^2}\right).
	\end{align}
On the other hand, for $p=P_Kv_1-P_Kv_2$ and $q=Q_Kv_1-Q_Kv_2$ one has the following differential inequality
    \begin{align}\label{est:sync:q:ineq}
         \frac{d}{dt}\|q\|^2+\nu|Aq|^2
        &\leq 8\nu^3\left(\frac{|h|}{\nu^2}\right)^2+4\textcolor{blue}{}(C_L^2+C_A)K\left(\frac{\|v^\tht\|}{\nu}\right)^2\nu\|p\|^2\notag
        \\
        &\quad+12\nu\textcolor{blue}{}\left[C_A\left(\frac{\|p\|}{\nu}\right)^2+36\left(\textcolor{blue}{}C_L^4+\frac{C_A^2}{K^2}\right)\left(\frac{\|v^\tht\|}{\nu}\right)^4\right]\|q\|^2,
    \end{align}
for all $t>0$. In particular
    \begin{align}\notag
        \sup_{0\leq t\leq T}\|q(t)\|^2+\int_0^T|Aq(t)|^2dt<\infty,
    \end{align}
for all $T>0$.
\end{Lem}

Next, we state apriori bounds in the case $\tht_1\neq0$.

\begin{Lem}\label{lem:apriori:sync:lam}
Suppose that $\tht_1\neq0$. Given $g_1,g_2\in L^\infty(0,\infty;H)$, $v_0^1,v_0^2\in V$, we have that
    \begin{align}\notag
        \sup_{0\leq t\leq T}\left(|v^\tht(t)|^2+|w_\tht(t)|^2\right)+\nu\int_0^T\left(\|{v^\tht}(t)\|^2+\|w_\tht(t)\|^2\right)dt<\infty
    \end{align}
and 
    \begin{align}\notag
        \sup_{0\leq t\leq T}\left(\|v^\tht(t)\|^2+\|w_\tht(t)\|^2\right)+\int_0^T\left(|Av^\tht(t)|^2+|Aw_\tht(t)|^2\right)dt<\infty,  
    \end{align}
hold for all $T>0$.
\end{Lem}

Let us now prove \cref{lem:apriori:sync:lam0} and \cref{lem:apriori:sync:lam}.
\begin{proof}[Proof of \cref{lem:apriori:sync:lam0}]

Observe that since $\tht_1=0$, we have ${v^\tht}=v_1$, ${g^\tht}={g_1}$, and $\tht_1-\tht_2=-1$. In particular $v_1$ satisfies \eqref{eq:nse:ff}. Upon taking the inner product of  \eqref{eq:vlam:tp:tq} with $Av_1$, $Ap$, and $Aq$, respectively, then making use of the identity \eqref{eq:B:enstrophy} we obtain the the following enstrophy balance:
	\begin{align}\label{eq:balance:vlam:lam0}
            \begin{split}
		\frac{1}2\frac{d}{dt}\|v_1\|^2+\nu|Av_1|^2&=\lp{g_1},Av_1\rp
            \\
  \frac{1}2\frac{d}{dt}\|p\|^2+\nu|Ap|^2&=\lp h,Ap\rp
            \\
            \frac{1}2\frac{d}{dt}\|q\|^2+\nu|Aq|^2&=\lp h,Aq\rp+ \textcolor{blue}{}\Big(b(p,p,Aq)+b(p,q,Aq)+b(q,p,Aq)
            \\
            &\quad-b(v_1,p,Aq)-b(p,v_1,Aq)-b(v_1,q,Aq)-b(q,v_1,Aq)\Big).
            \end{split}
	\end{align}

From \eqref{eq:balance:vlam:lam0}, we apply the Cauchy-Schwarz inequality and Gronwall's inequality to obtain
    \begin{align}
        \frac{d}{dt}\|v_1\|^2+\nu|Av_1|^2\leq \frac{|{g_1}|^2}{\nu}.\notag
    \end{align}
Thus
    \begin{align}\label{est:tv1:H1:basic}
        \|v_1(t)\|^2+\nu\int_0^t|Av_1(s)|^2\leq \|v_0^1\|^2+\nu^2\left(\frac{\sup_{t\geq0}|g_1(t)|}{\nu^2}\right)^2,
    \end{align}
for all $t\geq0$, and
    \begin{align}\label{est:tv1:H1}
        \|v_1(t)\|^2\leq e^{-\nu t}\|v_0^1\|^2+\nu^2\gr_0^2,
    \end{align}
where $\gr_0$ is defined by \eqref{def:tGrlam}. Upon choosing $t_0\geq0$ satisfying \eqref{def:t0:tv1:H1}, we obtain \eqref{est:tv1:H1:abs:ball}.

Next, we estimate $\|q\|^2$. For this, we apply the Cauchy-Schwarz inequality, H\"older's inequality, \eqref{est:Bernstein}, \eqref{est:interpolation}, \eqref{est:Poincare}, and Young's inequality to obtain
    \begin{align}
        |\lp h,Aq\rp|&\leq 8\nu^3\left(\frac{|h|}{\nu^2}\right)^2+\frac{\nu}{32}|Aq|^2\label{est:sync:h:lam0}
        \\
        \textcolor{blue}{}|b(p,p,Aq)|&\leq \textcolor{blue}{}C_A^{1/2}|Ap|^{1/2}|p|^{1/2}\|p\||Aq|
        \leq \textcolor{blue}{}C_A^{1/2}K^{1/2}\|p\|^{3/2}|p|^{1/2}|Aq|\notag
        \\
        &\leq \textcolor{blue}{}8C_AK\nu^3\left(\frac{\|p\|}{\nu}\right)^4+\frac{\nu}{32}|Aq|^2.\label{est:sync:ppAq:lam0}
    \end{align}
On the other hand, observe that
    \begin{align}
        b(p,q,Aq)+b(q,p,Aq)
        &=b(\bdy_jp,q,\bdy_jq)+b(\bdy_jq,p,\bdy_jq)+b(q,\bdy_jp,\bdy_jq)\notag
    \end{align}
Then by repeated application of H\"older's inequality, \eqref{est:interpolation}, \eqref{est:Bernstein}, and Young's inequality, we estimate
    \begin{align}
        \textcolor{blue}{}|b(\bdy_jp,q,\bdy_jq)|&\leq \textcolor{blue}{}C_A^{1/2}\|A p\|^{1/2}\|p\|^{1/2}\|q\|^2\leq \textcolor{blue}{}C_A^{1/2}K\|p\|\|q\|^2\notag
        \\
        &\leq \textcolor{blue}{}C_A^{1/2}\|p\|\|q\||Aq|\leq 4\textcolor{blue}{}C_A\left(\frac{\|p\|}{\nu}\right)^2\nu\|q\|^2+\frac{\nu}{16}|Aq|^2\notag
        \\
        \textcolor{blue}{}|b(\bdy_j,q,p,\bdy_jq)|
        &\leq 4\textcolor{blue}{}C_A\left(\frac{\|p\|}{\nu}\right)^2\nu\|q\|^2+\frac{\nu}{16}|Aq|^2\notag
        \\
        \textcolor{blue}{}|b(q,\bdy_jp,\bdy_jq)|
        &\leq 4\textcolor{blue}{}C_A\left(\frac{\|p\|}{\nu}\right)^2\nu\|q\|^2+\frac{\nu}{16}|Aq|^2\notag,
    \end{align}
which implies
    \begin{align}\label{est:sync:nlt:a}
        \textcolor{blue}{}|b(p,q,Aq)+b(q,p,Aq)|&\leq 12\textcolor{blue}{}C_A\nu\left(\frac{\|p\|}{\nu}\right)^2\|q\|^2+\frac{3\nu}{16}|Aq|^2.
    \end{align}
Similarly
    \begin{align}
        \textcolor{blue}{}|b(v_1,p,Aq)|&\leq \textcolor{blue}{}C_L\|v_1\|^{1/2}|v_1|^{1/2}|Ap|^{1/2}\| p\|^{1/2}|Aq|\leq \textcolor{blue}{}C_LK^{1/2}\|v_1\|^{1/2}|v_1|^{1/2}|p||Aq|\notag
        \\
        &\leq 4\textcolor{blue}{}C_L^2K\left(\frac{\|v_1\||v_1|}{\nu^2}\right)\nu\|p\|^2+\frac{\nu}{16}|Aq|^2\notag
        \\
        \textcolor{blue}{}|b(p,v_1,Aq)|&\leq \textcolor{blue}{}C_A^{1/2}|Ap|^{1/2}|p|^{1/2}\|v_1\||Aq|\leq \textcolor{blue}{}C_A^{1/2}K^{1/2}\|p\|^{1/2}|p|^{1/2}\|v_1\||Aq|\notag
        \\
        &\leq 4\textcolor{blue}{}C_AK\left(\frac{\|v_1\|}{\nu}\right)^2\nu\|p\|^2+\frac{\nu}{16}|Aq|^2,\notag
    \end{align}
which implies
    \begin{align}\label{est:sync:nlt:b}
        \textcolor{blue}{}|b(v_1,p,Aq)+b(p,v_1,Aq)|&\leq 4\textcolor{blue}{}(C_L^2+C_A)K\left(\frac{\|v_1\|}{\nu}\right)^2\nu\|p\|^2+\frac{\nu}8|Aq|^2
    \end{align}
Lastly
    \begin{align}
        \textcolor{blue}{}|b(v_1,q,Aq)|&\leq \textcolor{blue}{}C_L\|v_1\|^{1/2}|v_1|^{1/2}\|\nabla q\|^{1/2}\|q\|^{1/2}|Aq|\leq \textcolor{blue}{}C_L\|v_1\|^{1/2}|v_1|^{1/2}\|q\|^{1/2}|Aq|^{3/2}\notag
        \\
        &\leq 432\textcolor{blue}{}C_L^4\left(\frac{\|v_1\||v_1|}{\nu^2}\right)^2\nu\|q\|^2+\frac{\nu}{16}|Aq|^2\notag
        \\
        \textcolor{blue}{}|b(q,v_1,Aq)|&\leq \textcolor{blue}{}C_A^{1/2}|q|^{1/2}\|v_1\||Aq|^{3/2}\leq \textcolor{blue}{}\frac{C_A^{1/2}}{K^{1/2}}\|q\|^{1/2}\|v_1\||Aq|^{3/2}\notag
        \\
        &\leq \frac{432\textcolor{blue}{}C_A^2}{K^2}\left(\frac{\|v_1\|}{\nu}\right)^4\nu\|q\|^2+\frac{\nu}{16}|Aq|^2,\notag
    \end{align}
which implies
    \begin{align}\label{est:sync:nlt:c}
        \textcolor{blue}{}|b(v_1,q,Aq)+b(q,v_1,Aq)|&\leq  432\left(C_L^4\textcolor{blue}{}+\frac{\textcolor{blue}{}C_A^2}{K^2}\right) \left(\frac{\|v_1\|}{\nu}\right)^4\nu\|q\|^2+\frac{\nu}8|Aq|^2.
    \end{align}

Finally, we combine \eqref{est:sync:h:lam0}, \eqref{est:sync:ppAq:lam0}, \eqref{est:sync:nlt:a}, \eqref{est:sync:nlt:b}, \eqref{est:sync:nlt:c} to arrive at
    \begin{align}
        \frac{d}{dt}\|q\|^2+\nu|Aq|^2
        &\leq 8\nu^3\left(\frac{|h|}{\nu^2}\right)^2+4\textcolor{blue}{}(C_L^2+C_A)K\left(\frac{\|v_1\|}{\nu}\right)^2\nu\|p\|^2\notag
        \\
        &\quad+12\nu\textcolor{blue}{}\left[C_A\left(\frac{\|p\|}{\nu}\right)^2+36\left(\textcolor{blue}{}C_L^4+\frac{C_A^2}{K^2}\right)\left(\frac{\|v_1\|}{\nu}\right)^4\right]\|q\|^2.\notag
    \end{align}
Application of \eqref{est:tv1:H1}, \eqref{est:p:L2}, followed by Gronwall's inequality yields finiteness of $\|q(t)\|$ and $\int_0^t|Aq(s)|^2ds$, for all $t>0$. 
\end{proof}

\begin{proof}[Proof of \cref{lem:apriori:sync:lam}]

Suppose $\tht_1\neq0$. We control $(v^\tht, w_\tht)$ in $H$ first, then $V$.

\subsubsection*{Step 1: Control of $(v^\tht,w_\tht)$ in $H$} 
From \eqref{eq:balance:vlam:tw:rescale}, one obtains the total energy balance for $({v^\tht},w_\tht)$ to be
	\begin{align}\label{eq:balance:vlam:wlam}
		\frac{1}2\frac{d}{dt}&\left(|{v^\tht}|^2+|w_\tht|^2\right)+\nu\left(\|{v^\tht}\|^2+\|w_\tht\|^2\right)\notag
            \\
            &=\lp {g^\tht},{v^\tht}\rp+b(w_\tht,p_\tht,v^\tht)\notag
            \\
            &\quad+\lp{h}_\tht,w_\tht\rp-\frac{\tht_1-\tht_2}{\sqrt{\tht_1\tht_2}}b(w_\tht,p_\tht,q_\tht)-b(v^\tht,p_\tht,q_\tht).
	\end{align}
We now estimate the terms of the right-hand side. First, by the Cauchy-Schwarz inequality and Young's inequality, we obtain
    \begin{align}
        |\lp {g^\tht},{v^\tht}\rp|&\leq |{g^\tht}||{v^\tht}|\leq \frac{4|{g^\tht}|^2}{\nu}+\frac{\nu}{16}|{v^\tht}|^2\notag
        \\|\lp{h}_\tht,w_\tht\rp|&\leq|{h}_\tht||w_\tht|\leq\frac{4|{h}_\tht|^2}{\nu}+\frac{\nu}{16}|w_\tht|^2.\notag
    \end{align}
On the other hand, we treat the trilinear terms with H\"older's inequality, \eqref{est:interpolation}, \eqref{est:Bernstein}, and Young's inequality to deduce
    \begin{align}
        |b(w_\tht,p_\tht,v^\tht)|
        &\leq C_L\|w_\tht\|^{1/2}|w_\tht|^{1/2}\|p_\tht\|\|{v^\tht}\|^{1/2}|{v^\tht}|^{1/2}\notag
        \\
        &\leq \frac{\nu}{8}\|w_\tht\|^2+\frac{3C_L^{4/3}}{2^{5/3}\nu^{1/3}}|w_\tht|^{2/3}\|p_\tht\|^{2/3}\|{v^\tht}\|^{2/3}|{v^\tht}|^{2/3}\notag
        \\
        &\leq \frac{\nu}{8}\left(\|w_\tht\|^2+\|{v^\tht}\|^2\right)+\frac{9C_L^{2}}{4}\left(\frac{\|p_\tht\|}{\nu}\right)\nu\left(|w_\tht|^2+|{v^\tht}|^2\right)\notag
    \end{align}
    \begin{align}
        |b(v^\tht,p_\tht,q_\tht)|
        &\leq C_L\|{v^\tht}\|^{1/2}|{v^\tht}|^{1/2}|Ap_\tht|^{1/2}\|p_\tht\|^{1/2}|q_\tht|\notag
        \\
        &\leq \frac{C_L}{K^{1/2}}\|{v^\tht}\|^{1/2}|{v^\tht}|^{1/2}\|p_\tht\|\|q_\tht\|\notag
        \\
        &\leq \frac{\nu}{8}\|w_\tht\|^2+\frac{4C_L^2}{\nu K^{1/2}}\|{v^\tht}\||{v^\tht}|\|p_\tht\|^2\notag
        \\
        &\leq\frac{\nu}{8}\left(\|w_\tht\|^2+\|{v^\tht}\|^2\right)+\frac{32C_L^4}{K}\left(\frac{\|p_\tht\|}{\nu}\right)^4\nu|{v^\tht}|^2\notag
    \end{align}
    \begin{align}
        \frac{|\tht_1-\tht_2|}{\sqrt{\tht_1\tht_2}}|b(w_\tht,p_\tht,q_\tht)|
        &\leq \frac{|\tht_1-\tht_2|}{\sqrt{\tht_1\tht_2}}|C_A^{1/2}|w_\tht|\|A p_\tht\|^{1/2}\|p_\tht\|^{1/2}|q_\tht|\notag
        \\
        &\leq \frac{|\tht_1-\tht_2|}{\sqrt{\tht_1\tht_2}}C_A^{1/2}|w_\tht|\|p_\tht\|\|q_\tht\|\notag
        \\
        &\leq \frac{\nu}8\|w_\tht\|^2+2C_A\frac{|\tht_1-\tht_2|^2}{\tht_1\tht_2}\left(\frac{\|p_\tht\|}{\nu}\right)^2\nu|w_\tht|^2\notag
    \end{align}
Combining these estimates, we deduce
    \begin{align}\label{est:vlam:wlam:L2:ineq}
            \frac{d}{dt}&\left(|{v^\tht}|^2+|w_\tht|^2\right)+\nu\left(\|{v^\tht}\|^2+\|w_\tht\|^2\right)\leq 8\nu^3\left[\left(\frac{|{g^\tht}|}{\nu^2}\right)^2+\left(\frac{|{h}_\tht|}{\nu^2}\right)^2\right]
            \\
            & +\left\{\frac{9C_L^{2}}{2}+\left[\frac{64C_L^4}{K}\left(\frac{\|p_\tht\|}{\nu}\right)^2+4C_A\frac{|\tht_1-\tht_2|^2}{\tht_1\tht_2}\right]\left(\frac{\|p_\tht\|}{\nu}\right)\right\}\left(\frac{\|p_\tht\|}{\nu}\right)\nu\left(|v^\tht|^2+|w_\tht|^2\right)\notag.
    \end{align}
An application of \cref{lem:heat:apriori} and Gronwall's inequality then yields
    \begin{align}\label{est:apriori:L2:lam}
    \sup_{0\leq t\leq T}\left(|v^\tht(t)|^2+|w_\tht(t)|^2\right)+\nu\int_0^T\left(\|{v^\tht}(t)\|^2+\|w_\tht(t)\|^2\right)dt<\infty,
    \end{align}
for all $T>0$.

\subsubsection*{Step 2: Control of $(v^\tht,w_\tht)$ in $V$} 
Upon taking the $H$-inner product of $A{v^\tht}$ and $Aw_\tht$ with their respective equations in \eqref{eq:vlam:tw:rescale}, we obtain the enstrophy balance
    \begin{align}\notag
        \begin{split}
        \frac{1}2\frac{d}{dt}\|{v^\tht}\|^2+\nu|A{v^\tht}|^2&=\lp{g^\tht},A{v^\tht}\rp -b(w_\tht,p_\tht,Av^\tht)-b(w_\tht,q_\tht,Av^\tht)
        \\
        \frac{1}2\frac{d}{dt}\|w_\tht\|^2+\nu|Aw_\tht|^2&=\lp{h}_\tht,Aw_\tht\rp-\frac{\tht_1-\tht_2}{\sqrt{\tht_2\tht_1}}b(w_\tht,w_\tht,Aq_\tht)-b(v^\tht,w_\tht,Aq_\tht)-b(w_\tht,v^\tht,Aq_\tht).
        \end{split}
    \end{align}
Observe that
    \begin{align}
        &-b(w_\tht,q_\tht,Av^\tht)\notag
        \\
        &=-\lp (\bdy_jw_\tht^k)(\bdy_kq_\tht^\ell),\bdy_j({v^\tht})^\ell\rp-\lp w_\tht^k\bdy_k\bdy_jq_\tht^\ell,\bdy_j({v^\tht})^\ell\rp\notag
        \\
        &=-\lp (\bdy_jw_\tht^k)(\bdy_kq_\tht^\ell),\bdy_j({v^\tht})^\ell\rp+\lp \bdy_jw_\tht^k\bdy_k\bdy_jq_\tht^\ell,({v^\tht})^\ell\rp+\lp w_\tht^k\bdy_k\bdy_j^2q_\tht^\ell,({v^\tht})^\ell\rp\notag
        \\
        &=-b(\bdy_jw_\tht,q_\tht,\bdy_jv^\tht)+b(\bdy_jw_\tht,\bdy_jq_\tht,v^\tht)+b(w_\tht,v^\tht,Aq_\tht).\notag
    \end{align}
On the other hand
    \begin{align}
        &-b(v^\tht,w_\tht,Aq_\tht)-b(w_\tht,v^\tht,Aq_\tht)\notag
        \\
        &=-\lp (\bdy_j({v^\tht})^k)(\bdy_kw_\tht^\ell),\bdy_jq_\tht^\ell\rp-\lp ({v^\tht})^k\bdy_k\bdy_jw_\tht^\ell,\bdy_jq_\tht^\ell\rp-b(w_\tht,v^\tht,Aq_\tht)\notag
        \\
        &=-b(\bdy_jv^\tht,w_\tht,\bdy_jq_\tht)-b(v^\tht,\bdy_jp_\tht,\bdy_jq_\tht)-b(w_\tht,v^\tht,Aq_\tht)\notag.
    \end{align}
Lastly, we see that
	\begin{align}
		b(w_\tht,w_\tht,Aq_\tht)&=-b(w_\tht,w_\tht,Ap_\tht)\notag
		\\
		b(\bdy_jw_\tht,\bdy_jq_\tht,v^\tht)&=b(\bdy_jw_\tht,v^\tht,\bdy_jq_\tht)\notag
	\end{align}
The enstrophy balance is then equivalently given by
    \begin{align}\label{eq:vlam:wlam:H1:balance}
        \begin{split}
    	 &\frac{1}2\frac{d}{dt}\left(\|{v^\tht}\|^2+\|w_\tht\|^2\right)+\nu\left(|A{v^\tht}|^2+|Aw_\tht|^2\right)
	 \\
	 &=\lp{g^\tht},A{v^\tht}\rp+\lp{h}_\tht,Aw_\tht\rp+\frac{\tht_1-\tht_2}{\sqrt{\tht_2\tht_1}} b(w_\tht,w_\tht,Ap_\tht)-b(v^\tht,\bdy_jp_\tht,\bdy_jq_\tht)
	 \\
&\quad
-b(\bdy_jw_\tht,q_\tht,\bdy_jv^\tht)-b(\bdy_jw_\tht,v^\tht,\bdy_jq_\tht)-b(\bdy_jv^\tht,w_\tht,\bdy_jq_\tht).
        \end{split}
    \end{align}
We now estimate the seven terms on the right-hand side above.

From the Cauchy-Schwarz inequality, we see that
	\begin{align}
		|\lp{g^\tht},A{v^\tht}\rp|&\leq |{g^\tht}||A{v^\tht}|\leq\frac{|{g^\tht}|^2}{\nu}+\frac{\nu}{4}|A{v^\tht}|^2\label{est:glam:Avlam}
		\\
		|\lp{h}_\tht,Aw_\tht\rp|&\leq\frac{|{h}_\tht|^2}{\nu}+\frac{\nu}{4}|Aw_\tht|^2 \label{est:hlam:Awlam}.
	\end{align}
For the remaining five terms, we estimate them with repeated application of \eqref{est:interpolation}, \eqref{est:Bernstein}, the Cauchy-Schwarz inequality, and Young's inequality. In particular, the last three terms can all be estimated the same way. Indeed, we have
\begin{align}
    |b(\bdy_jw_\tht,q_\tht,\bdy_jv^\tht)|&\leq C_L|Aw_\tht|^{1/2}\|w_\tht\|^{1/2}|Aq_\tht|^{1/2}\|q_\tht\|^{1/2}\|v^\tht\|\notag
    \\
    &\leq \frac{\nu}{16}|Aw_\tht|^2+\frac{4C_L^2}{\nu}\|v^\tht\|^2\|w_\tht\|\|q_\tht\|
    \notag
    \\
    &\leq \frac{\nu}{16}|Aw_\tht|^2+\frac{4C_L^2}{\nu}\|v^\tht\|\|w_\tht\|\left(\|v^\tht\|^2+\|w_\tht\|^2\right)\label{est:dwlam:qlam:dvlam}
\end{align}

\begin{align}
		|b(\bdy_jw_\tht,v^\tht,\bdy_jq_\tht)|
        &\leq\frac{\nu}{16}|Aw_\tht|^2+\frac{4C_L^2}{\nu}\|v^\tht\|\|w_\tht\|\left(\|v^\tht\|^2+\|w_\tht\|^2\right)\label{est:dwlam:vlam:dqlam}
\end{align}

	\begin{align}
	|b(\bdy_jv^\tht,w_\tht,\bdy_jq_\tht)|
         &\leq \frac{\nu}{16}|Aw_\tht|^2+\frac{4C_L^2}{\nu}\|v^\tht\|\|w_\tht\|\left(\|v^\tht\|^2+\|w_\tht\|^2\right)\label{est:dvlam:wlam:dqlam}
	\end{align}

For the other terms, observe that
	\begin{align}
	\frac{|\tht_1-\tht_2|}{\sqrt{\tht_2\tht_1}}|b(w_\tht,w_\tht,Ap_\tht)|
        &\leq \frac{|\tht_1-\tht_2|}{\sqrt{\tht_2\tht_1}}| A^{3/2}p_\tht|_\infty|w_\tht|^2\notag
        \\
        &\leq \frac{|\tht_1-\tht_2|}{\sqrt{\tht_2\tht_1}}C_A^{1/2}|A^{5/2}p_\tht|^{1/2}| A^{3/2}p_\tht|^{1/2}|w_\tht|^2\notag
        \\
        &\leq \frac{|\tht_1-\tht_2|}{\sqrt{\tht_2\tht_1}}C_A^{1/2}K^3\|p_\tht\||w_\tht|^2\label{est:wlam:wlam:Aplam}
        \\
	|b(v^\tht,\bdy_jp_\tht,\bdy_jq_\tht)|&\leq |{v^\tht}||Ap_\tht|_\infty\|q_\tht\|\notag
	\\
	&\leq C_A^{1/2}|{v^\tht}||A^2p_\tht|^{1/2}|Ap_\tht|^{1/2}\|q_\tht\|\notag
	\\
        &\leq C_A^{1/2}K|{v^\tht}|\|p_\tht\||Aq_\tht|\notag
	\\
        &\leq \frac{4C_AK^2}{\nu}\|p_\tht\|^2|v^\tht|^2+\frac{\nu}{16}|Aw_\tht|^2.\label{est:vlam:dplam:dqlam}
	\end{align}

Upon returning to \eqref{eq:vlam:wlam:H1:balance} and combining \eqref{est:glam:Avlam}, \eqref{est:hlam:Awlam}, \eqref{est:dwlam:qlam:dvlam}, \eqref{est:dwlam:vlam:dqlam}, \eqref{est:wlam:wlam:Aplam}, \eqref{est:dvlam:wlam:dqlam}, \eqref{est:vlam:dplam:dqlam} we obtain 
	\begin{align}
		&\frac{d}{dt}\left(\|{v^\tht}\|^2+\|w_\tht\|^2\right)+\nu\left(|A{v^\tht}|^2+|Aw_\tht|^2\right)\notag
        \\
        &\leq\nu^3\left[\left(\frac{|g^\tht|}{\nu^2}\right)+\left(\frac{|h_\tht|}{\nu^2}\right)\right]+\frac{12C_L^2}{\nu}\|v^\tht\|\|w_\tht\|\left(\|v^\tht\|^2+\|w_\tht\|^2\right)\notag
        \\
        &\quad+ 2C_A^{1/2}K^2\left[\frac{|\tht_1-\tht_2|}{\sqrt{\tht_2\tht_1}}K+4C_A^{1/2}\left(\frac{\|p_\tht\|}{\nu}\right)\right]\left(\frac{\|p_\tht\|}{\nu}\right)\nu\left(|v^\tht|^2+|w_\tht|^2\right).\notag
	\end{align}
It then follows from \cref{lem:heat:apriori}, \eqref{est:apriori:L2:lam}, and Gronwall's inequality that
    \begin{align}\notag
	 \sup_{0\leq t\leq T}\left(\|v^\tht(t)\|^2+\|w_\tht(t)\|^2\right)+\int_0^T\left(|Av^\tht(t)|^2+|Aw_\tht(t)|^2\right)dt<\infty,  
	\end{align}
for all $T>0$.
\end{proof}

We end this section by establishing refined estimates under the assumption that $|h(t)|\goesto0$ as $t\goesto\infty$. 

\begin{Lem}\label{lem:dr:apriori:mut:refined}
Suppose that $|h(t)|\goesto0$ as $t\goesto\infty$. If $K>0$ satisfies
    \begin{align}\label{cond:K:lam}
        K\geq 64\sqrt{6}(C_S\vee C_A^{1/2})\ln(e+K)^{1/2}\gr_\tht,
    \end{align}
then
    \begin{align}\label{est:vthtwtht:refined}
        \limsup_{t\goesto\infty}\left(\|v^\tht(t)\|^2+\|w_\tht(t)\|^2\right)\leq 96\nu^2\gr_\tht^2.
    \end{align}
\end{Lem}

\begin{proof}
    
By hypothesis, we may choose $t_0=t_0(K,\tht)$ such that
    \begin{align}\label{est:hyp:lam}
        \sup_{t\geq t_0}\|p(t)\|\leq \nu,\quad \sup_{t\geq t_0}|h(t)|\leq \nu^2\frac{\gr_\tht}{\sqrt{\tht_2\tht_1}},
    \end{align}
and
    \begin{align}\label{est:vlam:wlam:L2:sync}
            \sup_{t\geq t_0}\left\{\frac{9C_L^{2}}{2}+\left[\frac{64C_L^4}{K}\left(\frac{\|p_\tht(t)\|}{\nu}\right)^2+4C_A\frac{|\tht_1-\tht_2|^2}{\tht_1\tht_2}\right]\left(\frac{\|p_\tht(t)\|}{\nu}\right)\right\}\left(\frac{\|p_\tht(t)\|}{\nu}\right)\leq\frac{1}2
    \end{align}
and
    \begin{align}\label{est:vlam:wlam:H1:sync}
        \sup_{t\geq t_0}C_A^{1/2}K^2\left[\frac{|\tht_1-\tht_2|}{\sqrt{\tht_2\tht_1}}K+4C_A^{1/2}\left(\frac{\|p_\tht(t)\|}{\nu}\right)\right]\left(\frac{\|p_\tht(t)\|}{\nu}\right)\leq \frac{1}{64}
    \end{align}
    
Upon returning to \eqref{est:vlam:wlam:L2:ineq} and applying \eqref{est:vlam:wlam:L2:sync}, 
\eqref{est:hyp:lam}, we obtain
    \begin{align}\label{est:sync:L2}
            \frac{d}{dt}&\left(|{v^\tht}|^2+|w_\tht|^2\right)+\frac{\nu}2\left(\|{v^\tht}\|^2+\|w_\tht\|^2\right)\leq 16\nu^3\gr_\tht^2\notag,
    \end{align}
for all $t>t_0$. Thus, there exists $t_1\geq t_0$ sufficiently large such that
    \begin{align}
        \sup_{t\geq t_1}\left(|{v^\tht}(t)|^2+|w_\tht(t)|^2\right)\leq 32\nu^2\gr_\tht^2.
    \end{align}
and
    \begin{align}\notag
        \int_{t_1}^t\left(\|{v^\tht}(s)\|^2+\|w_\tht(s)\|^2\right)ds\leq 64\nu \gr_\tht^2+32\nu^2\gr_\tht^2(t-t_1).
    \end{align}
In particular, there exists a positive measure set of times $\cT$ such that
    \begin{align}\notag
        \cT=\{t\geq t_1:\|{v^\tht}(t)\|^2+\|w_\tht(t)\|^2\leq 96\nu^2\gr_\tht^2\}.
    \end{align}
Fix $t_1'\in \cT$ and define
    \begin{align}\label{def:H1:stop}
        \tau_1:=\inf\{t\geq t_1':\|{v^\tht}(t)\|^2+\|w_\tht(t)\|^2> 96\nu^2\gr_\tht^2\}
    \end{align}
We claim that $\tau_1=\infty$. Suppose to the contrary that $\tau_1<\infty$.

From now on, let $t\in[t_1',\tau_1]$.
First observe that we may alternatively estimate \eqref{est:dwlam:qlam:dvlam}. Indeed, observe that by orthogonality and \eqref{eq:B:skew}, we have
    \begin{align}\notag
    b(\bdy_jw_\tht,q_\tht,\bdy_jv^\tht)=b(\bdy_jp_\tht,q_\tht,\bdy_jv^\tht)-b(\bdy_jq_\tht,\bdy_jv^\tht,q_\tht).
    \end{align}
Then   
\begin{align}
     |b(\bdy_jp_\tht,q_\tht,\bdy_jv^\tht)|&\leq C_S\ln(e+K)^{1/2}|Ap_\tht|\|q_\tht\|\|v^\tht\|\notag
     \\
     &\leq \frac{C_S\ln(e+K)^{1/2}}{K}\|v^\tht\||Ap_\tht||Aq_\tht|\notag
     \\
     &\leq \frac{C_S\ln(e+K)^{1/2}}{K}\|v^\tht\||Aw_\tht|^2\notag
     \\
     &\leq \frac{4\sqrt{6}C_S\ln(e+K)^{1/2}\gr_\tht}{K}\nu|Aw_\tht|^2\notag
     \\
     &\leq \frac{\nu}{16}\left(|Aw_\tht|^2+|Av^\tht|^2\right),\notag
    \end{align}
and
    \begin{align}
     |b(\bdy_jq_\tht,\bdy_jv^\tht,q_\tht)|&\leq C_A^{1/2}\|q_\tht\||Av^\tht||Aq_\tht|^{1/2}|q_\tht|^{1/2}\notag
     \\
     &\leq \frac{C_A^{1/2}}{K}\|q_\tht
     \||Av^\tht||Aq_\tht|\notag
     \\
     &\leq \frac{2\sqrt{6}C_A^{1/2}\gr_\tht}{K}\nu\left(|Aw_\tht|^2+|Av^\tht|^2\right)\notag
     \\
     &\leq \frac{\nu}{16}\left(|Aw_\tht|^2+|Av^\tht|^2\right),\notag
    \end{align}
where we have applied both \eqref{def:H1:stop} and \eqref{cond:K:lam}. Similarly, for \eqref{est:dwlam:vlam:dqlam}, \eqref{est:dvlam:wlam:dqlam}, we have
    \begin{align}
   |b(\bdy_jw_\tht,v^\tht,\bdy_jq_\tht)|&\leq \frac{\nu}{16}\left(|Aw_\tht|^2+|Av^\tht|^2\right)\notag
    \\
    |b(\bdy_jv^\tht,w_\tht,\bdy_jq_\tht)|&\leq \frac{\nu}{16}\left(|Aw_\tht|^2+|Av^\tht|^2\right).\notag
    \end{align}
Treating all other terms in \eqref{eq:vlam:wlam:H1:balance} the same way as \textit{Step 2} of \cref{lem:apriori:sync:lam}, but additionally invoking \eqref{est:hyp:lam}, \eqref{est:sync:L2}, and \eqref{def:H1:stop}, we arrive at
    \begin{align}
		&\frac{d}{dt}\left(\|{v^\tht}\|^2+\|w_\tht\|^2\right)+\nu\left(|A{v^\tht}|^2+|Aw_\tht|^2\right)\notag
            \\
            &\leq2\nu^3\left[\left(\frac{|g^\tht|}{\nu^2}\right)^2+\left(\frac{|h_\tht|}{\nu^2}\right)^2+ 64C_A^{1/2}K^2\left[\frac{|\tht_1-\tht_2|}{\sqrt{\tht_2\tht_1}}K+4C_A^{1/2}\left(\frac{\|p_\tht\|}{\nu}\right)\right]\left(\frac{\|p_\tht\|}{\nu}\right)\gr_\tht^2\right]\notag,
    \end{align}
for all $t\in[t_1',\tau_1]$. From \eqref{est:hyp:lam} and \eqref{est:vlam:wlam:H1:sync}, we therefore deduce
    \begin{align}\notag
		\frac{d}{dt}\left(\|{v^\tht}\|^2+\|w_\tht\|^2\right)+\nu\left(|A{v^\tht}|^2+|Aw_\tht|^2\right)
        \leq 6\nu^3\gr_\tht^2.
    \end{align}
By Gronwall's inequality, we deduce
    \begin{align}\notag
        \sup_{t\in[t_1',\tau_1]}\left(\|{v^\tht}(t)\|^2+\|w_\tht(t)\|^2\right)\leq 6\nu^2\gr_\tht^2<96\nu^2\gr_\tht^2,
    \end{align}
which contradicts the definition of $\tau_1$. We conclude that $\tau_1=\infty$. In particular, we have
    \begin{align}\label{est:sync:H1}
        \sup_{t\geq t_1'}\left(\|{v^\tht}(t)\|^2+\|w_\tht(t)\|^2\right)\leq 96\nu^2\gr_\tht^2,
    \end{align}
for any $t_1'\in\cT$.

\end{proof}

\subsection{Symmetric Direct-Replacement Intertwinement}\label{sect:apriori:dr:sym}
For convenience, let us recall that the symmetric direct-replacement intertwinement is given by 
\begin{align}\label{eq:intertwined:sync:sym}
		\begin{split}
		\bdy_tv_1+\nu Av_1+B(v_1,v_1)&={g_1}+\tht_1P_KB(v_1,v_1)-\tht_2P_KB(v_2,v_2)
		\\
		\bdy_tv_2+\nu Av_2+B(v_2,v_2)&={g_2}-\tht_2P_KB(v_1,v_1)+\tht_1P_KB(v_2,v_2),
		\end{split}
	\end{align}
where $\tht_1+\tht_2=1$.

\subsubsection{Preliminaries} We consider new variables 
    \begin{align}\label{def:zw}
        z=v_1+v_2,\quad w=v_1-v_2,\quad k=g_1+g_2,\quad h=g_1-g_2.
    \end{align}
Observe that
    \begin{align}\notag
        v_1=\frac{z+w}2,\quad v_2=\frac{z-w}2,
    \end{align}
which in particular implies
    \begin{align}
        B(v_1,v_1)+B(v_2,v_2)&=\frac{1}2B(z,z)+\frac{1}2B(w,w)\notag
        \\
        B(v_1,v_1)-B(v_2,v_2)&=\frac{1}2\left(B(z,w)+B(w,z)\right)\notag
    \end{align}
Then \eqref{eq:intertwined:sync:sym} becomes
    \begin{align}\label{eq:dr:zw}
        \begin{split}
        \bdy_tz+\nu Az+\frac{1}2B(z,z)+\frac{1}2B(w,w)&=k+\frac{\tht_1-\tht_2}2\left(P_KB(z,z)+P_KB(w,w)\right)
        \\
        \bdy_tw+\nu Aw+\frac{1}2Q_KDB(z)w&=h
        \end{split}
    \end{align}
Proceeding further, let
    \begin{align}\label{def:zw:project}
        r=P_Kz,\quad s=Q_Kz,\quad p=P_Kw,\quad q=Q_Kw.
    \end{align}
Then $z=r+s$, $w=p+q$, and 
    \begin{align}\label{eq:rspq}
        \begin{split}
        \bdy_tr+\nu Ar+\tht_2P_KB(r,r)+\tht_2P_KDB(s)r&=P_Kk-\tht_2P_KB(s,s)-\tht_2P_KB(p+q,p+q)
        \\
        \bdy_ts+\nu As+\frac{1}2Q_KB(s,s)+\frac{1}2Q_KDB(r)s&=Q_Kk-\frac{1}2Q_KB(r,r)-\frac{1}2Q_KB(p+q,p+q)
        \\
        \bdy_tp+\nu Ap&=P_Kh
        \\
        \bdy_tq+\nu Aq+\frac{1}2Q_KDB(r+s)q&=Q_Kh-\frac{1}2Q_KDB(r+s)p
        \end{split}
    \end{align}
The energy balance becomes
    \begin{align}\label{eq:energy:zw}
        \begin{split}
        \frac{1}2\frac{d}{dt}|z|^2+\nu\|z\|^2&=\lp k,z\rp+\frac{\tht_1-\tht_2}2b(z,s,r)+\frac{\tht_1-\tht_2}2b(w,w,r)-\frac{1}2b(w,w,z)
        \\
        \frac{1}2\frac{d}{dt}|w|^2+\nu\|w\|^2&=\lp h,w\rp-\frac{1}2b(z,p,q)-\frac{1}2b(w,z,q)
        \end{split}
    \end{align}
Subsequently, the total energy balance is given by
    \begin{align}\label{eq:energy:zw:total}
        &\frac{1}2\frac{d}{dt}\left(|z|^2+|w|^2\right)+\nu\|z\|^2+\nu\|w\|^2=\lp k,z\rp+\lp h,w\rp\notag
        \\
        &\quad+\frac{\tht_1-\tht_2}2b(z,s,r)+\frac{\tht_1-\tht_2}2b(w,w,r)-\frac{1}2b(w,p,z)-\frac{1}2b(z,p,q)
    \end{align}

Lastly, by \cref{lem:heat:apriori}, observe that the low modes, $p$, of $w$ always satisfies
    \begin{align}\label{est:heat:basic}
        \frac{d}{dt}\|p\|^2+\nu|Ap|^2\leq \nu^3\left(\frac{\sup_{t\geq0}|h(t)|}{\nu^2}\right)^2,
    \end{align}
and for any $t_0\geq0$
    \begin{align}\notag
        \sup_{t\geq t_0}\|p(t)\|^2\leq \|p(t_0)\|^2+\nu^2\left(\frac{\sup_{t\geq t_0}|h(t)|}{\nu^2}\right)^2.
    \end{align}
For convenience, we set
    \begin{align}\label{def:Cp:Ch}
        \begin{split}
        \Ch(t_0)&:=\frac{\sup_{t\geq t_0}|h(t)|}{\nu^2},\qquad \qquad \hspace{1pt}\Ch:=\Ch(0),
        \\
        \quad\Cp(t_0)^2&:=\left(\frac{\|p(t_0)\|}{\nu}\right)^2+\Ch(t_0)^2,\quad \Cp:=\Cp(0).
        \end{split}
    \end{align}
Then \eqref{est:heat:basic} implies
    \begin{align}\label{est:heat:p}
        \sup_{t\geq t_0}\|p(t)\|^2\leq \Cp(t_0)^2\nu^2,
    \end{align}
and
    \begin{align}\label{est:heat:Ap}
            \nu\int_{t_0}^t|Ap(s)|^2ds\leq \Cp(t_0)^2\nu^2+\Ch(t_0)^2(t-t_0)\nu^3,
    \end{align}
for all $0\leq t_0\leq t$. It will also be useful to define
    \begin{align}\label{est:const:h:k}
        \Ck(t_0):=\frac{\sup_{t\geq t_0}|k(t)|}{\nu^2},\quad \Ck:=\Ck(0).
    \end{align}
Observe that 
    \begin{align}\label{est:Ck}
        \Ck\leq\sqrt{2}\gr.
    \end{align}

We will distinguish two special cases: one when $\tht_1=1$, $\tht_2=0$, which yields
    \begin{align}\label{eq:dr:zw:1}
        \begin{split}
         \bdy_tz+\nu Az+\frac{1}2Q_KB(z,z)&=k-\frac{1}2Q_KB(w,w)
        \\
        \bdy_tw+\nu Aw+\frac{1}2Q_KDB(z)w&=h,
        \end{split}
    \end{align}
and the other when $\tht_1=\tht_2$, which yields
    \begin{align}\label{eq:dr:zw:2}
        \begin{split}
        \bdy_tz+\nu Az+\frac{1}2B(z,z)&=k-\frac{1}2B(w,w)
        \\
        \bdy_tw+\nu Aw+\frac{1}2Q_KDB(z)w&=h.
        \end{split}
    \end{align}
Ultimately, we will treat these special cases separately, then develop global bounds for small perturbations from these cases.

\subsubsection{Local-in-time Apriori estimates} First, we will establish estimates that ensure a local existence theory that is unconditional. For this and for the remainder of the section, it will be convenient to have the following estimate in hand:

\begin{Lem}\label{lem:estimate:force:H1}
For any $k,h\in H$,and $s,q\in D(A)$, we have
    \begin{align}\notag
        |\lp k,As\rp|+|\lp h,Aq\rp|&\leq \frac{1}{2\de\nu}\left(|k|^2+|h|^2\right)+\frac{\de\nu}{2}\left(|As|^2+|Aq|^2\right),
    \end{align}
for any $\de>0$ and $\nu>0$.
\end{Lem}

\begin{Lem}\label{lem:dr:sym:apriori}
Let $v_1^0,v_2^0\in V$, $g_1,g_2\in L^\infty(0,\infty;H)$, and $\tht_1+\tht_2=1$. There exists $T_*$ depending only on $\|v_1^0\|,\|v_2^0\|$ and $\gr$ such that
    \begin{align}\label{est:local:main1}
        &\sup_{0\leq t\leq T_*/2}\left(\frac{\|z(t)\|^2+\|w(t)\|^2}{\nu^2}\right)\notag
        \\
        &\leq \frac{1+\frac{\|z_0\|^2+\|w_0\|^2}{\nu^2}}{1-\left(4\sup_{t\geq0}\frac{|k(t)|^2+|h(t)|^2}{\nu^4}+8C_L^4+9C_A^4\right) \left(1+\frac{\|z_0\|^2+\|w_0\|^2}{\nu^2}\right)(\nu T_*)}.
    \end{align}
Moreover
    \begin{align}\label{est:local:main2}
        \nu\int_0^{T_*/2}\left(|Az(s)|^2+|Aw(s)|^2\right)ds<\infty.
    \end{align}
\end{Lem}

These bounds can be used to develop time-derivative estimates which then eventually imply local well-posedness. However, we do not provide additional details for this. We refer the reader to the discussion at the end of \cref{sect:gwp}.

\begin{Thm}\label{thm:dr:sym:loc:exist}
Let $v_1^0,v_2^0\in V$, $g_1,g_2\in L^\infty(0,\infty;H)$, and $\tht_1+\tht_2=1$. Then there exists $T_*$ such that $v\in (C([0,T);H)\cap L^2(0,T;D(A)))^2$, for all $T<T_*$. Moreover, $\lim_{t\goesto T_*}\|v(t)\|=\infty$, whenever $T_*<\infty$.
\end{Thm}

Note that $T_*$ in \cref{thm:dr:sym:loc:exist} is referred to as the \textit{maximal time of existence}. From \cref{lem:dr:sym:apriori}, we also deduce that the following stopping time is well-defined.

\begin{Cor}\label{cor:stop:time}
Under the hypothesis of \cref{lem:dr:sym:apriori}, let $T_*$ denote the maximal time of local existence. Then the following stopping time is well-defined for any given $\Cm>0$:
    \begin{align}\label{def:stop:time}
        t_\Cm:=\inf\left\{t\leq T_*/2: \|z(t)\|^2+\|w(t)\|^2\leq \Cm^2\nu^2\right\}.
    \end{align}
\end{Cor}

Let us now prove \cref{lem:dr:sym:apriori}.

\begin{proof}[Proof of \cref{lem:dr:sym:apriori}]
Let us consider the system written in the form \eqref{eq:dr:zw} and observe that since $\tht_1+\tht_2=1$, then
    \begin{align}\label{eq:zw:loc:exist}
        \begin{split}
        \bdy_tz+\nu Az&=k-\tht_2\left(Q_KB(z,z)+Q_KB(w,w)\right)
        \\
        \bdy_tw+\nu Aw&=h-\frac{1}2Q_KDB(z)w.
        \end{split}
    \end{align}
Then, upon taking the $H$-inner product  with $Az, Aw$, respectively, and summing the result, we obtain
    \begin{align}\label{eq:enstrophy:balance:zw:loc:exist}
        &\frac{1}2\frac{d}{dt}\left(\|z\|^2+\|w\|^2\right)+\nu\left(|Az|^2+|Aw|^2\right)\notag
        \\
        &=\lp k,Az\rp +\lp h,Aw\rp-\tht_2b(z,z,As)-\tht_2b(w,w,As)-\frac{1}2b(z,w,Aq)-\frac{1}2b(w,z,Aq).
    \end{align}
The first two trilinear terms on the right-hand side of \eqref{eq:enstrophy:balance:zw:loc:exist} are then estimated using H\"older's inequality, \eqref{est:interpolation}, and Young's inequality to obtain
    \begin{align}
        |\tht_2||b(z,z,As)|&\leq C_A|\tht_2||Az|^{1/2}|z|^{1/2}\|z\||As|\leq \tht_2^2\frac{C_A^2}{\nu}|Az||z|\|z\|^2+\frac{\nu}{8}|As|^2\notag
        \\
        &\leq \tht_2^4\frac{C_A^4}{\nu^3}|z|^2\|z\|^4+\frac{\nu}{8}\left(|Az|^2+|As|^2\right)\notag
        \\
        |\tht_2||b(w,w,As)|&\leq C_A|\tht_2||Aw|^{1/2}|w|^{1/2}\|w\||As|\leq \tht_2^2\frac{C_A^2}{\nu}|Aw||w|\|w\|^2+\frac{\nu}{8}|As|^2\notag
        \\
        &\leq \tht_2^4\frac{C_A^4}{\nu^3}|w|^2\|w\|^4+\frac{\nu}{8}\left(|Aw|^2+|As|^2\right)\notag.
    \end{align}
In particular
    \begin{align}\label{est:loc:exist:trilinear1}
        &|\tht_2||b(z,z,As)|+\tht_2|b(w,w,As)|\notag
        \\
        &\leq \frac{\tht_2^4C_A^4}{\nu^3}\left(\|z\|^2+\|w\|^2\right)^3+\frac{3\nu}{8}|Az|^2+\frac{\nu}{8}|Aw|^2.
    \end{align}

We next treat the terms $b(z,w,Aq)$ and $b(w,z,Aq)$ using H\"older's inequality, \eqref{est:interpolation}, and Young's inequality to obtain
    \begin{align}
        \frac{1}2|b(z,w,Aq)|&\leq \frac{C_L}2\|z\|^{1/2}|z|^{1/2}|Aw|^{1/2}\|w\|^{1/2}|Aq|\leq \frac{2C_L^2}{\nu}\|z\||z||Aw|\|w\|+\frac{\nu}{32}|Aq|^2\notag
        \\
        &\leq \frac{8C_L^4}{\nu^3}\|z\|^2|z|^2\|w\|^2+\frac{\nu}{32}\left(|Aw|^2+|Aq|^2\right)\notag
        \\
        \frac{1}2|b(w,z,Aq)|&\leq \frac{C_A}2|Aw|^{1/2}|w|^{1/2}\|z\||Aq|\leq \frac{2C_A^2}{\nu}|Aw||w|\|z\|^2+\frac{\nu}{32}|Aq|^2\notag
        \\
        &\leq \frac{8C_A^4}{\nu^3}|w|^2\|z\|^4+\frac{\nu}{32}\left(|Aw|^2+|Aq|^2\right)\notag.
    \end{align}
In particular, we have
    \begin{align}\label{est:loc:exist:trilinear2}
    &\frac{1}2|b(z,w,Aq)|+\frac{1}2|b(w,z,Aq)|\leq \frac{8(C_L^4+C_A^4)}{\nu^3}\left(\|z\|^2+\|w\|^2\right)^3+\frac{\nu}{8}|Aw|^2
    \end{align}

Upon invoking  \cref{lem:estimate:force:H1} $\de=1/4$ to estimate the first two terms on the right-hand side of \eqref{eq:enstrophy:balance:zw:loc:exist}, and combining the result with \eqref{est:loc:exist:trilinear1}, \eqref{est:loc:exist:trilinear2}, we arrive at
    \begin{align}\label{est:local:balance}
        \frac{d}{dt}\left(\|z\|^2+\|w\|^2\right)+\nu\left(|Az|^2+|Aw|^2\right)&\leq  \frac{4}{\nu}\left(|k|^2+|h|^2\right)\notag
        \\
        &\quad+\frac{8C_L^4+9C_A^4}{\nu^3}\left(\|z\|^2+\|w\|^2\right)^3.
    \end{align}
An application of Gr\"onwall's inequality yields \eqref{est:local:main1}. On the other hand, integrating \eqref{est:local:balance} over $[0,T^*/2]$, then applying \eqref{est:local:main1} yields \eqref{est:local:main2}.
\end{proof}

\subsubsection{Global-in-time Apriori estimates:  $\tht_1=1,\tht_2=0$}\label{sect:apriori:tht1}

From \eqref{eq:rspq}, we see that $r$ and $p$ satisfy heat equations. Thus, the apriori bounds asserted by \cref{lem:heat:apriori} apply to $r$ and $p$. One may then use these apriori bounds to establish global well-posedness of the associated initial value problem in a standard fashion. In fact, in this particular case, global well-posedness was already established in \cite[Theorem 3.1]{OlsonTiti2003}. However, we present a different analysis here for this particular case, which ultimately proceeds from a more unified approach that allows us to treat the cases of mutual and symmetric intertwinements as well.

\begin{Lem}\label{lem:apriori:dr:tht1}
Suppose that $\tht_1=1$. Given $g_1,g_2\in L^\infty(0,\infty;H)$, $v_0^1,v_0^2\in V$, we have that
    \begin{align}\notag
        \sup_{0\leq t\leq T}\left(\|z(t)\|^2+\|w(t)\|^2\right)+\nu\int_0^T\left(|Az(s)|^2+|Aw(s)|^2\right)ds<\infty,
    \end{align}
holds for all $T>0$.
\end{Lem}

\begin{proof}
Upon taking the $H$-inner product of \eqref{eq:dr:zw:1} with $Az$ and $Aw$, respectively, invoking \eqref{eq:B:enstrophy}, then summing the result, we obtain
    \begin{align}\label{eq:enstrophy:zw}
        &\frac{1}2\frac{d}{dt}\left(\|z\|^2+\|w\|^2\right)+\nu|Az|^2+\nu|Aw|^2\notag
        \\
        &=(k,Az)+(h,Aw)+\frac{1}2b(z,z,Ar)-\frac{1}2b(w,w,As)-\frac{1}2b(w,z,Aq)-\frac{1}2b(z,w,Aq)
    \end{align}
To treat the trilinear terms above, first observe that
    \begin{align}\notag
        \begin{split}
        b(w,z,Aq)&=b(w,r,Aw)+b(w,s,Aw)-b(w,z,Ap)
        \\
        b(z,w,Aq)&=b(r,w,Aw)+b(s,w,Aw)-b(z,w,Ap)
        \end{split}
    \end{align}
Then by \eqref{eq:B:enstrophy:miracle}, we have
    \begin{align}\notag
        b(w,w,As)+b(w,z,Aq)+b(z,w,Aq)=b(w,r,Aw)-b(w,z,Ap)-b(z,w,Ap).
    \end{align}
Furthermore, by \eqref{eq:B:enstrophy}, we see that
    \begin{align}\notag
        b(z,z,Ar)=b(r,s,Ar)+b(s,r,Ar)+b(s,s,Ar)
    \end{align}
Upon returning to \eqref{eq:enstrophy:zw} and invoking these identities, we see that the final four terms are equivalent to the following six terms
    \begin{align}\notag
        \frac{1}2b(r,s,Ar)+\frac{1}2b(s,r,Ar)+\frac{1}2b(s,s,Ar)-\frac{1}2b(w,r,Aw)-\frac{1}2b(w,z,Ap)-\frac{1}2b(z,w,Ap).
    \end{align}
We now estimate each of these terms
    \begin{align}
        \frac{1}2|b(r,s,Ar)|&\leq \frac{1}2C_S\ln(e+K)^{1/2}\|r\|\|s\||Ar|\notag
        \\
        &\leq \frac{C_S^2\ln(e+K)}{\nu}\|r\|^2\|s\|^2+\frac{\nu}{16}|Az|^2\notag
        \\
        \frac{1}2|b(s,r,Ar)|&\leq \frac{1}2C_S\ln(e+K)^{1/2}|s|\|r\|\|Ar\|\leq C_S\ln(e+K)^{1/2}\|s\|\|r\||Ar|\notag
        \\
        &\leq \frac{C_S^2\ln(e+K)}{\nu}\|r\|^2\|s\|^2+\frac{\nu}{16}|Az|^2\notag
        \\
        \frac{1}2|b(s,s,Ar)|&\leq \frac{1}2C_S\ln(e+K)^{1/2}\|s\|^2|Ar|\leq \frac{1}2C_S\ln(e+K)^{1/2}\|s\||As|\|r\|\notag
        \\
        &\leq \frac{C_S^2\ln(e+K)}{\nu}\|r\|^2\|s\|^2+\frac{\nu}{16}|Az|^2.\notag
    \end{align}
On the other hand, we have
    \begin{align}
        \frac{1}2|b(w,r,Aw)|&\leq \frac{1}2|b(p,r,Aw)|+\frac{1}2|b(q,r,Aw)|\notag
        \\
        &\leq \frac{1}2C_S\ln(e+K)^{1/2}\|p\|\|r\||Aw|+\frac{1}2C_S\ln(e+K)^{1/2}\|q\|\|r\||Aw|\notag
        \\
        &\leq \frac{2C_S^2\ln(e+K)}{\nu}\|r\|^2\|w\|^2+\frac{\nu}{16}|Aw|^2\notag
        \\
        \frac{1}2|b(w,z,Ap)|&\leq \frac{1}2|b(p,z,Ap)|+\frac{1}2|b(q,z,Ap)|\notag
        \\
        &\leq \frac{1}2C_S\ln(e+K)^{1/2}\|p\|\|z\||Ap|+\frac{1}2C_S\ln(e+K)^{1/2}|Aq|\|z\|\|p\|\notag
        \\
        &\leq \frac{4C_S^2\ln(e+K)}{\nu}\|p\|^2\|z\|^2+\frac{\nu}{16}|Aw|^2\notag
        \\
        \frac{1}2|b(z,w,Ap)|&\leq \frac{1}2|b(r,w,Ap)|+\frac{1}2|b(s,w,Ap)|\notag
        \\
        &\leq \frac{1}2C_S\ln(e+K)^{1/2}\|r\|\|w\||Ap|+\frac{1}2C_S\ln(e+K)^{1/2}|As|\|w\|\|p\|\notag
        \\
        &\leq \frac{2C_S^2\ln(e+K)}{\nu}\left(\|r\|^2+\|p\|^2\right)\|w\|^2+\frac{\nu}{32}|Az|^2+\frac{\nu}{32}|Aw|^2.\notag
    \end{align}

Finally, upon combining the above estimates for the trilinear terms in conjunction with \cref{lem:estimate:force:H1} (applied with $\de=1/8$) to treat the first two terms on the right-hand side of \eqref{eq:enstrophy:zw}, we deduce
    \begin{align}
        &\frac{d}{dt}\left(\|z\|^2+\|w\|^2\right)+\nu\left(|Az|^2+|Aw|^2\right)-\frac{4}{\nu}\left(|k|^2+|h|^2\right)\notag
        \\
        &\leq \frac{4C_S^2\ln(e+K)}{\nu}\left(\|s\|^2+\|q\|^2\right)\|r\|^2+\frac{4C_S^2\ln(e+K)}{\nu}\left(\|z\|^2+\|w\|^2\right)\|p\|^2\label{est:apriori:tht1:final1}
        \\
        &\leq \frac{8C_S^2\ln(e+K)}{\nu}\left(\|r\|^2+\|p\|^2\right)\left(\|z\|^2+\|w\|^2\right)\label{est:apriori:tht1:final2}
    \end{align}
An application of \cref{lem:heat:apriori} and Gr\"onwall's inequality yields the desired estimate.
\end{proof}

We may also establish the following refined estimates under the assumption that $|h(t)|\goesto0$ as $t\goesto\infty$.

\begin{Lem}\label{lem:dr:apriori:tht1:refined}
Suppose that $|h(t)|\goesto0$ as $t\goesto\infty$. If $K>0$ satisfies
    \begin{align}\label{cond:K:tht1:refined}    
    K^2\geq 32C_S^2\ln(e+K)\Ck^2
    \end{align}
then
    \begin{align}\notag
        \limsup_{t\goesto\infty}\|v(t)\|\leq 4\Ck\nu.
    \end{align}
\end{Lem}

\begin{proof}
From \cref{lem:heat:apriori}, $t_0$ may be chosen so that
    \begin{align}\notag
        \sup_{t\geq t_0}\|r(t)\|\leq \sqrt{2}\nu\Ck.
    \end{align}
Since $|h(t)|\goesto0$ as $t\goesto\infty$, it follows from  \cref{cor:apriori:heat} that there exists $t_0$ such that
    \begin{align}\notag
            \sup_{t\geq t_0}|h(t)|\leq\frac{\sqrt{2}}2\Ck\nu^2,\quad \sup_{t\geq t_0}\|p(t)\|\leq \frac{1}{4C_S\ln(e+K)^{1/2}}\nu.
    \end{align}
On the other hand, observe that by orthogonality and \eqref{est:Bernstein}, we have
    \begin{align}
        &\frac{4C_S^2\ln(e+K)}{\nu}\left(\|s\|^2+\|q\|^2\right)\|r\|^2+\frac{4C_S^2\ln(e+K)}{\nu}\left(\|z\|^2+\|w\|^2\right)\|p\|^2\notag
        \\
        &\leq\frac{8C_S^2\ln(e+K)\Ck^2}{ K^2}\left(|Az|^2+|Aw|^2\right)+\frac{4C_S^2\ln(e+K)}{\nu}\left(\|z\|^2+\|w\|^2\right)\|p\|^2\notag
    \end{align}
Returning to \eqref{est:apriori:tht1:final1}, it then follows that
    \begin{align}
        \frac{d}{dt}\left(\|z\|^2+\|w\|^2\right)+\frac{\nu}2\left(|Az|^2+|Aw|^2\right)\leq 8\Ck^2\nu^3\notag.
    \end{align} 
An application of Gr\"onwall's inequality yields
    \begin{align}\notag
        \|z(t)\|^2+\|w(t)\|^2\leq e^{-(\nu/2) (t-t_0)}\left(\|z(t_0)\|^2+\|w(t_0)\|^2\right)+16\Ck^2\nu^2.
    \end{align}
For $t_0'\geq t_0$ sufficiently large, it follows that
    \begin{align}\notag
        \sup_{t\geq t_0'}\left( \|z(t)\|^2+\|w(t)\|^2\right)\leq 32\Ck^2\nu^2,
    \end{align}
as claimed.
\end{proof}

We now observe that \cref{thm:uniform:tht1} follows immediately from \cref{lem:dr:apriori:tht1:refined} and \eqref{est:Ck}.

\subsubsection{Global-Time Apriori estimates: $\tht_1=\tht_2=1/2$}\label{sect:apriori:tht12}

We carry out a similar analysis but for the system \eqref{eq:dr:zw:2} (instead of \eqref{eq:dr:zw:1}) to prove the following.

\begin{Lem}\label{lem:apriori:dr:tht12}
Suppose that $\tht_1=\tht_2=1/2$. Given $g_1,g_2\in L^\infty(0,\infty;H)$, $v_0^1,v_0^2\in V$, we have that
    \begin{align}\notag
        \sup_{0\leq t\leq T}\left(\|z(t)\|^2+\|w(t)\|^2\right)+\nu\int_0^T\left(|Az(s)|^2+|Aw(s)|^2\right)ds<\infty,
    \end{align}
holds for all $T>0$.
\end{Lem}

\begin{proof}
Upon taking the $H$-inner product of \eqref{eq:dr:zw:2} with $Az$ and $Aw$, respectively, invoking \eqref{eq:B:enstrophy}, then summing the result, we obtain
    \begin{align}\label{eq:enstrophy:zw:2}
    &\frac{1}2\frac{d}{dt}\left(\|z\|^2+\|w\|^2\right)+\nu|Az|^2+\nu|Aw|^2\notag
        \\
        &=(k,Az)+(h,Aw)-\frac{1}2b(w,w,Az)-\frac{1}2b(w,z,Aq)-\frac{1}2b(z,w,Aq)
    \end{align}

To treat the last three terms, first observe that \eqref{eq:B:enstrophy:miracle} implies
    \begin{align}\notag
        b(q,q,Az)+b(q,z,Aq)+b(z,q,Aq)=0.
    \end{align}
In particular, the last three terms are equivalent to the following five terms
    \begin{align}\notag
        -\frac{1}2b(p,p,Az)-\frac{1}2b(p,q,Az)-\frac{1}2b(q,p,Az)
        -\frac{1}2b(p,z,Aq)-\frac{1}2b(z,p,Aq)
    \end{align}
We now estimate each of these terms as follows:
    \begin{align}
        \frac{1}2|b(p,p,Az)|&\leq \frac{1}2C_S\ln(e+K)^{1/2}\|p\|^2|Az|\notag
        \\
        &\leq \frac{C_S^2\ln(e+K)}{\nu}\|p\|^2\|w\|^2+\frac{\nu}{16}|Az|^2\notag
        \\
        \frac{1}2|b(p,q,Az)|&\leq\frac{1}2C_S\ln(e+K)^{1/2}\|p\|\|q\||Az|\notag
        \\
        &\leq \frac{C_S^2\ln(e+K)}{\nu}\|p\|^2\|w\|^2+\frac{\nu}{16}|Az|^2\notag
        \\
        \frac{1}2|b(q,p,Az)|&\leq \frac{1}2C_S\ln(e+K)^{1/2}|q||Ap||Az|\notag
        \\
        &\leq \frac{C_S^2\ln(e+K)}{\nu}\|p\|^2\|w\|^2+\frac{\nu}{16}|Az|^2\notag
        \\
        \frac{1}2|b(p,z,Aq)|&\leq \frac{1}2C_S\ln(e+K)^{1/2}\|p\|\|z\||Aq|\notag
        \\
        &\leq \frac{C_S^2\ln(e+K)}{\nu}\|p\|^2\|z\|^2+\frac{\nu}{16}|Aw|^2\notag.
    \end{align}
Lastly, we estimate the fifth and final term as
    \begin{align}
        \frac{1}2|b(z,p,Aq)|&\leq \frac{1}2|b(r,p,Aq)|+\frac{1}2|b(s,p,Aq)|\notag
        \\
        &\leq\frac{1}2C_S\ln(e+K)^{1/2}\|r\|\|p\||Aq|+\frac{1}2C_S\ln(e+K)^{1/2}\|s\|\|p\||Aq|\notag
        \\
        &\leq \frac{2C_S^2\ln(e+K)}{\nu}\|p\|^2\|z\|^2+\frac{\nu}{16}|Aw|^2\notag
    \end{align}
Applying these estimates in \eqref{eq:enstrophy:zw:2} and \cref{lem:estimate:force:H1} (applied with $\de=1/8$), we arrive at
    \begin{align}\label{est:dr:sym:tht12:global}
        &\frac{d}{dt}\left(\|z\|^2+\|w\|^2\right)+\nu\left(|Az|^2+|Aw|^2\right)\notag
        \\
        &\leq \frac{4}{\nu}\left(|k|^2+|h|^2\right)+\frac{6C_S^2\ln(e+K)}{\nu}\|p\|^2\left(\|z\|^2+\|w\|^2\right).
    \end{align} 
An application of \cref{lem:heat:apriori} and Gr\"onwall's inequality yields the desired estimate.

\end{proof}

We may also establish the following refined estimates under the assumption that $|h(t)|\goesto0$ as $t\goesto\infty$.

\begin{Lem}\label{lem:apriori:tht12:refined}
Suppose that $|h(t)|\goesto0$ as $t\goesto\infty$. Then
    \begin{align}\notag
        \limsup_{t\goesto\infty}\|v(t)\|\leq 4\Ck\nu.
    \end{align}
\end{Lem}

\begin{proof}
Since $|h(t)|\goesto0$ as $t\goesto\infty$, it follows from \cref{cor:apriori:heat} that there exists $t_0$ such that
    \begin{align}\notag
            \sup_{t\geq t_0}|h(t)|\leq\Ck\nu^2,\quad \sup_{t\geq t_0}\|p(t)\|\leq \frac{\sqrt{6}}{12C_S\ln(e+K)^{1/2}}\nu
    \end{align}
Returning to \eqref{est:dr:sym:tht12:global}, it then follows that
    \begin{align}
        \frac{d}{dt}\left(\|z\|^2+\|w\|^2\right)+\frac{\nu}2\left(|Az|^2+|Aw|^2\right)\leq 8\Ck^2\nu^3\notag.
    \end{align} 
By Gr\"onwall's inequality, we have
    \begin{align}\notag
        \|z(t)\|^2+\|w(t)\|^2\leq e^{-(\nu/2)(t-t_0)}\left(\|z_0\|^2+\|w_0\|^2\right)+16\Ck^2\nu^2,
    \end{align}
For $t_0'\geq t_0$ sufficiently large, we therefore deduce
    \begin{align}\notag
        \sup_{t\geq t_0'}\left(\|z(t)\|^2+\|w(t)\|^2\right)\leq 32\Ck^2\nu^2.
    \end{align}
\end{proof}

We then see that \cref{thm:uniform:tht12} follows immediately from \cref{lem:apriori:tht12:refined} and \eqref{est:Ck}.

\subsubsection{Global-Time Apriori estimates: $|\tht_2|\ll1$}\label{sect:apriori:tht2}

Next, we establish estimates and identify conditions that guarantee global existence. To state the result, it will be helpful to introduce the following quantities:
    \begin{align}\label{def:Cd:Cf}
        \Cd^2:=1+\frac{\|z_0\|^2+\|w_0\|^2}{\nu^2},\quad \Cf^2:=\left(\frac{\sup_{t\geq0}|k(t)|}{\nu^2}\right)^2+\left(\frac{\sup_{t\geq0}|h(t)|}{\nu^2}\right)^2.
    \end{align}
Also
    \begin{align}\label{def:Cr}
        \Cr^2:=16\left[\left(\frac{\|r_0\|}{\nu}\right)^2+\left(\frac{\sup_{t\geq0}|k(t)|}{\nu^2}\right)^2\right].
    \end{align}
    
\begin{Lem}\label{lem:apriori:tht2}
Let $v_1^0,v_2^0\in V$, $g_1,g_2\in L^\infty(0,\infty;H)$, and $\tht_1+\tht_2=1$. Let $\Cm$ be defined by
    \begin{align}\label{def:M:tht2}
        \Cm^2:=16\left(1+C_S^2\ln(e+K)(1+\Cp^2)\right)\left(\Cd^2+\Cf^2\right)    
    \end{align}
Suppose that $K$ satisfies
    \begin{align}\label{cond:K:tht2}
        K\geq e^{1/(8C_S^2)},\quad \frac{C_A\Cr}{K}+\frac{16C_S^2\ln(e+K)\left(\Cp^2+\Cr^2\right)}{K^2}\leq 2.
    \end{align}
Further suppose that $\tht_2$ satisfies
    \begin{align}\label{cond:tht2}
        {\tht_2^2C_S^2\ln(e+K)}\Cm^4\leq2\left(\frac{\|r_0\|}{\nu}\right)^2+\left(\frac{\sup_{t\geq0}|k(t)|}{\nu^2}\right)^2.
    \end{align}
Then
    \begin{align}    \notag   \sup_{t\geq0}\left(\|z(t)\|^2+\|w(t)\|^2\right)\leq \Cm^2\nu^2,
    \end{align}
and
    \begin{align}\notag
        \nu\int_0^t\left(|Az(s)|^2+|Aw(s)|^2\right)ds\leq \Cm^2\nu^2t
    \end{align}
for all $t\geq0$.
\end{Lem}

\begin{proof}
For this, we consider the system written in the form \eqref{eq:rspq}. Upon taking the $H$-inner product of the $r$-equation with $Ar$ and invoking \eqref{eq:B:enstrophy}, we obtain
    \begin{align}\label{eq:enstrophy:balance:r:tht2}
    &\frac{1}2\frac{d}{dt}\|r\|^2+\nu|Ar|^2\notag
    \\
    &=\lp k,Ar\rp-\tht_2b(s,s,Ar)-\tht_2b(p+q,p+q,Ar)-\tht_2b(s,r,Ar)-\tht_2b(r,s,Ar)
    \end{align}
Let us treat the first, third, and fourth trilinear terms on the right-hand side of \eqref{eq:enstrophy:balance:r:tht2}. Indeed, we apply H\"older's inequality, \eqref{est:Sobolev}, and Young's inequality to estimate
    \begin{align}
        |\tht_2||b(s,s,Ar)|&\leq |\tht_2|C_S\ln(e+K)^{1/2}|s|\|s\|\|Ar\|\leq |\tht_2|C_S\ln(e+K)^{1/2}\|s\|^2|Ar|\notag
        \\
        &\leq \frac{4\tht_2^2C_S^2\ln(e+K)}{\nu}\|s\|^4+\frac{\nu}{16}|Ar|^2\notag
        \\
        |\tht_2||b(s,r,Ar)|&\leq |\tht_2|C_S\ln(e+K)|s|\|r\|\|Ar\|\leq |\tht_2|C_S\ln(e+K)^{1/2}\|s\|\|r\||Ar|\notag
        \\
        &\leq \frac{4\tht_2^2C_S^2\ln(e+K)}{\nu}\|s\|^2\|r\|^2+\frac{\nu}{16}|Ar|^2\notag
        \\
        |\tht_2||b(r,s,Ar)|&\leq |\tht_2|C_S\ln(e+K)^{1/2}\|r\|\|s\||Ar|\notag
        \\
        &\leq \frac{4\tht_2^2C_S^2\ln(e+K)}{\nu}\|s\|^2\|r\|^2+\frac{\nu}{16}|Ar|^2\notag.   
    \end{align}
Now for the second trilinear term, observe that
    \begin{align}\notag
        b(p+q,p+q,Ar)=b(p,p,Ar)+b(p,q,Ar)+b(q,p,Ar)+b(q,q,Ar).
    \end{align}
Then
    \begin{align}
        |\tht_2||b(p,p,Ar)|&\leq |\tht_2|C_S\ln(e+K)^{1/2}\|p\|^2|Ar|\notag
        \\
        &\leq \frac{4\tht_2^2C_S^2\ln(e+K)}{\nu}\|p\|^4+\frac{\nu}{16}|Ar|^2\notag
        \\
        |\tht_2||b(p,q,Ar)|&\leq |\tht_2|C_S\ln(e+K)^{1/2}\|p\|\|q\||Ar|\notag
        \\
        &\leq \frac{4\tht_2^2C_S^2\ln(e+K)}{\nu}\|p\|^2\|q\|^2+\frac{\nu}{16}|Ar|^2\notag
        \\
        |\tht_2||b(q,p,Ar)|&\leq |\tht_2|C_S\ln(e+K)^{1/2}|q||Ap||Ar|\leq |\tht_2|C_S\ln(e+K)^{1/2}\|q\|\|p\||Ar|\notag
        \\
        &\leq \frac{4\tht_2^2C_S^2\ln(e+K)}{\nu}\|p\|^2\|q\|^2+\frac{\nu}{16}|Ar|^2\notag
        \\
        |\tht_2||b(q,q,Ar)|&\leq |\tht_2|C_S\ln(e+K)^{1/2}|q|\|q\|\|Ar\|\leq |\tht_2|C_S\ln(e+K)^{1/2}\|q\|^2|Ar|\notag
        \\
        &\leq \frac{4\tht_2^2C_S^2\ln(e+K)}{\nu}\|q\|^4+\frac{\nu}{16}|Ar|^2.\notag
    \end{align}
Therefore, upon combining the estimates for trilinears and applying \cref{lem:estimate:force:H1} with $\de=1/8$ and $h\equiv0$ to estimate the first term on the right-hand side of \eqref{eq:enstrophy:balance:r:tht2}, we deduce
    \begin{align}\notag
        &\frac{d}{dt}\|r\|^2+\nu|Ar|^2-\frac{8}{\nu}|k|^2\notag
        \\
        &\leq\frac{8\tht_2^2C_S^2\ln(e+K)}{\nu}\left(\|z\|^2\|r\|^2+\|w\|^2\|p\|^2\right)+\frac{4\tht_2^2C_S^2\ln(e+K)}{\nu}\left(\|s\|^4+\|q\|^4\right) \notag
        \\
        &\leq\frac{16\tht_2^2C_S^2\ln(e+K)}{\nu}\left(\|z\|^2+\|w\|^2\right)^2.\notag
    \end{align}
Recall that by \cref{cor:stop:time}, the stopping time, $t_\Cm$, defined by \eqref{def:stop:time} is well-defined for any $\Cm>0$. Thus, for $t\leq t_\Cm$, we have
    \begin{align}\label{est:r:tht2}
        &\frac{d}{dt}\|r\|^2+\nu\|r\|^2\leq\frac{d}{dt}\|r\|^2+\nu|Ar|^2\leq \frac{8}{\nu}|k|^2 +{8\tht_2^2C_S^2\ln(e+K)}\Cm^4\nu^3.
    \end{align}
By Gr\"onwall's inequality, it follows that
    \begin{align}
        \|r(t)\|^2\leq e^{-\nu t}\|r_0\|^2+8\left(\left(\frac{\sup_{t\geq0}|k(t)|}{\nu^2}\right)^2+{\tht_2^2C_S^2\ln(e+K)}\Cm^4\right)\nu^2(1-e^{-\nu t}),\notag
    \end{align}
for all $0\leq t\leq t_\Cm$. Recall that we choose $\tht_2$ to satisfy \eqref{cond:tht2}. Thus, upon recalling \eqref{def:Cr}, we see that for now, we have guaranteed that
    \begin{align}\label{est:tht2:r:apriori:final2}
        \sup_{0\leq t\leq t_\Cm}\|r(t)\|^2\leq \Cr^2\nu^2,
    \end{align}
holds for all $\Cm>0$. Further observe that $\Cr$ depends only on $|k|, \|r_0\|$.

Next, we estimate the high-modes. To this end, we take the $H$-inner product of the $s$-equation with $As$ and the $q$-equation with $q$, invoke \eqref{eq:B:enstrophy:miracle}, then sum the result to write
    \begin{align}
        &\frac{1}2\frac{d}{dt}\left(\|s\|^2+\|q\|^2\right)+\nu|As|^2+\nu|Aq|^2=\lp k, As\rp+\lp h,Aq\rp
        +\frac{1}2b(s,s,Ar)
        -\frac{1}2b(r,r,As)\notag
        \\
        &\quad-\frac{1}2b(p,p,As)-\frac{1}2b(p,q,As)-\frac{1}2b(q,p,As)
        -\frac{1}2b(r,q,Aq)\label{est:tht2:sq:apriori}
        \\
        &\quad -\frac{1}2b(r,p,Aq)-\frac{1}2b(p,r,Aq)-\frac{1}2b(s,p,Aq)-\frac{1}2b(p,s,Aq)\notag
    \end{align}
 We now treat the trilinear terms. The first such term on the right-hand side of \eqref{est:tht2:sq:apriori} is estimated using H\"older's inequality, \eqref{est:interpolation}, \eqref{est:Bernstein}, and Young's inequality:
    \begin{align}\label{est:tht2:trilinear1}
        \frac{1}2|b(s,s,Ar)|&\leq \frac{C_A}2|A^{5/2}r|^{1/2}|A^{3/2}r|^{1/2}|s|^2\leq \frac{C_A}2K^3\|r\||s|^2\leq\frac{C_A}{2K}\|r\||As|^2
    \end{align}
For the next term, we apply \eqref{eq:B:enstrophy:miracle} and see that
    \begin{align}\notag
        -b(r,r,As)=b(s,r,Ar)+b(r,s,Ar).
    \end{align}
Then upon applying H\"older's inequality, \eqref{est:Bernstein}, \eqref{est:Sobolev}, and Young's inequality, we obtain
    \begin{align}
        \frac{1}2|b(s,r,Ar)|&\leq \frac{1}2C_S\ln(e+K)^{1/2}|s||Ar|^2\leq \frac{C_S\ln(e+K)^{1/2}}{2K}|As|\|r\||Ar|\notag
        \\
        &\leq \frac{2C_S^2\ln(e+K)}{K^2\nu}\|r\|^2|Ar|^2+\frac{\nu}{32}|As|^2\notag
        \\
        \frac{1}2|b(r,s,Ar)|&\leq \frac{1}2C_S\ln(e+K)^{1/2}\|r\|\|s\||Ar|\leq \frac{C_S\ln(e+K)^{1/2}}{2K}\|r\||Ar||As|\notag
        \\
        &\leq \frac{2C_S^2\ln(e+K)}{K^2\nu}\|r\|^2|Ar|^2+\frac{\nu}{32}|As|^2.\notag
    \end{align}
In particular   
    \begin{align}\label{est:tht2:trilinear2}
        \frac{1}2|b(r,r,As)|&\leq \frac{4C_S^2\ln(e+K)}{K^2\nu}\|r\|^2|Ar|^2+\frac{\nu}{16}|As|^2.
    \end{align}
For the next three terms, we similarly estimate
    \begin{align}\label{est:tht2:trilinear3}
        \begin{split}
        \frac{1}2|b(p,p,As)|&\leq \frac{1}2C_S\ln(e+K)^{1/2}\|p\|^2|As|\leq\frac{2C_S^2\ln(e+K)}{\nu}\|p\|^4+\frac{\nu}{32}|As|^2
        \\
        \frac{1}2|b(p,q,As)|&\leq \frac{1}2C_S\ln(e+K)^{1/2}\|p\|\|q\||As|\leq \frac{2C_S^2\ln(e+K)}{K^2\nu}\|p\|^2|Aq|^2+\frac{\nu}{32}|As|^2
        \\
        \frac{1}2|b(q,p,As)|&\leq\frac{1}2C_S\ln(e+K)^{1/2}|q||Ap||As|\leq\frac{1}2C_S\ln(e+K)^{1/2}\|q\|\|p\||As|
        \\
        &\leq \frac{2C_S^2\ln(e+K)}{\nu K^2}|Aq|^2\|p\|^2+\frac{\nu}{32}|As|^2.
        \end{split}
    \end{align}
For the last five terms, we also estimate
    \begin{align}\label{est:tht2:trilinear4}
        \begin{split}
        \frac{1}2|b(r,q,Aq)|&\leq \frac{1}2C_S\ln(e+K)^{1/2}\|r\|\|q\||Aq|\leq \frac{2C_S^2\ln(e+K)}{\nu K^2}\|r\|^2|Aq|^2+\frac{\nu}{32}|Aq|^2
        \\
        \frac{1}2|b(r,p,Aq)|&\leq \frac{1}2C_S\ln(e+K)^{1/2}\|r\|\|p\||Aq|\leq \frac{2C_S^2\ln(e+K)}{\nu}\|r\|^2\|p\|^2+\frac{\nu}{32}|Aq|^2
        \\
        \frac{1}2|b(p,r,Aq)|&\leq \frac{1}2C_S\ln(e+K)^{1/2}\|p\|\|r\||Aq|\leq\frac{2C_S\ln(e+K)}{\nu}\|p\|^2\|r\|^2+\frac{\nu}{32}|Aq|^2
        \\
        \frac{1}2|b(s,p,Aq)|&\leq\frac{1}2C_S\ln(e+K)^{1/2}|s||Ap||Aq|\leq\frac{1}2C_S\ln(e+K)^{1/2}\|s\|\|p\||Aq|
        \\
        &\leq \frac{2C_S^2\ln(e+K)}{\nu K^2}|As|^2\|p\|^2+\frac{\nu}{32}|Aq|^2
        \\
        \frac{1}2|b(p,s,Aq)|&\leq\frac{1}2C_S\ln(e+K)^{1/2}\|p\|\|s\||Aq|\leq\frac{2C_S^2\ln(e+K)}{\nu K^2}\|p\|^2|As|^2+\frac{\nu}{32}|Aq|^2.
        \end{split}
    \end{align}
Therefore, combining \eqref{est:tht2:trilinear1}, \eqref{est:tht2:trilinear2}, \eqref{est:tht2:trilinear3}, \eqref{est:tht2:trilinear4} with \cref{lem:estimate:force:H1} (applied with $\de=1/8$) to treat the first two terms on the right-hand side of \eqref{est:tht2:sq:apriori}, we obtain
    \begin{align}\label{est:tht2:sq:final1}
        &\frac{1}2\frac{d}{dt}\left(\|s\|^2+\|q\|^2\right)+\frac{3}2\nu\left(|As|^2+|Aq|^2\right)\leq \frac{4}{\nu}\left(|k|^2+|h|^2\right)+\frac{2C_S^2\ln(e+K)}{\nu}\|p\|^4\notag
        \\
        &\ \ +\frac{C_A}{2K}\|r\||As|^2+\frac{4C_S^2\ln(e+K)}{\nu K^2}\|p\|^2|Aq|^2+\frac{2C_S^2\ln(e+K)}{\nu K^2}\|r\|^2|Aq|^2\notag
        \\
        &\ \ +\frac{4C_S^2\ln(e+K)}{\nu K^2}|As|^2\|p\|^2+\frac{2C_S^2\ln(e+K)}{K^2\nu}\|r\|^2|Ar|^2+\frac{4C_S^2\ln(e+K)}{\nu}\|r\|^2\|p\|^2.
    \end{align}
We recall \eqref{est:heat:p} and \eqref{est:tht2:r:apriori:final2} to deduce that for $t\leq t_\Cm$, we have
    \begin{align}
        &\frac{d}{dt}\left(\|s\|^2+\|q\|^2\right)+3\nu\left(|As|^2+|Aq|^2\right)\leq \frac{8}{\nu}\left(|k|^2+|h|^2\right)+{8C_S^2\ln(e+K)}\left(\Cp^2+\Cr^2\right)\Cp^2\nu^3\notag
        \\
        &\quad+\nu\left(\frac{C_A\Cr}{K}+\frac{8C_S^2\ln(e+K)\Cp^2}{K^2}\right)|As|^2+\nu\frac{8C_S^2\ln(e+K)\left(\Cp^2+\Cr^2\right)}{K^2}|Aq|^2\notag
        \\
        &\quad+\nu\frac{4C_S^2\ln(e+K)\Cr^2}{K^2}|Ar|^2.\notag
    \end{align}
Adding \eqref{est:heat:basic}, \eqref{est:r:tht2}, we finally arrive at
    \begin{align}\label{est:apriori:tht2:final}
        &\frac{d}{dt}\left(\|z\|^2+\|w\|^2\right)+3\nu\left(|Az|^2+|Aw|^2\right)\notag
        \\
        &\leq \frac{16}{\nu}\left(|k|^2+|h|^2\right)+8C_S^2\ln(e+K)\left(\Cp^2+\Cr^2\right)\Cp^2\nu^3\notag
        \\
        &\quad+\nu\left(\frac{C_A\Cr}{K}+\frac{16C_S^2\ln(e+K)\left(\Cp^2+\Cr^2\right)}{K^2}\right)\left(|Az|^2+|Aw|^2\right).
    \end{align}
Recall that we have chosen $K$ so that the second condition in \eqref{cond:K:tht2} holds, i.e.,
    \begin{align}\notag
        \frac{C_A\Cr}{K}+\frac{16C_S^2\ln(e+K)\left(\Cp^2+\Cr^2\right)}{K^2}\leq 2.
    \end{align}
Also observe that
    \begin{align}\notag
       \Cf^2\leq \Cp^2+\Cr^2\leq \Cd^2+\Cf^2
    \end{align}
Recall \eqref{def:Cd:Cf}. We then deduce that
    \begin{align}
        \frac{\|z(t)\|^2+\|w(t)\|^2}{\nu^2}
        &\leq e^{-\nu t}\Cd^2+8\left\{2+C_S^2\ln(e+K)(1+\Cp^2)\right\}\left(\Cp^2+\Cr^2\right)(1-e^{-\nu t})\notag
        \\
        &\leq e^{-\nu t}\Cd^2+8\left\{2+C_S^2\ln(e+K)(1+\Cp^2)\right\}(\Cd^2+\Cf^2)(1-e^{-\nu t}),\notag
    \end{align}
for all $t\leq t_\Cm$. Recall that we have further chosen $K$ so that the first condition in \eqref{cond:K:tht2} holds, i.e., $8C_S^2\ln(e+K)\geq1$. Finally, we invoke our choice of $\Cm$ given by \eqref{def:M:tht2}. For this choice of $\Cm$, we see that
    \begin{align}\notag
        \sup_{t\leq t_\Cm}\frac{\|z(t)\|^2+\|w(t)\|^2}{\nu^2}\leq\frac{1}2\Cm^2<\Cm^2.
    \end{align}
This contradicts the definition of $t_\Cm$ if $t_\Cm<\infty$. We therefore conclude that $t_\Cm=\infty$. Moreover, upon returning to \eqref{est:apriori:tht2:final}, we may further deduce that
    \begin{align}\notag
        \nu\int_0^t\left(|Az(s)|^2+|Aw(s)|^2\right)ds\leq \|z_0\|^2+\|w_0\|^2+8\left\{2+C_S^2\ln(e+K)(1+\Cp^2)\right\}(\Cd^2+\Cf^2)t,
    \end{align}
for all $t\geq0$, as desired.
\end{proof}

We now established improved estimates under the assumption that $h$ decays to zero asymptotically.

\begin{Lem}\label{lem:apriori:tht2:refined}
Suppose that $|h(t)|\goesto0$ as $t\goesto\infty$. Under the assumptions of \cref{lem:apriori:tht2}, if $K>0$ further satisfies
    \begin{align}\label{cond:K:tht2:refined}    
    K^2\geq 20(C_A\vee C_S)^2\ln(e+K)\Ck^2
    \end{align}
then
    \begin{align}\notag
        \limsup_{t\goesto\infty}\|v(t)\|\leq 6\Ck\nu.
    \end{align}
\end{Lem}

\begin{proof}
First recall that since $t_\Cm=\infty$, for $\Cm$ given by \eqref{def:M:tht2} and $\tht_2$ satisfying \eqref{cond:tht2}, we deduce from \eqref{est:r:tht2} that
    \begin{align}\label{est:tht2:r:final:refined}
        \frac{d}{dt}\|r\|^2+\nu|Ar|^2&\leq 16\Ck^2\nu^3.
    \end{align}
In particular, we have
    \begin{align}\notag
        \|r(t)\|^2\leq e^{-\nu t}\|r_0\|^2+9\Ck^2\nu^2(1-e^{-\nu t})
    \end{align}
Thus, for $t_0$ sufficiently large, we have
    \begin{align}\label{est:tht2:r:refined}
        \sup_{t\geq t_0}\|r(t)\|\leq 3\Ck\nu.
    \end{align}
Moreover, since $|h(t)|\goesto0$ as $t\goesto\infty$, it follows from \cref{cor:apriori:heat} we may further assume that $t_0$ is sufficiently large such that
    \begin{align}\notag
            \sup_{t\geq t_0}|h(t)|\leq\Ck\nu^2,\quad \sup_{t\geq t_0}\|p(t)\|\leq \min\left\{\frac{\sqrt{2}}{2C_S\ln(e+K)^{1/2}},\frac{1}{8C_S\ln(e+K)^{1/2}},\Ck\right\}\nu
    \end{align}
We consider \eqref{est:tht2:sq:final1} over the interval $t\geq t_0$ and apply \eqref{est:tht2:r:refined} to obtain
     \begin{align}
        &\frac{d}{dt}\left(\|s\|^2+\|q\|^2\right)+3\nu\left(|As|^2+|Aq|^2\right)\leq 32\Ck^2\nu^3\notag
        \\
        &\quad+\left(\frac{3C_A\Ck}{K}+\frac{8C_S^2\ln(e+K)\Ck^2}{ K^2}\right)\nu|As|^2+\frac{20C_S^2\ln(e+K)\Ck^2}{K^2}\nu|Aq|^2+\frac{12C_S^2\ln(e+K)\Ck^2}{K^2}\nu|Ar|^2.\notag
    \end{align}
Upon adding \eqref{est:tht2:r:refined}, \eqref{est:heat:basic} and making use of \eqref{cond:K:tht2:refined}, we deduce
    \begin{align}
        &\frac{d}{dt}\left(\|z\|^2+\|w\|^2\right)+\frac{\nu}2\left(|Az|^2+|Aq|^2\right)\leq 50\Ck^2\nu^3\notag.
    \end{align}
An application of Gr\"onwall's inequality yields
    \begin{align}\notag
        \|z(t)\|^2+\|w(t)\|^2\leq e^{-(\nu/2)(t-t_0)}\left(\|z(t_0)\|^2+\|w(t_0)\|^2\right)+100\Ck^2\nu.
    \end{align}
We choose $t_0'\geq t_0$ sufficiently large such that
    \begin{align}
        \sup_{t\geq t_0'}\left(\|z(t)\|^2+\|w(t)\|^2\right)\leq 144\Ck^2\nu^2\notag.\notag
    \end{align}
The claim then follows.
\end{proof}

Observe that \cref{thm:uniform:tht1tht2} in the regime $|\tht_2|\ll1$ follows from \cref{lem:apriori:tht2:refined} and \eqref{est:Ck}.

\subsubsection{Global-Time Apriori estimates: $|\tht_1-\tht_2|\ll1$}\label{sect:apriori:tht1tht2}

We prove the following.

\begin{Lem}\label{lem:apriori:tht1tht2}
Let $v_1^0,v_2^0\in V$, $g_1,g_2\in L^\infty(0,\infty;H)$ and $\tht_1+\tht_2=1$. Given $K>0$, let $\Cm$ be defined by
    \begin{align}\label{def:M:tht1tht2}
        \Cm^2:=2e\left(\frac{\|z_0\|^2+\|w_0\|^2}{\nu^2}+\Ck^2+\Ch^2+ {C_S^2\ln(e+K)}\Cp^4\right)
    \end{align}
Suppose that $K$ satisfies
    \begin{align}\label{cond:K:tht1tht2}
        K^2\geq 1024 C_S^2\ln(e+K)\left(\Cp^2+\Ch^2\right).
    \end{align}
Further suppose that $\tht_1,\tht_2$ satisfies
    \begin{align}\label{cond:tht1tht2}
        |\tht_1-\tht_2|\leq\frac{\sqrt{2}}{8C_S\ln(e+K)^{1/2}\Cm}.
    \end{align}
Then
    \begin{align}\notag
        \sup_{t\geq0}\left(\|z(t)\|^2+\|w(t)\|^2\right)\leq \Cm^2\nu^2.
    \end{align}
\end{Lem}

\begin{proof}

Upon taking the $H$-inner product of \eqref{eq:dr:zw}, with $Az, Aw$, respectively, we obtain
    \begin{align}
        &\frac{1}2\frac{d}{dt}\left(\|z\|^2+\|w\|^2\right)+\nu\left(|Az|^2+|Aw|^2\right)+\frac{1}2b(w,w,Az)+\frac{1}2b(z,w,Aq)+\frac{1}2b(w,z,Aq)\notag
        \\
        &=\lp k,Az\rp+\lp h,Aw\rp+\frac{\tht_1-\tht_2}2\left(b(z,z,Ar)+b(w,w,Ar)\right)\notag
    \end{align}
where we have applied \eqref{eq:B:enstrophy}. Observe that by \eqref{eq:B:enstrophy:miracle}, we have
    \begin{align}\notag
        b(q,q,Az)+b(z,q,Aq)+b(q,z,Aq)=0.
    \end{align}
Thus
    \begin{align}\label{eq:enstrophy:balance:tht12:apriori}
        &\frac{1}2\frac{d}{dt}\left(\|z\|^2+\|w\|^2\right)+\nu\left(|Az|^2+|Aw|^2\right)\notag
        \\
        &\quad+\frac{1}2b(p,p,Az)+\frac{1}2b(p,q,Az)+\frac{1}2b(q,p,Az)+\frac{1}2b(z,p,Aq)+\frac{1}2b(p,z,Aq)\notag
        \\
        &=\lp k,Az\rp+\lp h,Aw\rp+\frac{\tht_1-\tht_2}2b(z,z,Ar)+\frac{\tht_1-\tht_2}2b(w,w,Ar).
    \end{align}
The first two terms on the right-hand side of \eqref{eq:enstrophy:balance:tht12:apriori} are estimated using \cref{lem:estimate:force:H1} with $\de=1/8$. We now estimate the trilinear terms. For the first two trilinear terms on the left-hand side of \eqref{eq:enstrophy:balance:tht12:apriori}, we employ H\"older's inequality, \eqref{est:Sobolev}, \eqref{est:Bernstein}, and Young's inequality to obtain
    \begin{align}
        \frac{1}2|b(p,p,Az)|&\leq \frac{1}2C_S\ln(e+K)^{1/2}\|p\|^2|Az|\leq \frac{2C_S^2\ln(e+K)}{\nu}\|p\|^4+\frac{\nu}{32}|Az|^2\notag
        \\
        \frac{1}2|b(p,q,Az)|&\leq \frac{1}2C_S\ln(e+K)^{1/2}\|p\|\|q\||Az|\leq \frac{2C_S^2\ln(e+K)}{\nu K^2}\|p\|^2|Aq|^2+\frac{\nu}{32}|Az|^2\notag
        \\
        \frac{1}2|b(q,p,Az)|&\leq \frac{1}2C_S\ln(e+K)^{1/2}|q||Ap||Az|\leq\frac{1}2C_S\ln(e+K)^{1/2}\|q\|\|p\||Az|\notag
        \\
        &\leq \frac{2C_S^2\ln(e+K)}{\nu K^2}|Aq|^2\|p\|^2+\frac{\nu}{32}|Az|^2.\notag
    \end{align}
In particular, we have
    \begin{align}\label{est:tht12:apriori:summary123}
             &\frac{1}2|b(p,p,Az)|+ \frac{1}2|b(p,q,Az)|+\frac{1}2|b(q,p,Az)|\notag
             \\
             &\leq\frac{2C_S^2\ln(e+K)}{\nu}\|p\|^4+\frac{4C_S^2\ln(e+K)}{\nu K^2}\|p\|^2|Aw|^2+\frac{3\nu}{32}|Az|^2.
    \end{align}

For the fourth trilinear term, we first observe that
    \begin{align}\notag
    b(z,p,Aq)&=b(r,p,Aq)+b(s,p,Aq)\notag
    \\
    &=b(\nabla r,p,\nabla q)+b(r,\nabla p,\nabla q)+b(s,p,Aq)\notag.
    \end{align}
Then by H\"older's inequality, \eqref{est:Sobolev}, \eqref{est:Bernstein}, and Young's inequality, we obtain
    \begin{align}
        \frac{1}2|b(\nabla r,p,\nabla q)|&\leq \frac{C_S}2\ln(e+K)^{1/2}|Ar|\|p\|\|q\|\leq \frac{C_S\ln(e+K)^{1/2}}{2K}|Ar|\|p\||Aq|\notag 
        \\
        &\leq \frac{2C_S^2\ln(e+K)}{\nu K^2}\|p\||Ar|^2+\frac{\nu}{32}|Aq|^2\notag
        \\
        \frac{1}2|b(r,\nabla p,\nabla q)|&\leq \frac{C_S\ln(e+K)^{1/2}}{2}\|r\||Ap|\|q\|\leq \frac{C_S\ln(e+K)^{1/2}}{2K}\|r\||Ap||Aq|\notag
        \\
        &\leq \frac{2C_S^2\ln(e+K)}{\nu K^2}|Ap|^2\|r\|^2+\frac{\nu}{32}|Aq|^2\notag
        \\
        \frac{1}2|b(s,p,Aq)|&\leq \frac{1}2C_S\ln(e+K)|s||Ap||Aq|\leq \frac{1}2C_S\ln(e+K)\|s\|\|p\||Aq|\notag
        \\
        &\leq\frac{2C_S^2\ln(e+K)}{\nu K^2}|As|^2\|p\|^2+\frac{\nu}{32}|Aq|^2.\notag
    \end{align}
In particular, we have
    \begin{align}\label{est:tht12:apriori:summary4}
        \frac{1}2|b(z,p,Aq)|&\leq \frac{2C_S^2\ln(e+K)}{\nu K^2}\|p\|^2|Az|^2+\frac{2C_S^2\ln(e+K)}{\nu K^2}|Ap|^2\|z\|^2+\frac{3\nu}{32}|Aw|^2.
    \end{align}

For the fifth trilinear term, we first see that
    \begin{align}
        b(p,z,Aq)&=b(p,r,Aq)+b(p,s,Aq)\notag\\
        &=b(\nabla p,r, \nabla q)+b(p,\nabla r, \nabla q)+b(p,s,Aq)\notag.
    \end{align}
Similar to the first four terms, we similarly estimate these terms as
    \begin{align}
        \frac{1}2|b(\nabla p,r,\nabla q)|&\leq \frac{C_S\ln(e+K)^{1/2}}{2}\|p\||Ar|\|q\|\leq \frac{C_S\ln(e+K)^{1/2}}{2K}\|p\||Ar|\|q\|\notag
        \\
        &\leq \frac{2C_S^2\ln(e+K)}{\nu K^2}\|p\|^2|Ar|^2+\frac{\nu}{32}|Aq|^2\notag
        \\
        \frac{1}2|b(p,\nabla r,\nabla q)|&\leq \frac{C_S\ln(e+K)^{1/2}}{2}\|p\||Ar|\|q\|\leq \frac{C_S\ln(e+K)^{1/2}}{2K}\|p\||Ar||Aq|\notag
        \\
        &\leq \frac{2C_S^2\ln(e+K)}{\nu K^2}\|p\|^2|Ar|^2+\frac{\nu}{32}|Aq|^2\notag
        \\
        \frac{1}2|b(p,s,Aq)|&\leq \frac{C_S\ln(e+K)^{1/2}}{2}\|p\|\|s\||Aq|\leq \frac{C_S\ln(e+K)^{1/2}}{2K}\|p\||As||Aq|\notag
        \\
        &\leq \frac{2C_S^2\ln(e+K)}{\nu K^2}\|p\|^2|As|^2+\frac{\nu}{32}|Aq|^2\notag.
    \end{align}
In particular, we have
    \begin{align}\label{est:tht12:apriori:summary5}
        \frac{1}2|b(p,z,Aq)|&\leq \frac{6C_S^2\ln(e+K)}{\nu K^2}\|p\|^2|Az|^2+ \frac{3\nu}{32}|Aw|^2
    \end{align}
Thus, by \eqref{est:tht12:apriori:summary123}, \eqref{est:tht12:apriori:summary4}, \eqref{est:tht12:apriori:summary5}, we see that the first five trilinear terms in \eqref{eq:enstrophy:balance:tht12:apriori} are collectively bounded above by
    \begin{align}\label{est:tht12:apriori:summary12345}
        &\frac{2C_S^2\ln(e+K)}{\nu}\|p\|^4+\frac{2C_S^2\ln(e+K)}{\nu K^2}|Ap|^2\|z\|^2+\frac{12C_S^2\ln(e+K)}{\nu K^2}\|p\|^2|Aw|^2+\frac{3\nu}{16}|Aw|^2+\frac{3\nu}{32}|Az|^2.
    \end{align}
    
Now for the trilinear terms on the right-hand side of \eqref{eq:enstrophy:balance:tht12:apriori}, we see that the first of these can be rewritten using \eqref{eq:B:enstrophy:miracle} as
    \begin{align}
        b(z,z,Ar)&=-b(r,z,Az)-b(z,r,Az)=-b(r,z,Az)-b(r,r,Az)-b(s,r,Az)\notag.
    \end{align}
Then by H\"older's inequality, \eqref{est:Sobolev}, and Young's inequality, we have
    \begin{align}
        \frac{|\tht_1-\tht_2|}2|b(r,z,Az)|&\leq \frac{|\tht_1-\tht_2|}C_S\ln(e+K)^{1/2}\|r\|\|z\||Az|\notag
        \\
        &\leq \frac{2|\tht_1-\tht_2|^2C_S^2\ln(e+K)}{\nu}\|r\|^2\|z\|^2+\frac{\nu}{32}|Az|^2\notag
        \\
        \frac{|\tht_1-\tht_2|}2|b(r,r,Az)|&\leq \frac{2|\tht_1-\tht_2|^2C_S^2\ln(e+K)}{\nu}\|r\|^2\|r\|^2+\frac{\nu}{32}|Az|^2\notag
        \\
         \frac{|\tht_1-\tht_2|}2|b(s,r,Az)|&\leq\frac{|\tht_1-\tht_2|C_S\ln(e+K)^{1/2}}2|s||Ar||Az|\notag
         \\
         &\leq \frac{|\tht_1-\tht_2|C_S\ln(e+K)^{1/2}}2\|s\|\|r\||Az|\notag
         \\
         &\leq \frac{2|\tht_1-\tht_2|^2C_S^2\ln(e+K)}{\nu}\|r\|^2\|s\|^2+\frac{\nu}{32}|Az|^2\notag
    \end{align}
In particular
    \begin{align}\label{est:tht12:apriori:summary6}
        \frac{|\tht_1-\tht_2|}2|b(z,z,Ar)|\leq\frac{4|\tht_1-\tht_2|^2C_S^2\ln(e+K)}{\nu}\|r\|^2\|z\|^2+\frac{3\nu}{32}|Az|^2.
    \end{align}
For the second trilinear term on the right-hand side of \eqref{eq:enstrophy:balance:tht12:apriori}, we observe that
    \begin{align}
        b(w,w,Ar)=-b(r,w,Aw)-b(w,r,Aw)=-b(r,w,Aw)-b(p,r,Aw)-b(q,r,Aw).\notag
    \end{align}
Then
    \begin{align}
         \frac{|\tht_1-\tht_2|}2|b(r,w,Aw)|&\leq  \frac{|\tht_1-\tht_2|}2C_S\ln(e+K)^{1/2}\|r\|\|w\||Aw|\notag
         \\
         &\leq  \frac{2|\tht_1-\tht_2|^2C_S^2\ln(e+K)}{\nu}\|r\|^2\|w\|^2+\frac{\nu}{32}|Aw|^2\notag
         \\
        \frac{|\tht_1-\tht_2|}2|b(p,r,Aw)|&\leq \frac{|\tht_1-\tht_2|}2\|p\|\|r\||Aw|\notag
        \\
        &\leq  \frac{2|\tht_1-\tht_2|^2C_S\ln(e+K)}{\nu}\|p\|^2\|r\|^2+\frac{\nu}{32}|Aw|^2\notag
        \\
        \frac{|\tht_1-\tht_2|}2|b(q,r,Aw)|&\leq  \frac{|\tht_1-\tht_2|}2C_S\ln(e+K)^{1/2}|q||Ar||Aw|\notag
        \\
        &\leq \frac{|\tht_1-\tht_2|}2C_S\ln(e+K)^{1/2}\|q\|\|r\||Aw|\notag
        \\
        &\leq  \frac{2|\tht_1-\tht_2|^2C_S^2\ln(e+K)}{\nu}\|q\|^2\|r\|^2+\frac{\nu}{32}|Aw|^2.\notag
    \end{align}
In particular
    \begin{align}\label{est:tht12:apriori:summary7}
         \frac{|\tht_1-\tht_2|}2|b(w,w,Ar)|&\leq \frac{4|\tht_1-\tht_2|^2C_S^2\ln(e+K)}{\nu}\|r\|^2\|w\|^2+\frac{3\nu}{32}|Aw|^2.
    \end{align}
    
Combining \eqref{est:tht12:apriori:summary12345}, \eqref{est:tht12:apriori:summary6}, and \eqref{est:tht12:apriori:summary7}, we arrive at
    \begin{align}\label{est:summary:tht12:apriori}
        &\frac{d}{dt}\left(\|z\|^2+\|w\|^2\right)+\frac{5}4\nu\left(|Az|^2+|Aw|^2\right)\leq \frac{8}{\nu}\left(|k|^2+|h|^2\right)\notag
        \\
        &\quad+ \frac{4C_S^2\ln(e+K)}{\nu}\|p\|^4+\frac{4C_S^2\ln(e+K)}{\nu K^2}|Ap|^2\|z\|^2+\frac{24C_S^2\ln(e+K)}{\nu K^2}\|p\|^2|Aw|^2\notag
        \\
        &\quad+\frac{8|\tht_1-\tht_2|^2C_S^2\ln(e+K)}{\nu}\|r\|^2\left(\|z\|^2+\|w\|^2\right).
    \end{align}
Now, given $\Cm>0$, we once again consider the stopping time $t_\Cm$ from \cref{cor:stop:time}. Then upon recalling \eqref{est:heat:p}, we see that
    \begin{align}
       &\frac{d}{dt}\left(\|z\|^2+\|w\|^2\right)+\frac{5}4\nu\left(|Az|^2+|Aw|^2\right)\leq 8\nu^3\left(\Ck^2+\Ch^2\right)\notag
        \\
        &\quad+ {4C_S^2\ln(e+K)}\Cp^4\nu^3+\frac{24C_S^2\ln(e+K)}{K^2}\Cp^2\nu|Aw|^2\notag
        \\
        &\quad+4\left(\frac{C_S^2\ln(e+K)}{ K^2}\left(\frac{|Ap|}{\nu}\right)^2+{2|\tht_1-\tht_2|^2C_S^2\ln(e+K)}\Cm^2\right)\nu\left(\|z\|^2+\|w\|^2\right)\notag.
    \end{align}
holds for all $t\leq t_\Cm$. Recall that we choose $K$ (independently of $\Cm$) such that \eqref{cond:K:tht1tht2} holds. Moreover, by \eqref{cond:tht1tht2}, we see that $\tht_1,\tht_2$ satisfy
    \begin{align}\notag
        {2|\tht_1-\tht_2|^2C_S^2\ln(e+K)}\Cm^2\leq\frac{1}{16}
    \end{align}
Then
    \begin{align}
        &\frac{d}{dt}\left(\|z\|^2+\|w\|^2\right)+\nu\left(|Az|^2+|Aw\|^2\right)+\frac{1}8\nu\left(\|z\|^2+\|w\|^2\right)\notag
        \\
        &\leq 4\left[2\left(\Ck^2+\Ch^2\right)+ {C_S^2\ln(e+K)}\Cp^4\right]\nu^3
        +\frac{4C_S^2\ln(e+K)}{ K^2}\left(\frac{|Ap|}{\nu}\right)^2\nu\left(\|z\|^2+\|w\|^2\right)\notag
    \end{align}
Recall that $p$ satisfies \eqref{est:heat:Ap}. Therefore, an application of Gr\"onwall's inequality yields
    \begin{align}
    &\|z(t)\|^2+\|w(t)\|^2\notag
    \\
    &\leq \exp\left(-\frac{9}8\nu t+\frac{4C_S^2\ln(e+K)}{ \nu K^2}\left(\Cp^2+\Ch^2\nu t\right)\right)\left(\|z_0\|^2+\|w_0\|^2\right)\notag
    \\
    &\quad+4\left[2\left(\Ck^2+\Ch^2\right)+ {C_S^2\ln(e+K)}\Cp^4\right]\nu^3\notag
    \\
    &\qquad\times\left(\int_0^t\exp\left(-\frac{9}8\nu (t-s)+\frac{4C_S^2\ln(e+K)}{ \nu K^2}\left(\Cp^2+\Ch^2\nu(t-s)\right)\right)ds\right)\notag
    \\
    &\leq e^{-\nu t+1}\left(\|z_0\|^2+\|w_0\|^2\right)+e\left[2\left(\Ck^2+\Ch^2\right)+ {C_S^2\ln(e+K)}\Cp^4\right]\nu^2(1-e^{-\nu t}),\notag
    \end{align}
where we once again invoked our choice of $K$ in obtaining the final inequality. Finally, we choose $\Cm$ given by \eqref{def:M:tht1tht2} to therefore deduce
    \begin{align}\notag
        \|z(t)\|^2+\|w(t)\|^2\leq \frac{1}2\Cm^2<\Cm^2.
    \end{align}
We conclude that $t_\Cm$ cannot be finite for this choice of $\Cm$, as desired.

\end{proof}

We will now further refine these bounds under the assumption that the low-modes of the intertwinement are synchronous.

\begin{Lem}\label{lem:apriori:tht1tht2:refined}
Under the assumptions of \cref{lem:apriori:tht1tht2}, if $|h(t)|\goesto0$ as $t\goesto\infty$, then
    \begin{align}\notag
    \limsup_{t\goesto\infty}\|v(t)\|\leq8\nu\Ck.
    \end{align}
\end{Lem}       

\begin{proof}
Since $|h(t)|\goesto0$ as $t\goesto\infty$, \cref{cor:apriori:heat} implies that for all $m\geq0$ and $\eps>0$, there exists $t_m$ such that
    \begin{align}\label{est:p:small}
            \sup_{t\geq t_m}\sup_{0\leq \ell\leq m}\|p(t)\|_\ell\leq \eps\nu
    \end{align}
We may moreover choose $t_m$ large enough, so that
    \begin{align}\label{est:h:small}
        \sup_{t\geq t_2}|h(t)|\leq\Ck\nu^2.
    \end{align}
Now us then recall \eqref{est:summary:tht12:apriori}:
        \begin{align}
        &\frac{d}{dt}\left(\|z\|^2+\|w\|^2\right)+\frac{5}4\nu\left(|Az|^2+|Aw|^2\right)\leq \frac{8}{\nu}\left(|k|^2+|h|^2\right)\notag
        \\
        &\quad+ \frac{4C_S^2\ln(e+K)}{\nu}\|p\|^4+\frac{4C_S^2\ln(e+K)}{\nu K^2}|Ap|^2\|z\|^2+\frac{24C_S^2\ln(e+K)}{\nu K^2}\|p\|^2|Aw|^2\notag
        \\
        &\quad+\frac{8|\tht_1-\tht_2|^2C_S^2\ln(e+K)}{\nu}\|r\|^2\left(\|z\|^2+\|w\|^2\right)\notag.
    \end{align}
For $m=2$ and $t\geq t_2$, we deduce
    \begin{align}
     &\frac{d}{dt}\left(\|z\|^2+\|w\|^2\right)+\frac{5}4\nu\left(|Az|^2+|Aw|^2\right)\leq \frac{16}{\nu}\Ck^2\notag
        \\
        &\quad+ {4C_S^2\ln(e+K)}\eps^4\nu^3+\frac{4C_S^2\ln(e+K)}{\nu  K^2}\eps^2\nu\|z\|^2+\frac{24C_S^2\ln(e+K)}{ K^2}\eps^2\nu|Aw|^2\notag
        \\
        &\quad+8|\tht_1-\tht_2|^2C_S^2\ln(e+K)\Cm^2\nu\left(\|z\|^2+\|w\|^2\right)\notag
    \end{align}
We choose $\eps>0$ such that
    \begin{align}\notag
         {4C_S^2\ln(e+K)}\eps^4\leq \Ck^2,\quad \frac{28C_S^2\ln(e+K)}{K^2}\eps^2\leq\frac{1}{8}
    \end{align}
Since $\tht_1,\tht_2$ is still assumed to satisfy \eqref{cond:tht1tht2}, we arrive at
    \begin{align}
         &\frac{d}{dt}\left(\|z\|^2+\|w\|^2\right)+\nu\left(|Az|^2+|Aw|^2\right)\leq 32\nu^3\Ck^2.\notag
    \end{align}
It follows that
    \begin{align}\notag
        \|z(t)\|^2+\|w(t)\|^2\leq e^{-\nu (t-t_2)}\left(\|z(t_2)\|^2+\|w(t_2)\|^2\right)+32\nu^2\Ck^2,
    \end{align}
and
    \begin{align}\notag
        \nu\int_{t_2}^2\left(|Az(s)|^2+|Aw(s)|^2ds\right)\leq \left(\|z(t_2)\|^2+\|w(t_2)\|^2\right)+32\nu^2\Ck^2
    \end{align}
hold for all $t\geq t_2$. Thus, for $t_2'\geq t_2$ sufficiently large, we deduce that
    \begin{align}\notag
         \sup_{t\geq t_2'}\left(\|z(t)\|^2+\|w(t)\|^2\right)\leq 64\nu^2\Ck^2,
    \end{align}
as desired.
\end{proof}

Observe that \cref{thm:uniform:tht1tht2} in the regime $|\tht_1-\tht_2|\ll1$ follows from \cref{lem:apriori:tht1tht2:refined} and \eqref{est:Ck}.

\subsection*{Ethics Declarations}

\subsubsection*{Competing interests} The authors declare no competing interests.

\subsubsection*{Data Availability} The authors declare there is no data.

\subsubsection*{Acknowledgments} The authors would like to thank Francesco Fosella for insightful discussions in the course of this work. E.C. was supported in part by the Department of Defense Vannevar Bush Faculty Fellowship, under
ONR award N00014-22-1-2790.  A.F. was supported in part by the National Science Foundation through DMS 2206493. V.R.M. was in part supported by the National Science Foundation through DMS 2213363, DMS 2206491, and NSF-DMS 2511403, as well as the Dolciani Halloran Foundation.

\newcommand{\etalchar}[1]{$^{#1}$}
\providecommand{\bysame}{\leavevmode\hbox to3em{\hrulefill}\thinspace}
\providecommand{\MR}{\relax\ifhmode\unskip\space\fi MR }
\providecommand{\MRhref}[2]{%
  \href{http://www.ams.org/mathscinet-getitem?mr=#1}{#2}
}
\providecommand{\href}[2]{#2}

\vfill 

\noindent Elizabeth Carlson\\
{\footnotesize
Department of Computing \& Mathematical Sciences\\
California Institute of Technology\\
Department of Mathematics\\
Oregon State University\footnote{Present address}\\
Web: \url{https://sites.google.com/view/elizabethcarlsonmath}\\
Email: \url{carleliz@oregonstate.edu}\\
}

\noindent Aseel Farhat\\
{\footnotesize
Department of Mathematics \\
University of Virginia \\
Email: \url{af7py@virginia.edu}\\
}

\noindent Vincent R. Martinez\\
{\footnotesize
Department of Mathematics \& Statistics\\
CUNY Hunter College \\
Department of Mathematics \\
CUNY Graduate Center \\
Department of Computing \& Mathematical Sciences\\
California Institute of Technology\footnote{Present address} \\
Web: \url{http://math.hunter.cuny.edu/vmartine/}\\
Email: \url{vrmartinez@hunter.cuny.edu}, \url{vrm@caltech.edu}\\
}

\noindent Collin Victor\\
{\footnotesize
Department of Mathematics \\
Texas A\&M University \\
Web: \url{https://collinvictor.me/}\\
Email: \url{collin.victor@tamu.edu}\\
}

\end{document}